\documentclass[11pt]{article}
\usepackage[T1]{fontenc}
\usepackage{amsfonts}
\usepackage{amsmath}
\usepackage{amssymb}
\usepackage{amsthm}
\usepackage{bbm}
\usepackage{bm}
\usepackage{mathrsfs}
\usepackage{verbatim}
\usepackage{setspace}
\usepackage{color}
\usepackage{pdfsync}
\usepackage{nicefrac} 
\usepackage{enumitem}
\theoremstyle{plain}
\newtheorem{theorem}{Theorem}[section]
\newtheorem{proposition}[theorem]{Proposition}
\newtheorem{lemma}[theorem]{Lemma}
\newtheorem{corollary}[theorem]{Corollary}

\theoremstyle{definition}
\newtheorem{definition}[theorem]{Definition}
\newtheorem{remark}[theorem]{Remark}

\newtheorem{example}[theorem]{Example}

\newtheorem{assumption}[theorem]{Assumption}

\newtheorem*{conditionB}{Condition~(B)}
\newtheorem*{conditionJ}{Condition~(J)}
\theoremstyle{remark}

\renewenvironment{thebibliography}[1]{%
\begin{oldthebibliography}{#1}%
\setlength{\baselineskip}{.9em}
\linespread{1}
\small
\setlength{\parskip}{0.3ex}%
\setlength{\itemsep}{.5em}%
}%
{%
\end{oldthebibliography}%
}

\newcommand{\D}{\mathbb{D}}
\newcommand{\E}{\mathbb{E}}
\newcommand{\F}{\mathbb{F}}
\newcommand{\G}{\mathbb{G}}

\newcommand{\N}{\mathbb{N}}
\renewcommand{\P}{\mathbb{P}}
\newcommand{\Q}{\mathbb{Q}}
\newcommand{\R}{\mathbb{R}}
\renewcommand{\S}{\mathbb{S}}

\newcommand{\X}{\mathbb{X}}

\newcommand{\cB}{\mathcal{B}}
\newcommand{\cC}{\mathcal{C}}

\newcommand{\cF}{\mathcal{F}}

\newcommand{\cK}{\mathcal{K}}
\newcommand{\cL}{\mathcal{L}}
\newcommand{\cM}{\mathcal{M}}

\newcommand{\cP}{\mathcal{P}}

\newcommand{\cV}{\mathcal{V}}

\newcommand{\fD}{\mathfrak{D}}
\newcommand{\fJ}{\mathfrak{J}}
\newcommand{\fL}{\mathfrak{L}}

\newcommand{\fP}{\mathfrak{P}}

\newcommand{\fM}{\mathfrak{M}}

\DeclareMathOperator{\proj}{proj}

\DeclareMathOperator{\supp}{supp}

\DeclareMathOperator{\Var}{Var}

\newcommand{\ov}{\overline}
\numberwithin{equation}{section}

\usepackage[pdfborder={0 0 0}]{hyperref}
\hypersetup{
  urlcolor = black,
  pdfauthor = {Chong Liu, Ariel Neufeld},
  pdfkeywords = {T.B.D.;
},
  pdftitle = {Compactness Criterion for Semimartingale Laws and 
  Semimartingale Optimal Transport},
  pdfsubject = {Compactness Criterion for Semimartingale Laws and 
  Semimartingale Optimal Transport},
  pdfpagemode = UseNone
}

\begin{document}

\title{\vspace{-0em}
Compactness Criterion for Semimartingale Laws and 
Semimartingale Optimal Transport
\date{\today}
\author{
  Chong Liu%
  \thanks{
  Department of Mathematics, ETH Zurich, \texttt{chong.liu@math.ethz.ch}.
  }
  \and
  Ariel Neufeld%
   \thanks{
   Department of Mathematics, ETH Zurich, \texttt{ariel.neufeld@math.ethz.ch}.
   Financial support by the Swiss National Foundation grant SNF 200021$\_$153555 is gratefully acknowledged.
   }
 }
}
\maketitle \vspace{-1.2em}

\begin{abstract}
We provide a compactness criterion for the set  of laws $\fP^{ac}_{sem}(\Theta)$ on the Skorokhod space for which the canonical process $X$ is a semimartingale having absolutely continuous characteristics with differential characteristics taking values in some given set $\Theta$ of L\'evy triplets. Whereas boundedness of $\Theta$ implies tightness of $\fP^{ac}_{sem}(\Theta)$, closedness fails in general, even when choosing $\Theta$ to be additionally closed and convex, 
as  a sequence of purely discontinuous martingales  may converge to a diffusion. To that end, we provide a necessary and sufficient condition that prevents the purely discontinuous martingale part in the canonical representation of $X$ to create a diffusion part in the limit. As a result, we obtain a sufficient criterion 
for $\fP^{ac}_{sem}(\Theta)$ to be compact, which turns out to be also a necessary one if the geometry of $\Theta$ is similar to a box on the product space. 

As an application, we consider a semimartingale optimal transport problem, where the transport plans are elements of $\fP^{ac}_{sem}(\Theta)$. We prove the existence of an optimal transport law $\widehat{\P}$ and obtain a duality result extending the classical Kantorovich duality to this setup.
\end{abstract}

\vspace{.9em}

{\small

}
\section{Introduction}\label{sec:Intro}
The goal of this paper is to provide a compactness criterion  for semimartingale laws. Given a set $\Theta$ of L\'evy triplets, we denote by $\fP^{ac}_{sem}(\Theta)$ the set of all probability measures $\P$ on the Skorokhod space for which the canonical process $X=(X_t)_{0\leq t \leq T}$ is a semimartingale with differential characteristics taking values in $\Theta$; this means that the semimartingale characteristics $(B^\P,C^\P,\nu^\P)$ are of the form $(b^\P_t\,dt, c^\P_t\,dt, F^\P_t\,dt)$ and the processes $(b^\P,c^\P,F^\P)$ evolve in $\Theta$. For simplicity in the introduction, let us only consider laws with starting point $X_0=0$. We restrict our attention to sets of L\'evy triplets $\Theta\subseteq \R^d\times \S^d_+ \times \cL$ satisfying the following boundedness condition
\begin{equation}\label{eq:cond-B-Intro}
\sup_{(b,c,F)\in \Theta} \Big\{|b|+|c| + \int_{\R^d} |x|^2 \wedge |x|\, F(dx)\Big\}<\infty,
\end{equation}
where $\cL$ denotes the set of L\'evy measures.
It turns out that the condition \eqref{eq:cond-B-Intro} on $\Theta$ implies tightness of $\fP^{ac}_{sem}(\Theta)$ and that under any limit law $\P_0$ of a sequence $(\P_n) \subseteq \fP^{ac}_{sem}(\Theta)$ the canonical process remains a semimartingale with absolutely continuous characteristics, see Corollary~\ref{co:tightness-P} and Proposition~\ref{prop:limit-P-P-enlarge}.

Finding tightness conditions on semimartingale laws is a well-studied problem. In 
\cite{LepeltierMarchal.76}, 
tightness of solutions of a martingale problem is proven. In fact, \eqref{eq:cond-B-Intro} can be seen as the analogue of the boundedness conditions on the triplet in the integro-differential operator of the martingale problem in \cite{LepeltierMarchal.76}. In \cite{Rebolledo.79}, conditions on the predictable quadratic covaration  of a sequence of locally square-integrable martingales $(M_n)$ is studied, which guarantee tightness of the laws of $(M_n)$. Tightness for semimartingale laws  was studied in \cite{MeyerZheng.84} on the space of c\`adl\`ag paths, but endowed with a weaker topology than the usual Skorokhod $J_1$-topology, which makes it easier to obtain tightness. They show that a sequence of quasimartingales laws with uniformly bounded conditional variation is tight, and any such limit is a quasimartingale law. We refer to \cite{JacodMeminMetivier.83,Billingsley.99,JacodShiryaev.03} for general tightness results for 
processes.

For continuous semimartingales, i.e. when $F\equiv0$, it was shown in \cite{TanTouzi.11} using techniques developed in \cite{Zheng.85} that if in addition to \eqref{eq:cond-B-Intro} $\Theta\subseteq \R^d \times \S^d_+$ is closed and convex, then $\fP^{ac}_{sem}(\Theta)$ is compact. In other words, in the continuous case, closedness and boundedness of $\Theta$ provides compactness of the corresponding set of probability measures
$\fP^{ac}_{sem}(\Theta)$. 

However, on the Skorokhod space, it is straightforward to see that closedness of $\fP^{ac}_{sem}(\Theta)$ may fail, even when choosing $\Theta$, in addition to \eqref{eq:cond-B-Intro}, to be closed (and convex). We are not aware of any closedness criterion of that type. To understand the difficulty, consider  in
dimension one the set 
\begin{equation*}
\Theta:=\Big\{(b,c,F) \in \R \times [0,\infty)\times \cL\, \Big| \, b=0,\, c=0, \, \supp(F)\subseteq \{|x|\leq 1\}\Big\}.
\end{equation*}
We see that $\Theta$ satisfies \eqref{eq:cond-B-Intro}, is convex, and it is straightforward to argue that $\Theta$ is closed. However, $\Theta$ contains the sequence of L\'evy triplets $\big((0,0, n\,\delta_{\frac{1}{\sqrt{n}}})\big)_{n \in \N}$. Consider the  sequence of laws $(\P_n) \subseteq \fP^{ac}_{sem}(\Theta)$, where each element $\P_n$ is defined as the unique law such that $X$ becomes a L\'evy process with triplet $(0,0, n\delta_{\frac{1}{\sqrt{n}}})$. Under each $\P_n$, $X$ 
is the scaled compensated Poisson process $\frac{1}{\sqrt{n}}(N^{n}_t -nt)$ with intensity $n$, which is known to converge weakly to Brownian motion. However, the Wiener measure $\P_0$ is not an element of $\fP^{ac}_{sem}(\Theta)$, as the corresponding triplet $(0,1,0)$ is not an element of $\Theta$. So closedness of $\fP^{ac}_{sem}(\Theta)$ fails.

The key difficulty of getting closedness of $\fP^{ac}_{sem}(\Theta)$ is that a sequence of purely discontinuous martingales may converge to a diffusion, and this (additional) diffusion may destroy the closedness of $\fP^{ac}_{sem}(\Theta)$, even when choosing $\Theta$ to be closed. Therefore,
we need to characterize the condition on $\Theta$ that prevents  the purely discontinuous (local) martingale part of $X$ to create an additional diffusion part in the limit. 

We show in Theorem~\ref{thm:purely-disc} that any limit  of a sequence $(\P_n)$ of purely discontinuous martingale laws being in $\fP^{ac}_{sem}(\Theta)$ is a martingale law; the necessary and sufficient condition for the limit $\P_0$ to be a purely discontinuous martingale law turns out to be
\begin{equation*}
\lim\limits_{\delta\downarrow0}\limsup_{n \to \infty} \E^{\P_n}\Big[\int_0^T\int_{\{|x|\leq \delta\}} |x|^2\, F^{\P_n}_t(dx)\,dt\Big] =0.
\end{equation*}

Now, the intuition is the following. If we can exclude the sequence of purely discontinuous martingale parts to create an additional diffusion in the limit, then the boundedness condition \eqref{eq:cond-B-Intro} together with closedness of $\Theta$ should lead to compactness of $\fP^{ac}_{sem}(\Theta)$. The condition
\begin{equation}\label{eq:cond-J-Intro}
\lim\limits_{\delta\downarrow0}\sup_{(b,c,F)\in \Theta}\int_{\{|x|\leq\delta\}} |x|^2 \, F(dx)=0
\end{equation}
is the weakest condition on $\Theta$ one can impose that prevents limits of purely discontinuous martingales to possess a diffusion part. 

We show in Theorem~\ref{thm:compactness} that  if $\Theta$ is closed, convex and satisfies \eqref{eq:cond-B-Intro}, then \eqref{eq:cond-J-Intro}  is a sufficient condition that 
$\fP^{ac}_{sem}(\Theta)$ is compact. Moreover, if the geometry of $\Theta \subseteq \R^d\times \S^d_+ \times \cL$ is similar to the one of a box on the product space (see Definition~\ref{def:box}), then  \eqref{eq:cond-J-Intro} is also a necessary condition. 

As an application of the compactness criterion for $\fP^{ac}_{sem}(\Theta)$, we study a semimartingale optimal transport problem. In classical optimal transport theory, the Monge-Katorovich transport problem for given distributions $\mu_0, \mu_1$ on $\R^d$ and cost function $c:\R^d \times \R^d \to [0,\infty)$ consists of minimizing the functional
\begin{equation*} 
\P \mapsto \E^\P\big[c(X_0,X_1)\big]
\end{equation*}
over all distributions $\P$ on $\R^d\times \R^d$ with initial marginal $\P\circ X_0^{-1}=\mu_0$ and final marginal $\P\circ X_1^{-1}=\mu_1$. A duality result was established in \cite{Kantorovich.42,Kellerer.84} under some suitable conditions on the cost function $c$. We refer to \cite{Villani.09} as reference for classical optimal transport. Recently, martingale optimal transport has been a very active field of research (\cite{BeiglbockHenryLaborderePenkner.11,GalichonHenryLabordereTouzi.11,BeiglbockJuillet.12,HenryLabordereTouzi.16,CampiLaachirMartini.14,BeiglbockNutzTouzi.16,GuoTanTouzi.15}), also in connection to Skorokohod Embedding Problems (\cite{Hobson.98,Hobson.11,Obloj.04,HenryLabordereOblojSpoidaTouzi.16,BeiglbockHenryLabordereTouzi.15}) and to robust pricing in mathematical finance 
(\cite{BeiglbockHenryLaborderePenkner.11,GalichonHenryLabordereTouzi.11,BeiglbockJuillet.12,FahimHuang.16,DolinskySoner.12,DolinskySoner.15,Stebegg.14}), to name but a few. 

Semimartingale optimal transport was introduced in \cite{MikamiThieullen.06} as an extension of the classical Monge-Kantorovich optimal transport problem, where for a given cost function $c^\P(X)$ (which may vary in $\P$) and given distributions $\mu_0, \mu_1$ on $\R^d$, one minimizes 
\begin{equation*}
\P \mapsto \E^\P[c^\P(X)]
\end{equation*} 
over all semimartingale laws $\P$ with initial marginal $\P\circ X_0^{-1}=\mu_0$ and final marginal $\P\circ X_1^{-1}=\mu_1$. We see that the semimartingale optimal transport problem recover the classical Monge-Kantorovich optimal transport problem when choosing $c^\P(X)\equiv c(X_0,X_1)$.

We study the following semimartingale optimal transport problem. For any distributions $\mu_0$, $\mu_1$ on $\R^d$ we denote by $\fP_{\Theta}(\mu_0, \mu_1)$ the set of all probability measures $\P \in \fP^{ac}_{sem}(\Theta)$ which  have initial marginal $\P\circ X_0^{-1}=\mu_0$ and final marginal $\P\circ X_1^{-1}=\mu_1$. Then, given a cost function $L$, distributions $\mu_0, \mu_1$ on $\R^d$, and a set $\Theta\subseteq \R^d\times \S^d_+\times \cL$,  we want to minimize the functional
\begin{equation*}
\P \mapsto \E^\P\Big[\int_0^1 L(t,X,b^\P_t,c^\P_t,F^\P_t)\,dt\Big]
\end{equation*}
over all probability measures $\P \in \fP_{\Theta}(\mu_0,\mu_1)$.
We prove the existence of a minimizer $\widehat{\P}\in \fP_{\Theta}(\mu_0, \mu_1)$ and a  duality theorem which is an extension of the classical Kantorovich duality to this setup (see \cite{Kantorovich.42,Kellerer.84}). When choosing $\Theta:=\R^d \times \{I_{d\times d}\}\times \{0\}$ and $\Theta\subseteq \R^d \times \S^d_+\times \{0\}$, we recover the semimartingale optimal transport and the corresponding strong duality result in \cite{MikamiThieullen.06} and \cite{TanTouzi.11}, respectively. We also refer to \cite{CruzeiroLassalle.15,Mikami.15} for related work on semimartingale optimal transport for continuous semimartingales.
Following the arguments of \cite{MikamiThieullen.06,TanTouzi.11}, the existence of a minimizer $\widehat{\P}$ and the duality theorem is obtained in Theorem~\ref{thm:main} using classical convex duality results, which require that the value function
\begin{equation*}
(\mu_0,\mu_1)\mapsto V(\mu_0,\mu_1):= \inf_{\P \in \fP_{\Theta}(\mu_0, \mu_1)}E^\P\Big[\int_0^1 L(t,X,b^\P_t,c^\P_t,F^\P_t)\,dt\Big]
\end{equation*} 
is convex and lower semicontinuity. 
The lower semicontinuity of the value function heavily depends on the compactness property $\fP^{ac}_{sem}(\Theta)$, which we now can ensure from the compactness criterion by imposing $\Theta$ to satisfy \eqref{eq:cond-B-Intro} and \eqref{eq:cond-J-Intro}.

The remainder of this paper is organized as follows. In Section~2, we introduce the notions and present the main results. In Section~3, we lift the laws $\fP^{ac}_{sem}(\Theta)$ to an enlarged space $\ov \Omega$, where we prove our main results about the characterization when purely discontinuous martingales can create a diffusion in the limit and the compactness criterion for the semimartingale laws. The enlarged space allows us to consider the joint distribution of the summands in the canonical representation of $X$. The main tool for the proof is the Skorokhod representation theorem, which enables us to translate weak convergence to pointwise convergence. In Section~4, we prove the above results on the original Skorokhod space using the corresponding results derived on the enlarged space. To that end, an analysis of semimartingale characteristics with respect to a filtration smaller than the original one on the enlarged space is needed. In Section~5, we prove the existence of a minimizer and the duality result of the semimartingale optimal transport problem.

\section{Setup and Main Results}\label{sec:main-thms}
Fix $d \in \N$, $T<\infty$ and let $\Omega:=\D([0,T],\R^d)$ be the space of all c\`adl\`ag paths $\omega=(\omega_t)_{0\leq t \leq T}$ endowed with the usual Skorokhod $J_1$-topology, where we equip $\Omega$ with the corresponding Borel $\sigma$-field $\cF$. Denote by $X=(X_t)_{0\leq t \leq T}$ the canonical process $X_t(\omega)=\omega(t)$ and by $\F=(\cF_t)_{0\leq t \leq T}$ the (raw) filtration generated by $X$. Moreover, we denote by $\fM_1(\Omega)$ the space of all probability measures on $(\Omega,\cF)$, which is a Polish space by the usual topology of weak convergence, see e.g. \cite[Chapter~7]{BertsekasShreve.78}.

For a given probability measure $\P$ and a filtration $\G$ on $(\Omega,\cF)$, we call a $\G$-adapted process $Y$ with c\`adl\`ag paths a $\P$-$\G$-semimartingale if there exist right-continuous, $\G$-adapted processes $M$ and $A$ with $M_0=A_0=0$ such that $M$ is a $\P$-$\G$-local martingale, $A$ has paths of finite varation $\P$-a.s. such that $Y=Y_0 + M + A \ \P$-a.s.

Fix a continuous truncation function $h:\R^d \to \R^d$; that is a bounded continuous function such that $h(x)=x$ in a neighborhood of zero. Let $\P$ be a probability measure on $(\Omega, \cF)$ under which the canonical process $X$ is a $\P$-$\F$-semimartingale. Denote by $(B^\P, C^\P, \nu^\P)$ its $\P$-$\F$-semimartingale charateristics, this means that $(B^\P, C^\P, \nu^\P)$ is a triplet of processes such that $\P$-a.s., $B^\P$  is the predictable finite variation part in the canonical decomposition of $X-\sum_{0\leq s \leq \cdot} (\Delta X_s-h(\Delta X_s))$ under $\P$, $C^\P$ is the quadratic covariation  of the continuous local martingale part of $X$ under $\P$ and $\nu^\P$ is the $\P$-compensator of the measure $\mu^X$ associated to the jumps of $X$. We refer to \cite{JacodShiryaev.03} as our main reference for standard notions related to general semimartingale theory. Under a given probability measure $\P$, $X$ is a $\P$-$\F$-semimartingale if and only if it is one with respect to the right continuous filtration $\F_+$ or the usual augmentation $\F^\P_+$, and the semimartingale characteristics with these filtrations are the same, see \cite[Proposition~2.2]{NeufeldNutz.13a}.
We denote by $\fP_{sem}$ the set of all probability measures on $(\Omega,\cF)$ such that $X$ is a  $\P$-$\F$-semimartingale.

Given a semimartingale law $\P$, we say that $X$ has absolutely continuous characteristics under $\P$-$\F$ if $(dB^\P,dC^\P,d\nu^\P)= (b^\P_t\,dt, c^\P_t\,dt, F^\P_t\,dt)$. The differential characteristics $(b^\P,c^\P,F^\P)$ take values in $\R^d\times \S^d_+ \times \cL$, where $\S^d_+$ denotes the set of all  symmetric, nonnegative definite $d\times d$ matrices and 
\begin{equation*}
\cL:=\Big\{F \mbox{ measure on } \R^d \, \Big| \, \int_{\R^d} |x|^2 \wedge 1 \, F(dx)<\infty \mbox{ and } F(\{0\})=0 \Big\}
\end{equation*}
denotes the set of L\'evy measures. 
We write
\begin{equation*}
\fP^{ac}_{sem}:=\big\{ \P \in \fP_{sem}\, \big| \, \mbox{$X$ has absolutely continuous characteristics} \big\}.
\end{equation*}

Given a set $\Theta\subseteq \R^d \times \S^d_+ \times \cL$, we denote by
\begin{align*}
\fP^{ac}_{sem}(\Theta):= \big\{ \P \in \fP^{ac}_{sem} \, \big| \, (b^\P,c^\P, F^\P) \in \Theta \ \,  \P\times dt\mbox{-a.e.} \big\}
\end{align*}  
the set of all semimartingale laws in $\fP^{ac}_{sem}$ with differential characteristics taking value in $\Theta \ \P\times dt$-a.s. Given a set $\Gamma_0\subseteq \fM_1(\R^d)$ of distributions on $\R^d$, we write
\begin{equation*}
\fP^{ac}_{sem}(\Theta)(\Gamma_0):=\big\{ \P \in \fP^{ac}_{sem}(\Theta) \, \big| \, \P \circ X_0^{-1} \in \Gamma_0 \big\}
\end{equation*}
for those elements in $\fP^{ac}_{sem}(\Theta)$ with corresponding initial distribution lying in $\Gamma_0$. 
%
%

To provide our compactness criterion for semimartingale laws, we first need to introduce a topology on $\Theta$.
To that end we endow $\cL$ with  the topology of weak convergence  induced by the bounded continuous functions on $\R^d$ vanishing in a neighborhood of the origin. More precisely, given $(F_n)_{n \in \N} \subseteq \cL$ and $F \in \cL$, we say that $(F_n)_{n\in \N}$ converges to $F$ if $\lim_{n \to \infty} \int_{\R^d} g(x)\,F_n(dx)=\int_{\R^d} g(x)\,F(dx)$ for all bounded continuous functions $g$ on $\R^d$ vanishing in a neighborhood of the origin.
Having defined a topology on $\cL$, we can equip $\R^d\times \S^d_+ \times \cL$ with the corresponding product topology.

Consider the following condition on the values of the differential characteristics.
\begin{conditionB} A set  $\Theta\subseteq \R^d \times \S^d_+ \times \cL$ satisfies Condition (B) if
	\begin{equation}\label{eq:bound-characteristics}
	\cK:=\sup_{(b,c,F)\in \Theta} \Big\{|b|+|c| + \int_{\R^d} |x|^2 \wedge |x|\, F(dx)\Big\}<\infty.
	\end{equation}
\end{conditionB}
%
%
%
%
Note that Condition~(B) guarantees that both the continuous and purely discontinuous local martingale part of the canonical representation of $X$ under any $\P \in \fP^{ac}_{sem}(\Theta)$ are true martingales, and every element $\P \in \fP^{ac}_{sem}(\Theta)$ being a $\sigma$-martingale law satisfying $\E^\P[|X_0|]<\infty$ is a true martingale law, see the proof of \cite[Lemma~5.2]{NeufeldNutz.13b}.

%
Now, let us state our first result, which provides a necessary and sufficient condition for limits of purely discontinuous martingale measures to be again a purely discontinuous martingale measure. To that end, denote
\begin{align*}
\fP^{ac}_{m}&:=\big\{ \P \in \fP^{ac}_{sem}\,|\, X \mbox{ is a $\P$-$\F$-martingale}\big\},\\
\fP^{ac}_{m,d}&:=\big\{ \P \in \fP^{ac}_{m}\,|\, X \mbox{ is a  $\P$-$\F$-purely discontinuous martingale}\big\}\\& \, \, =\big\{ \P \in \fP^{ac}_{m}\,|\, C^\P=0\big\},\\
\fP^{ac}_{m,d}(\Theta)&:= \big\{ \P \in \fP^{ac}_{m,d}\,|\,(b^\P,c^\P, F^\P) \in \Theta \ \, \P\times dt\mbox{-a.e.} \big\}.
\end{align*}
\begin{theorem}\label{thm:purely-disc}
	Let $\Theta\subseteq \R^d \times \S^d_+ \times \cL$ satisfy Condition (B) and $(\P_n)_{n \in \N}\subseteq \fP^{ac}_{m,d}(\Theta)$ be a sequence of purely discontinuous martingale measures which converges weakly to some $\P_0 \in \fM_1(\Omega)$ satisfying $\E^{\P_0}[|X_0|]<\infty$. Then the following hold true.
	\begin{enumerate}
		\item[1)] We have $\P_0 \in \fP^{ac}_{m}$.
		\item[2)] We have the following necessary and sufficient criterion for the limit law to be a purely discontinuous martingale measure:
		\begin{equation}\label{eq:criterion-purely-disc}
		\P_0 \in \fP^{ac}_{m,d}\Longleftrightarrow \lim\limits_{\delta\downarrow0}\limsup_{n \to \infty} \E^{\P_n}\Big[\int_0^T\int_{\{|x|\leq \delta\}} |x|^2\, F^{\P_n}_t(dx)\,dt\Big] =0.
		\end{equation} 
	\end{enumerate} 
\end{theorem}
\begin{corollary}\label{co:closedness-M-d-law}
	Let $\Theta\subseteq \R^d \times \S^d_+ \times \cL$ satisfy Condition (B) and $(\P_n)_{n \in \N}\subseteq \fP^{ac}_{sem}(\Theta)$. Assume that the sequence of laws $(\P_n \circ (M^{d,\P_n})^{-1})_{n \in \N}$ of the corresponding purely discontinuous martingale part of the canonical process $X$ under $\P_n$ 
	converges to some law $\P_0 \in \fM_1(\Omega)$. Then 1) and 2) of Theorem~\ref{thm:purely-disc} hold true, too. 
\end{corollary}
\begin{remark}\label{rem:thm:purely-disc-infinite-time}
	The condition for a martingale to be 
	a purely discontinuous one
	is of local nature. Therefore, the above theorem and corollary also hold true when considering (semi-) martingale laws defined on the time interval $[0,\infty)$, up to the slight modification that the right-hand side of \eqref{eq:criterion-purely-disc} must then hold for all $T \in [0,\infty)$.
\end{remark}
%
%
We can use Theorem~\ref{thm:purely-disc} to obtain a necessary and sufficient criterion on $\Theta$ for $\fP^{ac}_{sem}(\Theta)(\Gamma_0)$ to be compact. To that end, we introduce the following condition.
\begin{conditionJ} A set  $\Theta\subseteq \R^d \times \S^d_+ \times \cL$ satisfies Condition (J) if
	\begin{equation*}
	\lim\limits_{\delta\downarrow0}\sup_{(b,c,F)\in \Theta}\int_{\{|x|\leq\delta\}} |x|^2 \, F(dx)=0.
	\end{equation*}
\end{conditionJ}
Moreover, denote
\begin{align*}
\proj_c(\Theta)&:=\big\{c \in \S^d_+ \, \big| \, \mbox{ there exists $(b,F)\in \R^d\times \cL$ such that }\, (b,c,F)\in \Theta\big\},\\
\proj_F(\Theta)&:=\big\{F \in \cL \, \big| \, \mbox{ there exists $(b,c)\in \R^d\times \S^d_+$ such that }\, (b,c,F)\in \Theta\big\}.
\end{align*}
%
%
\begin{definition}\label{def:box}
	Given $\Theta\subseteq \R^d \times \S^d_+ \times \cL$, we say that $\proj_c(\Theta)$ does not depend on $\proj_F(\Theta)$ if for any fixed $F \in \proj_F(\Theta)$ 
	\begin{equation*}
	\big\{c \in \proj_c(\Theta)\, \big|\,\mbox{ there exists $b \in \R^d$ such that }\, (b,c,F) \in \Theta \big\}=\proj_c(\Theta).
	\end{equation*}
\end{definition}
%
%
The condition on $\Theta$ introduced in Definition~\ref{def:box} can be seen as a generalization of a box. Indeed, for any subsets $D_b\subseteq \R^d, D_c \subseteq \S^d_+$, $D_ F\subseteq \cL$, the box
$\Theta_{Box}:=D_b\times D_c \times D_F$ satisfies the above condition.
%
%
Let us state the compactness criterion for $\fP^{ac}_{sem}(\Theta)(\Gamma_0)$.
\begin{theorem}\label{thm:compactness}
	Let $\Theta\subseteq \R^d \times \S^d \times \cL$ be closed, convex and satisfy Condition (B). Then the following hold true:
	\begin{enumerate}
		\item[1)] $\Theta$ satisfying Condition (J) implies closedness of  $\fP^{ac}_{sem}(\Theta)$.
		\item [2)] On the other hand, if additionally $\Theta$ is satisfying the condition in Definition~\ref{def:box}, then closedness of $\fP^{ac}_{sem}(\Theta)$ implies that $\Theta$ satisfies Condition (J).
	\end{enumerate}
	In addition, consider a set of distributions $\Gamma_0 \subseteq \fM_1(\R^d)$ being compact. Then:
	\begin{enumerate}
		\item[3)] $\Theta$ satisfying Condition (J) implies compactness of  $\fP^{ac}_{sem}(\Theta)(\Gamma_0)$.
		\item [4)] On the other hand, if additionally $\Theta$ is satisfying the condition in Definition~\ref{def:box}, then compactness of $\fP^{ac}_{sem}(\Theta)(\Gamma_0)$ implies that $\Theta$ satisfies Condition (J).
	\end{enumerate}
\end{theorem}
%
%
\begin{remark}\label{rem:thm:compactness-vgl-closedness}
	Under the above conditions, tightness of $\Gamma_0 \subseteq \fM_1(\R^d)$ implies tightness of $\fP^{ac}_{sem}(\Theta)(\Gamma_0)$, see Corollary~\ref{co:tightness-P}. Moreover, closedness of $\Gamma_0$ ensures that any limit law $\P_0$ of a sequence $(\P_n) \in \fP^{ac}_{sem}(\Theta)(\Gamma_0)$ satisfies $\P_0 \circ X_0^{-1}\in \Gamma_0$. Therefore, part 3) and 4) are simple consequences of part 1) and 2).
\end{remark}
\begin{remark}\label{rem:compactness:infinite-time}
	By the same arguments, the results of Theorem~\ref{thm:compactness} also hold true when considering semimartingale laws defined on the time interval $[0,\infty)$.  
\end{remark}
%
%
Next, we provide a second compactness criterion for $\fP_{sem}^{ac}(\Theta)(\Gamma_0)$
expressed by the infinitesimal generator of the L\'evy laws in $\fP_{sem}^{ac}(\Theta)$. To that end, for any set of L\'evy triplets $\Theta$ set
\begin{equation*}
\fP_L(\Theta):=\big\{ \P \in \fP^{ac}_{sem}(\Theta)\,|\, X \mbox{ is a $\P$-$\F$-L\'evy process}\big\}.
\end{equation*}
For any $(b,c,F) \in \Theta$ we consider the infinitesimal generator $\mathfrak{L}^{(b,c,F)}$, which satisfies for any smooth enough function $f$ that  
\begin{equation*}
\label{def:infit-gener}
\begin{split}
(\mathfrak{L}^{(b,c,F)}f)(x):= &\sum_{i = 1}^d b^i \frac{\partial f}{\partial x^i}(x) + \frac{1}{2}\sum_{i,j=1}^d c^{ij}\frac{\partial^2 f}{\partial x^i \partial x^j}(x) \\
&+ \int_{\R^d} \Big(f(x + y) -f(x) - \sum_{i =1}^d\frac{\partial f}{\partial x^i}(x)h^i(y)\Big) \, F(dy),
\end{split}
\end{equation*}
and introduce the L\'evy exponent function $\psi^{(b,c,F)} \in C(\R^d, \mathbb C)$  given by 
\begin{equation*}
\label{def:exponent}
\psi^{(b,c,F)}(x) = \mathrm{i}x\cdot b - \frac{1}{2}x \cdot c \cdot x + \int_{\R^d}(\mathrm{e}^{\mathrm{i}x \cdot y} - 1 - \mathrm{i}x\cdot h(y) )\, F(dy).
\end{equation*}
Furthermore, consider a function $u\colon \R^d \times \S^d_+ \times \cL \to \R^d \times \S^d_+ \times \cL$ defined by
\begin{equation}
\label{eq:Def-func-u}
(b,c,F)\mapsto (b,\widetilde{c}, F), \quad \mbox{where} \quad  \widetilde{c}^{ij}=c^{ij}+ \int_{\R^d} h^i(x) h^j(x)\, F(dx).
\end{equation}
Denote by $C^2_b(\R^d)$ the set of bounded continuous functions with bounded continuous derivatives up to the second order.
Then we obtain the following second compactness criterion for semimartingale laws $\fP^{ac}_{sem}(\Theta)(\Gamma_0)$.
\begin{theorem}
	\label{thm:compactness2}
	Let $\Theta \subset \R^d \times \S^d_+ \times \cL$ be closed, convex and satisfy condition (B).
	Then the following are equivalent:
	\begin{enumerate}
		\item[1)] $u(\Theta)\subseteq \R^d \times \S^d_+ \times \cL$ is closed. 
		\item[2)] For any compact set $\Gamma_0 \subseteq \fM_1(\R^d)$ of distributions on $\R^d$, the set of laws   $\fP^{ac}_{sem}(\Theta)(\Gamma_0)$ is compact. 
		\item[3)] $\fP_{L}(\Theta)$ is compact. 
		\item[4)] The set $\{\fL^{(b,c,F)}: (b,c,F) \in \Theta\}$ is sequentially compact with respect to $C^2_b(\R^d)$-test functions for the topology of pointwise convergence. More precisely, for every sequence $((b^n,c^n,F^n))_{n \in \N} \subseteq \Theta$ there exists a subsequence $((b^{n_k},c^{n_k},F^{n_k}))_{k\in  \N}$ and  $(b^0,c^0,F^0) \in \Theta$ such that for all $f \in C^2_b(\R^d)$ the subsequence of functions $((\fL^{(b^{n_k},c^{n_k},F^{n_k})}f))_{k \in \N}$ converges pointwise to the function $(\fL^{(b^0,c^0,F^0)}f)$. 
		%
		\item[5)] The set of functions $\{\psi^{(b,c,F)}: (b,c,F) \in \Theta\}$ is sequentially compact 
		for the  topology of pointwise convergence.  More precisely, for every sequence $((b^n,c^n,F^n))_{n \in \N} \subseteq \Theta$ there exists a subsequence $((b^{n_k},c^{n_k},F^{n_k}))_{k\in  \N}$ and  $(b^0,c^0,F^0) \in \Theta$ such that the subsequence of functions $(\psi^{(b^{n_k},c^{n_k},F^{n_k})})_{k \in \N}$ converges pointwise to the function $\psi^{(b^{0},c^{0},F^{0})}$.
		%
	\end{enumerate}
\end{theorem}
\begin{remark}\label{rem:compactness2:infinite-time}
	Since Conditions~1), 4) and 5) are local conditions, Theorem~\ref{thm:compactness2} remains valid  for semimartingale laws  defined on the time interval $[0,\infty)$.  
\end{remark}
\begin{remark}
	\label{rem:compactness2}
	For most sets of L\'evy triplets $\Theta$, if one wants to check whether one of the equivalent conditions in  Theorem~\ref{thm:compactness2} to $\fP_{sem}^{ac}(\Theta)(\Gamma_0)$ being compact holds true, one needs to verify 
	that any sequence of (sums of compensated) small jumps cannot create an additional diffusion in the limit.  This immediately links to  the results obtained in Theorem~\ref{thm:purely-disc} and Theorem~\ref{thm:compactness}. In particular,  Theorem~\ref{thm:compactness2} cannot replace Theorem~\ref{thm:purely-disc} and Theorem~\ref{thm:compactness}, which are much deeper results as they explain how one can prevent purely discontinuous martingales to create an additional diffusion in the limit which then  leads to a compactness criterion expressed through Condition~(J).  
	However, if roughly speaking, $\Theta$ is described by a boundedness condition on  the modified second differential characteristic (cf., e.g., Example~\ref{ex:cond-J}), then Condition~1) in Theorem~\ref{thm:compactness2} can be verified directly. The intuition why this holds  is the following. 
	
	A semimartingale is described by its drift, its diffusion, its sum of big jumps and its sum of compensated small jumps. Notice that drifts converge to a drift, diffusions converge to a diffusion, and the sum of big jumps converges to a sum of big jumps; we refer to  Subsection~\ref{subsec:semimart-enlarge} for a precise statement. However, a sequence of  sums of compensated small jumps converges to a sum of compensated small jumps and a possible additional diffusion. This is why if no additional diffusion may be created in the limit, then closedness of $\Theta$ implies closedness of $\fP_{sem}^{ac}(\Theta)(\Gamma_0)$, and hence also compactness of $\fP_{sem}^{ac}(\Theta)(\Gamma_0)$, due to $\Theta$ satisfying Condition~(B). But notice that the modified second characteristic is defined as the sum of the (quadratic variation of the) diffusion and the (compensator of the sum of the squared) small jumps. Therefore, a sequence of modified second characteristics converges to the modified second characteristic of the limit process. Thus, if $\Theta$ is described by a constraint only on the modified second differential characteristic, without any constraints on the second differential characteristic, then Condition~1) in Theorem~\ref{thm:compactness2} can be verified directly due to the definition of the map $u$ involved in Condition~1).
\end{remark}
%
The proofs of Theorem~\ref{thm:purely-disc}, Theorem~\ref{thm:compactness}, and Theorem~\ref{thm:compactness2} are given in Section~\ref{sec:prf-thm-semi}.
Now, we present an example showing that in Theorem~\ref{thm:compactness}, Condition~(J) is not necessary for $\fP^{ac}_{sem}(\Theta)(\Gamma_0)$ 
to be compact.
\begin{example}\label{ex:cond-J}
	Let the dimension be $d=1$. For any $F \in \cL$ denote by $\supp(F)$ the support of the  measure 
	$F$ and define
	$\Theta \subseteq \R\times [0,\infty)\times \cL$ by
	\begin{equation*}
	\Theta:=\Big\{(b,c,F) \, \Big| \, \supp(F)\subseteq\{|x|\leq 1\},\ b=0, \ c+ \int_{\R} |x|^2\, F(dx)=1\Big\},
	\end{equation*}
	where the (differential) characteristics are defined with respect to a continuous truncation function $h$ satisfying $h(x)=x$ on $\{|x|\leq 1\}$. Observe that $\Theta$ does not satisfy the condition in Definition~\ref{def:box}.
	It turns out that $\Theta$ is closed, convex and satisfies Condition~(B), but Condition~(J) fails. However, 
	the set $\fP^{ac}_{sem}(\Theta)(\{\delta_0\})$ is compact.  Compactness of $\fP^{ac}_{sem}(\Theta)(\{\delta_0\})$ is shown by verifying that Condition~1) in Theorem~\ref{thm:compactness2} holds true. We provide the proof at the end of Section~\ref{sec:prf-thm-semi}.
\end{example}
\begin{remark}\label{rem:filtration}
	As mentioned above, we know from \cite[Proposition~2.2]{NeufeldNutz.13a} that the precise choice of the filtration in the
	definition of the set $\fP_{sem}$ and its subsets introduced above is not crucial in the sense that the results above also hold true with respect to any filtration $\F\subseteq\G\subseteq \F^\P_+$.
\end{remark}

\begin{remark}\label{rem:empty-martingale}
	Even though $\fP^{ac}_{sem}(\Theta)\neq \emptyset$, we might have that
	$\fP^{ac}_{m}(\Theta)=\emptyset$ or $\fP^{ac}_{m,d}(\Theta)=\emptyset$. However, one can easily impose conditions (additional to the one in Theorem~\ref{thm:purely-disc}) on $\Theta$ such that both sets above are nonempty. In fact, if in  addition to Condition~(B) we assume that
	\begin{equation*}
	\Theta \subseteq \Big\{(b,c,F) \in \R^d \times \S^d_+\times \cL \, \Big| \, b + \int_{\R^d} [x-h(x)]\,F(dx)=0 \Big\}=:\Theta_{\cM},
	\end{equation*}
	then we have $\{\P \in \fP^{ac}_{sem}(\Theta)\,|\, \E^\P[|X_0|]<\infty\}=\fP^{ac}_{m}(\Theta)$. Moreover, if we  impose that
	\begin{equation*}
	\Theta \subseteq \big\{(b,c,F) \in \Theta_{\cM} \, \big| \, c=0\big\}=:\Theta_{\cM^d},
	\end{equation*}
	then we even obtain $\{\P \in\fP^{ac}_{sem}(\Theta)\,|\, \E^\P[|X_0|]<\infty\}=\fP^{ac}_{m,d}(\Theta)$.
\end{remark}
%
%
As an application of Theorem~\ref{thm:compactness}, we consider the following semimartingale optimal transportation problem. First, let $T=1$ and fix any set $\Theta\subseteq \R^d \times \S^d \times \cL$ being closed, convex and satisfying Condition (B) and Condition (J). 
To shorten the notation, we write $\fP_\Theta\equiv \fP^{ac}_{sem}(\Theta)$. Given two arbitrary probability measures $\mu_0$ and $\mu_1$ in $\fM_1(\R^d)$, we denote 
\begin{align}
\fP_\Theta(\mu_0)&:=\big\{ \P \in \fP_\Theta \, \big| \,  \P \circ X_0^{-1}=\mu_0\}, 
\nonumber\\
\fP_\Theta(\mu_0,\mu_1)&:=\big\{ \P \in \fP_\Theta(\mu_0) \, \big| \,  \P \circ X_1^{-1}=\mu_1\}. \nonumber 
\end{align} 
The semimartingale $X$ under $\P\in \fP_\Theta(\mu_0,\mu_1)$ can be viewed as a medium of mass transportation from the initial distribution $\mu_0$ to the target distribution $\mu_1$. We couple $\P$ with a transportation cost
\begin{equation*} 
\fJ(\P):= \E^\P \Big[\int_0^1 L(t,X,b^\P_t,c^\P_t,F^\P_t)\,dt\Big],
\end{equation*}
where
$L:[0,1]\times \Omega \times \Theta \to [0,\infty)$ is a given cost function. Then, we are interested in the following optimal transport problem
\begin{equation}\label{eq:primal}
V(\mu_0,\mu_1):= \inf_{\P \in \fP_\Theta(\mu_0,\mu_1)} \fJ(\P),
\end{equation}
using the convention $\inf \emptyset=\infty$. The  goal is to prove a duality result for the minimizing problem \eqref{eq:primal}, which can be seen as an extension of the classical Kantorovich duality in  optimal transportation. To that end, we need some conditions on 
$L$.
\begin{assumption}\label{ass:cost-function-L}
	The cost function $L:(t,\omega,\theta)\in[0,1]\times \Omega \times \Theta \mapsto L(t,\omega,\theta) \in [0,\infty)$ satisfies:
	\begin{enumerate}
		\item[1)] $L$ is nonnegative, continuous in $(t,\omega,\theta)$ and convex in $\theta$.
		\item[2)] $L$ is uniformly continuous in $t$ in the sense that
		\begin{equation*}
		\Delta_t L(\epsilon) := \sup_{0\leq s,t \leq 1, |t-s|<\epsilon, \omega \in \Omega, \theta \in \Theta}\frac{|L(s,\omega,\theta) - L(t,\omega,\theta)|}{1 + L(t,\omega,\theta)} \longrightarrow 0 \  \mbox{as } \varepsilon \to 0.
		\end{equation*}
	\end{enumerate}
\end{assumption}
\begin{remark}\label{rem:ass-L}
	In the case $\Theta=(U\times \{0\})$ where $U\subseteq \R^d\times \S^d_+$ and here $0$ denotes the zero-measure $\in \cL$, Condition 1) and 2) coincide with \cite[Assumption~3.1 and Assumption~3.2]{TanTouzi.11}.
	In fact, when $U=\{I_{d\times d}\}\times \R^d$, Assumption~2) coincide with \cite[Assumption~A.1]{MikamiThieullen.06}.
\end{remark}
We define the dual formulation of \eqref{eq:primal} by 
\begin{equation*} 
\mathcal{V}(\mu_0,\mu_1) := \sup_{\lambda_1 \in C_b(\R^d)}\Big\{ \int_{\R^d} \lambda^{\lambda_1}_0(x)\, \mu_0(dx) -  \int_{\R^d} \lambda_1(x)\, \mu_1(dx)\Big\},
\end{equation*}
where 
\begin{equation*}
\lambda^{\lambda_1}_0(x) := \inf_{\P \in \fP_\Theta(\delta_{x})} \E^\P\Big[\int_0^1 L(t,X,b^\P_t,c^\P_t,F^\P_t)\,dt + \lambda_1(X_1)\Big].
\end{equation*}
The integral $\int_{\R^d} \lambda^{\lambda_1}_0(x)\, \mu_0(dx)$ is well-defined as the function $\lambda^{\lambda_1}_0$ 
is bounded from below (might taking value $\infty$) and is measurable, cf.  Lemma~\ref{le:meas.selec}.
Then, our main  duality result is the following.
\begin{theorem}\label{thm:main}
	Let $\Theta\subseteq \R^d \times \S^d_+ \times \cL$ being closed, convex, satisfying Condition (B) and Condition (J). Moreover, let the cost function $L$ satisfy Assumption~\ref{ass:cost-function-L}. Then
	\begin{equation*}
	V(\mu_0,\mu_1) = \mathcal{V}(\mu_0,\mu_1), \text{ for all }\, \mu_0,\mu_1 \in \fM_1(\R^d),
	\end{equation*}
	and the infimum is achieved by some $\widehat \P \in \fP_\Theta(\mu_0,\mu_1)$ for the primal problem $V(\mu_0,\mu_1)$ defined in \eqref{eq:primal}, whenever it is finite.
\end{theorem}
\begin{remark}\label{rem:empty-transport}
	In general, $\fP_\Theta(\mu_0,\mu_1)$ might be empty.
	However, for example when $\mu_0=\delta_{x}$ for some $x \in \R^d$ and $\mu_1$ is some suitable infinitely divisible distribution (suitable with respect to $\Theta$), then  $\fP_\Theta(\mu_0,\mu_1)$ is clearly nonempty containing at least the corresponding L\'evy law. 
	Moreover, $\fP_\Theta(\mu_0,\mu_1)$ being nonempty is closely related to the Skorokhod Embedding Problem for L\'evy processes with nontrivial initial law.
\end{remark}
\section{A Version of Theorem~\ref{thm:purely-disc}~\&~\ref{thm:compactness} on an Enlarged Space}\label{sec:enlarge}
\subsection{An Enlarged Space}\label{subsec:enlarg}
In this subsection, we introduce an enlarged space. The strategy of the proofs of our main results  is to derive them first in the enlarged space and then conclude for the original space.

Let $\mathcal{C}^+(\R^d):=\{g_i \mid i \in \N\}$ be a countable family of bounded continuous functions on $\R^d$ vanishing in a neighborhood of the origin with the following properties:\\
$\bullet$ it is a law-determining class for L\'evy measures, i.e.\ for any two L\'evy measures $F, F^\prime \in \cL$, if one has $\int_{\R^d} g(x) \, F(dx) = \int_{\R^d} g(x)\,F^\prime(dx)$ for all $g \in \cC^+(\R^d)$, then $F = F^\prime$.\\
$\bullet$ it is a convergence-determining class for the weak convergence induced by the bounded continuous functions on $\R^d$ vanishing in a neighborhood of the origin, i.e.\ 
given any sequence $(F_n)_{n \in \N}  \subseteq \cL$ and $F \in \cL$, if $\lim_{n \to \infty} \int_{\R^d} g(x)\,F_n(dx)=\int_{\R^d} g(x)\,F(dx)$, for all $g \in \cC^+(\R^d)$, then $\lim_{n \to \infty} \int_{\R^d} g(x)\,F_n(dx)=\int_{\R^d} g(x)\,F(dx)$ for all bounded continuous functions on $\R^d$ vanishing in a neighborhood of the origin. \\
$\bullet$ For each $m \in \N$, let $g_{2m}$ be a continuous function on $\R^d$  satisfying for  all $x \in \R^d$ that $0\leq g_{2m}(x)\leq |x|^2 \wedge 1 $ and that
$$
g_{2m}(x) = \begin{cases}
0       & \quad \text{if } |x| \leq \frac{1}{2m};\\
|x|^2 \wedge 1  & \quad \text{if } |x| > \frac{1}{m}.\\
\end{cases}
$$
We refer to \cite[II.2.20]{JacodShiryaev.03} together with \cite[VII.2.7\&2.8]{JacodShiryaev.03} for the existence of such a class $\mathcal{C}^+(\R^d)$.

For any $n \in \N$, write $\Omega_n=\D([0,T],\R^n)$ endowed with the Skorokhod topology, and $\Omega^c_n:=C([0,T],\R^n)$ for the space of continuous function with the usual uniform topology. Recall $\Omega=\Omega_d=\D([0,T],\R^d)$ our original space. We introduce the enlarged space
\begin{equation*} 
\overline{\Omega} := \Omega \times \Omega^c_d \times \Omega^c_d \times \Omega \times \Omega \times \Omega^c_{d^2}\times \Omega_{d^2} \times \Omega^c_{d^2} \times (\Omega_1 \times \Omega^c_1)^\N,
\end{equation*}
endowed with the product topology, becoming a Polish space. 
We write
\begin{equation*}\label{eq:canocical-process-enlarged}
\ov \X:=\big( \ov X,\ov B,\ov M^c,\ov M^d,\ov J, \ov C, \ov{[M]}, \ov{ \widetilde{C}},(\ov U^i,\ov V^i)_{i\in \N}\big)
\end{equation*}
for the canonical process $\ov \X_t(\overline{\omega})=\overline{\omega}(t)$,\, $\ov \omega \in \ov \Omega$. We endow $\ov \Omega$ with its Borel $\sigma$-field $\ov \cF$ and denote by $\ov \F =(\ov \F_t)_{0\leq t \leq T}$ the (raw) filtration generated by $\ov\X$. 
%

Denote by $\ov \fP^{ac}_{sem}$ the set of all probability measures on $(\ov \Omega, \ov \F)$ such that $\ov X$ is a $\ov \P$-$\ov \F$-semimartingale with absolutely continuous characteristics having canonical representation
\begin{equation}\label{eq:can-dec-X-enlarged}
\ov X-\ov X_0 = \ov M^{c} + \ov M^{d} + \ov B + \ov J \quad \ov  \P\mbox{-a.s., satisfying}
\end{equation}

\vspace*{0.15cm}
\noindent
$\bullet$ $\ov M^{c}_0 =\ov M^{d}_0 = \ov B_0=0 \ \,\ov \P$-a.s.\\
$\bullet$ $\ov B$ has $\ov \P$-a.s. finite variation paths.\\
$\bullet$ $\ov M^c$ is a continuous $\ov \P$-$\ov \F$-local martingale with quadratic covariation $\ov C$.\\
$\bullet$ $\ov M^d$ is a purely-discontinuous $\ov \P$-$\ov \F$-local martingale. \hspace*{\fill} $(\ast)$\\
$\bullet$ $\ov J= \sum_{0\leq s \leq \cdot} [\Delta \ov X_s -h(\Delta \ov X_s)] \ \ov \P$-a.s..\\
$\bullet$ $\ov{ \widetilde{C}}$ is the  modified second characteristic of $\ov X$ under $\ov\P$-$\ov \F$. (for definition, see e.g. \cite[p.79]{JacodShiryaev.03})\\
$\bullet$ $\ov{[M]}$ is the quadratic covariation of the local martingale $\ov M:= \ov M^c + \ov M^d$.\\
$\bullet$ $\ov \P$-a.s., each $\ov U^i$ coincides with  $\int_0^\cdot \int_{\R^d} g_i(x)\, \mu^{\ov X}(dx,dt)$ (writing $g_i(x)\ast \mu^{\ov X}$ for brevity) and each $\ov V^i$ coincides with  $g_i(x)\ast \nu^{\ov \P}$, where $\nu^{\ov \P}(dx,dt)$ denotes the  $\ov \P$-$\ov \F$-compensator of the measure $\mu^{\ov X}(dx,dt)$ associated to the jumps of $\ov X$, and $g_i$ denotes the $i$-th element in $\cC^+(\R^d)$. 

We also define a slightly bigger class $\ov \fP^{ac,w}_{sem}$ of probability measures which contains all laws $\ov \P$ on $(\ov \Omega, \ov \cF)$ satisfying the same conditions as above to be in $\ov \fP^{ac}_{sem}$, up to  $(\ast)$, where we instead impose the weaker condition on $\ov M^d$ to be  a $\ov \P$-$\ov \F$-local martingale (but not necessarily a purely discontinuous one). Moreover, for any $\ov \P \in \ov\fP^{ac}_{sem}$ denote by $(\ov b^{\ov\P},\ov c^{\ov\P}, \ov F^{\ov\P})$ the differential characteristics of $\ov X$ with respect to $\ov \P$-$\ov \F$.
Then, given  $\Theta \subseteq \R^d\times \S^d_+\times \cL$ and  $\Gamma_0\subseteq \fM_1(\R^d)$, define $\ov \fP^{ac}_{sem}(\Theta)$ in accord with the definition of the corresponding sets on the original space $\Omega$ by 
\begin{align*}
\ov \fP^{ac}_{sem}(\Theta)&:= \big\{ \ov \P \in \ov\fP^{ac}_{sem} \, \big| \, (\ov b^{\ov\P},\ov c^{\ov\P}, \ov F^{\ov\P}) \in \Theta \ \, \ov\P\times dt\mbox{-a.e.} \big\},\\
\ov \fP^{ac}_{sem}(\Theta)(\Gamma_0)&:=\big\{ \ov \P \in \ov\fP^{ac}_{sem}(\Theta) \, \big| \, \ov \P \circ \ov X_0^{-1} \in \Gamma_0 \big\}.
\end{align*}

Next, we introduce a function $\ov \varphi: \R^d\times \S^d_+\times \cL \to \R^d\times \S^d_+\times\R^\N$, which turns out to be useful for proving 
that the differential characteristics under a limit law $\ov \P_0$ of a sequence $(\ov \P_n)\subseteq \ov \fP^{ac}_{sem}(\Theta)$ are taking values in $\Theta$. 
Let $\{g_i \mid i \in \N\} = \cC^+(\R^d)$ and define first an additive, positive homogeneous function $\varphi : \cL \rightarrow \R^\N$  via
\begin{equation}\label{eq:map}
\varphi(F) := (\int_{\R^d} g_i(x)\, F(dx))_{i \in \N}.
\end{equation}
We deduce from $\cC^+(\R^d)$ being a law-determining class on $\cL$ that $\varphi$ is injective.
Now, define $\ov \varphi: \R^d\times\S^d_+ \times \cL \to \R^d\times\S^d_+ \times\R^\N$ by
\begin{equation}\label{eq:def-ov-varphi}
(b,c,F) \mapsto (b,c, \varphi (F))
\end{equation} Clearly, $\ov \varphi$ is additive, positive homogeneous and also a bijection onto its image. Mostly, we will use the function $\ov \varphi$ in the following way: Define the processes
\begin{align*}
\ov b_t:=\limsup_{n\to \infty}n(\ov B_t-\ov B_{(t-\frac{1}{n})\vee 0}), \ \ \, 
\ov c_t:=\limsup_{n\to \infty}n(\ov C_t-\ov C_{(t-\frac{1}{n})\vee 0}), \ \ \, t\in [0,T],
\end{align*}
as well as the sequence of processes $\ov v_t:=(\ov v^1_t, \ov v^2_t,\dots)$ by setting for each $i \in \N$
\begin{equation*}
\ov v^i_t:=\limsup_{n \to \infty} n \big(\ov V^i_t -\ov V^i_{(t-\frac{1}{n})\vee0}\big), \quad t\in [0,T].
\end{equation*}
Whenever $\ov \P \in \ov \fP^{ac}_{sem}$, we have
\begin{equation}\label{eq:char-in-Theta-help}
\big(\ov b_t^{\ov \P}, \ov c_t^{\ov \P}, (\int_{\R^d} g_i(x)\,\ov F_t^{\ov\P}(dx))_{i \in \N} \big)= \big(\ov b_t, \ov c_t, (\ov v^i_t)_{i \in \N}\big) \quad \ov \P\times dt\mbox{-a.s.}
\end{equation}
Therefore, for any given $\Theta \subseteq\R^d\times \S^d_+\times \cL$ and $\ov \P\in \ov\fP^{ac}_{sem}$ we have 
\begin{equation*}\label{eq:char-in-Theta-weak}
(\ov b^{\ov \P}, \ov c^{\ov \P}, \ov F^{\ov \P}) \in \Theta \ \ \ov \P\times dt\mbox{-a.s.} \ \ \Longleftrightarrow \ \ (\ov b, \ov c, \ov v) \in \ov \varphi(\Theta)  \ \ \ov \P\times dt\mbox{-a.s..}
\end{equation*}
In addition, note that 
as $\cC^+(\R^d)$ is a convergence-determining class for the weak convergence induced by the bounded continuous functions on $\R^d$ vanishing in a neighborhood of the origin, we see from \eqref{eq:char-in-Theta-help} that 
for any given $\Theta \subseteq\R^d\times \S^d_+\times \cL$ which is closed and $\ov \P\in \ov\fP^{ac}_{sem}$ we have 
\begin{equation}\label{eq:char-in-Theta}
(\ov b^{\ov \P}, \ov c^{\ov \P}, \ov F^{\ov \P}) \in \Theta \ \ \ov \P\times dt\mbox{-a.s.} \ \ \Longleftrightarrow \ \ (\ov b, \ov c, \ov v) \in \mbox{cl}(\ov \varphi(\Theta))  \ \ \ov \P\times dt\mbox{-a.s.,}
\end{equation}
where $\mbox{cl}(\ov \varphi(\Theta))$ denotes the closure of $\ov \varphi(\Theta) \subseteq \R^d\times\S^d_+ \times\R^\N$.
%
%

The main goal of Section~\ref{sec:enlarge} is to formulate and prove the following proposition, which is a version of Theorem~\ref{thm:purely-disc} on the enlarged space $\ov \Omega$.
Recall that under each $\ov \P \in \ov\fP^{ac}_{sem}$, $\ov M^d$ is a purely discontinuous local martingale.
\begin{proposition}\label{prop:purely-disc-enl}
	Let $\Theta \subseteq \R^d\times \S^d_+\times \cL$ satisfy Condition~(B) and let $(\ov \P_n)_{n \in \N}\subseteq \ov \fP^{ac}_{sem}(\Theta)$ be a sequence converging weakly to some law $\ov \P_0\in \fM_1(\ov \Omega)$. Then, the following hold true:
	\begin{enumerate}
		\item[1)] $\ov M^d$ is a $\ov \P_0$-$\ov \F$-martingale. 
		\item[2)] We have the following necessary and sufficient criterion for $\ov M^d$ being a purely discontinuous martingale under $\ov \P_0$-$\ov \F$:
		\begin{align*}\label{eq:criterion-purely-disc-enl}
		& \ \ov M^d \mbox{is a purely discontinuous $\ov \P_0$-$\ov \F$-martingale} \\
		\Longleftrightarrow & \  \lim\limits_{\delta\downarrow0}\limsup_{n \to \infty} \E^{\ov \P_n}\Big[\int_0^T\int_{\{|x|\leq \delta\}} |x|^2\, \ov F^{\ov \P_n}_t(dx)\,dt\Big] =0.
		\end{align*} 
	\end{enumerate}
\end{proposition}
%
%
As a consequence of Proposition~\ref{prop:purely-disc-enl}, we obtain a closedness criterion for $\ov \fP^{ac}_{sem}(\Theta)$
\begin{corollary}\label{co:closedness-P-Theta-enlarge}
	Let $\Theta \subseteq \R^d\times \S^d_+\times \cL$ satisfy Condition~(B) and Condition~(J). Consider a sequence $(\ov \P_n)\subseteq \ov \fP^{ac}_{sem}(\Theta)$ converging weakly to some $\ov \P_0 \in \fM_1(\ov \Omega)$. The following hold:
	\begin{enumerate}
		\item[1)] We have $\ov \P_0 \in \ov \fP^{ac}_{sem}$.
		\item[2)] If in addition, $\Theta$ is closed and convex, then $\ov \P_0 \in \ov \fP^{ac}_{sem}(\Theta)$. In particular, $\ov\fP^{ac}_{sem}(\Theta)$ is closed.
	\end{enumerate} 
\end{corollary}
%
One of the key technique to prove Proposition~\ref{prop:purely-disc-enl} and Corollary~\ref{co:closedness-P-Theta-enlarge} is the Skorokhod representation theorem \cite{Skorohod.56} (see also \cite{Jakubowski.97}), which states the following.
\begin{theorem}[Skorokhod representation theorem]
	\label{thm:skorohod-repr-thm}
	There exists a sequence of $\ov\Omega$-valued random variables $(z^n)_{n \in\N_0}$ defined on $([0,1],\cB([0,1]),\lambda)$, where $\lambda$ denotes the Lebesgue measure, such that for each $n \in \N_0$
	, $\ov \P_n=\lambda\circ (z^n)^{-1}$ and $(z^n)$ converges to $z^0$ with respect to the product topology on $\ov \Omega.$
\end{theorem}
Denote by $z^{n,\ov X}:=\ov X\circ z^n \in \Omega$, and define in the same way $z^{n,\ov B},\dots,z^{n,\ov C}$ as well as $z^{n,\ov U^i}, z^{n,\ov V^i}$ for each $i \in \N$. The convergence of $(z^n)$ to $z^0$ on $\ov \Omega$ implies the convergence of each sequence $(z^{n,\ov X}), \dots,  (z^{n,\ov V^i})$ to  $z^{0,\ov X}, \dots,  z^{0,\ov V^i}$, respectively, on the corresponding space.

The proof of Proposition~\ref{prop:purely-disc-enl} and Corollary~\ref{co:closedness-P-Theta-enlarge} are divided into several lemmas, provided in the following subsections. For standard notation appearing in the theory of weak convergence of processes, we refer to \cite{JacodShiryaev.03}.

%
%
%
%
\subsection{Semimartingale property of $\ov X$ under the limit law $\ov \P_0$}\label{subsec:semimart-enlarge}
The main objective of this subsection is to prove that  any limit law $\ov \P_0$ of a sequence $(\ov \P_n)_{n \in \N}\subseteq \ov \fP^{ac}_{sem}(\Theta)$ with $\Theta$ satisfying Condition~(B) is an element of $\ov\fP^{ac,w}_{sem}$. In particular, $\ov X$ has no fixed time of discontinuity. 
\begin{proposition}\label{prop:X-semim-cont-char}
	Let $\Theta \subseteq \R^d\times \S^d_+\times \cL$ satisfy Condition~(B) and 
	let $(\ov \P_n)_{n \in \N}\subseteq \ov \fP^{ac}_{sem}(\Theta)$ be a sequence converging weakly to some law $\ov \P_0$. Then, $\ov \P_0 \in \ov\fP^{ac,w}_{sem}$. In addition, we have
	
	\vspace*{0.15cm}
	\noindent
	(i) $\E^{\ov \P_0}\Big[\int_0^T \big[|\ov b^{\ov \P_0}_s|  + |\ov c^{\ov \P_0}_s| +\int_{\R^d} |h(x)|^2\,\ov F^{\ov \P_0}_s(dx)\big]\, ds\Big]<\infty$;\\
	(ii) Both $\ov M^c$ and $\ov M^d$ are $\ov \P_0$-$\ov \F$-martingales.
\end{proposition}
This is the first step towards the proof of Proposition~\ref{prop:purely-disc-enl} and Corollary~\ref{co:closedness-P-Theta-enlarge}. The proof of Proposition~\ref{prop:X-semim-cont-char}
is divided into several lemmas. We will frequently use the fact that for any continuous function $f:\R^d \to \R^m$ vanishing on a neighborhood of 0, the map
\begin{equation}\label{eq:jumps-away-0-cont}
I^f:\D([0,T];\R^d) \to \D([0,T];\R^m), \quad \alpha \mapsto I^f(\alpha):=\sum_{0\leq s \leq \cdot} f(\Delta \alpha_s)
\end{equation}
is continuous, see \cite[Corollary~VI.2.8, p.340]{JacodShiryaev.03}. 
%
%
\begin{lemma}\label{le:canonical-decomp-limit-enlarge}
	Let $(\ov \P_n)_{n \in \N}\subseteq \ov \fP^{ac}_{sem}$ be a sequence converging weakly to some law $\ov \P_0\in \fM_1(\ov \Omega)$. Then we have:\\
	(i) $\ov J= \sum_{0\leq s \leq \cdot} [\Delta \ov X_s-h(\Delta \ov X_s)] \ \, \ov \P_0$-a.s.;\\
	(ii) For each $i \in \N$, we have $\ov U^i= \sum_{0\leq s \leq \cdot} g_i(\Delta X_s) \ \, \ov \P_0$-a.s.;\\
	(iii) $\ov X= \ov X_0 + \ov B + \ov M^c + \ov M^d + \ov J \ \, \ov \P_0$-a.s. 
\end{lemma}
\begin{proof}
	By the Skorokhod representation theorem, both sequences $(z^{n,\ov X})$ and $(z^{n,\ov J})$ converge pointwise to $z^{0,\ov X}$ and $z^{0,\ov J}$, respectively, on the Skorokhod space.  Now, as $x-h(x)$
	is a continuous function vanishing in a neighborhood of zero, the function $I^{x-h(x)}:\Omega \to \Omega$ defined in \eqref{eq:jumps-away-0-cont} is continuous, hence
	\begin{equation*}
	\lim\limits_{n \to \infty} I^{x-h(x)}(z^{n,\ov X})=I^{x-h(x)}(z^{0,\ov X}).
	\end{equation*}
	On the other hand, as $\ov J= \sum_{0\leq s \leq \cdot} [\Delta \ov X_s-h(\Delta \ov X_s)] \ \, \ov \P_n$-a.s. for each $n$, we deduce from the relation $\overline{\P}_n = \lambda \circ ({z^n})^{-1}$  that
	\begin{equation*}
	\lambda[z^{n,\ov J} = I^{x-h(x)}(z^{n,\ov X})] = 1.
	\end{equation*}
	Therefore, we conclude $z^{0,\ov J}= I^{x-h(x)}(z^{0,\ov X}) \ \lambda$-a.s., which in turn implies (i).
	
	Next,  as each $g_i(x)$ is a continuous function vanishing in a neighborhood of the origin, the same argument as in (i), but with  $g_i(x)$ and $\ov U^i$ yields (ii).
	
	To prove (iii), we notice the existence of a different topology on the Skorokhod space $\Omega$, the so-called $S$-topology, which is weaker than the usual Skorokhod $J_1$-topology, but has the property that the addition $(\alpha,\beta) \in \Omega \times \Omega \mapsto \alpha + \beta \in \Omega$ is sequentially continuous, see \cite[Theorem~2.13]{Jakubowski.97b}. Therefore, we conclude that
	\begin{align*}
	& \ \lim_{n \rightarrow \infty} \big(z^{n,\ov X} - z^{n,\ov X}_0 - z^{n,\ov B} - z^{n,\ov M^c} - z^{n,\ov  M^d} - z^{n,\ov J}\big)\\
	= & \  z^{0,\ov X} - z^{0,\ov X}_0 - z^{0,\ov B} - z^{0,\ov M^c} - z^{0,\ov M^d} - z^{0,\ov J}
	\end{align*}
	pointwise in the $S$-topology. On the other hand,
	$\ov X= \ov X_0 + \ov B + \ov M^c + \ov M^d + \ov J \ \, \ov \P_n$-a.s., hence $z^{n,\ov X} - z^{n,\ov X}_0 - z^{n,\ov B} - z^{n,\ov M^c} - z^{n,\ov  M^d} - z^{n,\ov J}=0 \ \lambda$-a.s. for each $n$. Therefore, we conclude that also $z^{0,\ov X} - z^{0,\ov X}_0 - z^{0,\ov B} - z^{0,\ov M^c} - z^{0,\ov M^d} - z^{0,\ov J}=0 \ \lambda$-a.s., which indeed gives us (iii).
\end{proof}
%
%
%
%
For any constant $K>0$, we denote by $\mbox{Lip}_K$ the set of all Lipschitz-continuous functions on $[0,T]$ with Lipschitz constant $K$. 
Recall that Condition~(B) provides the finiteness of 
\begin{equation*}
\cK:=\sup_{(b,c,F)\in \Theta} \Big\{|b|+|c| + \int_{\R^d} |x|^2 \wedge |x|\, F(dx)\Big\}<\infty.
\end{equation*}
%
%
\begin{lemma}\label{le:B-enlarge}
	Let $\Theta \subseteq \R^d\times \S^d_+\times \cL$ satisfy Condition~(B) and 
	let $(\ov \P_n)_{n \in \N}\subseteq \ov \fP^{ac}_{sem}(\Theta)$ be a sequence converging weakly to some law $\ov \P_0\in \fM_1(\ov \Omega)$. Then $\ov B$ is absolutely continuous $\ov \P_0$-a.s. satisfying $\E^{\ov \P_0}[\int_0^T |\ov b_s| \, ds]<\infty$. In particular, it is $\ov \F$-predictable of $\ov \P_0$-integrable variation. Moreover, the same holds true for $\ov C$ and $\ov{\widetilde{C}}$.
\end{lemma}
\begin{proof}
	As $\Theta \subseteq \R^d\times \S^d_+\times \cL$ satisfy Condition~(B), we have $\ov \P_n\mbox{-a.s.}$ that for any $0\leq s\leq t\leq T$ 
	\begin{equation*}
	|\ov B_t-\ov B_s| \leq \int_s^t |\ov b_r|\,dr \leq \cK |t-s|,
	\end{equation*}
	hence $\ov B \in \mbox{Lip}_\cK \ \ov \P_n$-a.s..We can conclude by 
	\begin{align*}
	1&= \overline{\P}_n\big[|\ov B_t - \ov B_s| \leq \cK|t - s|, \forall \text{ }t,s \in [0,T] \cap \Q\big]\\
	&= \lambda\big[|z^{n,\ov B}_t - z^{n,\ov B}_s| \leq \cK|t - s|, \forall \text{ }t,s \in [0,T] \cap \Q \big]
	\end{align*}
	that $z^{n,\ov B} \in \text{Lip}_\cK$, $\lambda$-a.s. for each $n$. As $z^{n,\ov B}$ converges to $z^{0,\ov B}$ pointwise in $\Omega^c_d$, we have $z^{0,\ov B} \in \mbox{Lip}_\cK$  $\lambda$-a.s., which means that $\ov B \in \mbox{Lip}_\cK \ \ov \P_0$-a.s.. Hence, by Rademacher's theorem, $ \ov B$ has $\ov \P_0$-a.s. absolutely continuous trajectories, and $\E^{\ov \P_0}[\int_0^T |\ov b_s| \, ds]\leq \cK T$. The finite variation property of $\ov B$ follows. Note that by continuity, $\ov B$ is $\overline{\F}$-predictable. 
	
	Next, the same arguments as above yield also the $\ov \P_0$-a.s. absolute continuity of $\ov C$ with the same integrability property. Moreover, as by \cite[II.2.18, p.79]{JacodShiryaev.03}, each component satisfies
	\begin{equation*}
	\ov{\widetilde{C}}^{ij}= \ov C^{ij} + \int_0^\cdot \int_{\R^d} h^i(x)h^j(x)\,\ov F_t^{\ov \P_n}(dx)\,dt \ \ \ov \P_n\mbox{-a.s. for each $n$,}
	\end{equation*}the same also hold for $\ov{\widetilde{C}}$.
\end{proof}
%
%
%
%
\begin{lemma}\label{le:M-enlarge}
	Let $\Theta \subseteq \R^d\times \S^d_+\times \cL$ satisfy Condition~(B) and 
	let $(\ov \P_n)_{n \in \N}\subseteq \ov \fP^{ac}_{sem}(\Theta)$ be a sequence converging weakly to some law $\ov \P_0\in \fM_1(\ov \Omega)$. Then:\\
	(i) both $\ov M^c$ and $\ov M^d$ are $\ov \P_0$-$\ov \F$-martingales;\\
	(ii) For each $i \in \N$, we have $\ov V^i=g_i(x) \ast \nu^{\ov \P_0} \ \ \ov \P_0$-a.s..
\end{lemma}
\begin{proof}
	To see that $\ov M^d$ is $\ov \P_0$-$\ov \F$-martingale  we first claim that for each $t\in [0,T]$, the sequence $(\ov M^d_t\,|\,\ov \P_n)$ is uniformly integrable in the sense that
	\begin{equation*}
	\lim\limits_{a\to \infty} \sup_{n \in \N} \E^{\ov \P_n}\big[\ov M^d_t\, \mathbbm{1}_{\{|\ov M^d_t|\geq a\}}\big]=0. 
	\end{equation*}
	To see this, observe that due to $\Theta$ satisfying Condition~(B),
	\begin{equation*}
	\E^{\ov \P_n}\big[ |\ov M^d_t|^2\big]
	=\E^{\ov \P_n}\big[|[\ov M^d]_t|\big]
	\leq\E^{\ov \P_n}\Big[\int_0^t \int_{\R^d} |h(x)|^2 \, \ov F^{\ov \P_n}_s(dx)\,ds\Big] \leq K \cK T
	\end{equation*}
	uniformly for all $n$, hence the result follows from the de la Vall\'ee-Poussin theorem.
	Then in terms of the Skorokhod representation, this means that for all $t \in [0,T]$, the sequence of random variables $(z^{n,\ov M^d}_t)_{n \in \N}$ defined on $([0,1], \cB([0,1]),\lambda)$ is uniformly integrable.
	
	Next, note that the canonical process $\ov \X$ on $\ov \Omega$ 
	with state space $S:=\R^{5d+3d^2}\times (\R^2)^{\N}$ 
	is right-continuous and the set $\mathfrak{J}:=\{t \in [0,T] \mid \overline{\P}_0 [ \Delta \ov \X_t \neq 0] >0 \}$ is at most countable. Thus, the set $\mathfrak{D}:= [0,T] \setminus \mathfrak{J}$ is  dense in $[0,T]$, and for all $t \in \fD$  the sequence of laws $(\ov \P_n \circ \ov \X^{-1}_t)_{n \in \N}$ converges weakly to  $\ov\P_0 \circ \ov \X_t^{-1}$. In terms of the Skorokhod representation, the latter means that for all $t \in \fD$, $\lim_{n \rightarrow \infty}z^n_t = z^0_t$ in $S$ $\lambda$-a.s.  
	Thus, for any $s,t \in \fD$ with $s < t$, for any finite partition $0 \leq s_1 < ... < s_m\leq s$ taking values in $\fD$, and for any $f_j \in C_b(S)$, $j =1,...,m$,
	we have  by the uniform integrability of the sequence $(z^{n,\ov M^d}_t - z^{n,\ov M^d}_s)_{n \in \N}$ that
	\begin{align*}
	0&=\lim_{n \rightarrow \infty} \E^{\overline{\P}_n}\Big[\prod_{j = 1}^m f_j(\ov \X_{s_j})(\ov M^d_t - \ov M^d_s)\Big]\\
	&= \lim_{n \rightarrow \infty} \E^{\lambda}\Big[\prod_{j = 1}^m f_j(z^{n}_{s_j})(z^{n,\ov M^d}_t - z^{n,\ov M^d}_s)\Big] \\
	&=\E^\lambda\Big[\prod_{j = 1}^m f_j(z^{0}_{s_j})(z^{0,\ov M^d}_t - z^{0,\ov M^d}_s)\Big] \\
	&=\E^{\overline{\P}_0}\Big[\prod_{j = 1}^m f_j(\ov \X_{s_j})(\ov M^d_t - \ov M^d_s)\Big].
	\end{align*}
	As the filtration $\overline{\F}$ is generated by the canonical process $(\ov \X_t)_{t \in [0,T]}$, we can deduce from the monotone class theorem that $(\ov M^d_t)_{t \in \fD}$ is an $\overline{\F}$-martingale under $\overline{\P}_0$, and as $\fD$ is dense in $[0,T]$, the whole process $(\ov M^d_t)_{t \in [0,T]}$ is also an $\overline{\F}$-martingale under $\overline{\P}_0$. Finally, we observe that the same argument also works for $\ov M^c$, hence (i) holds.
	
	Now, fix any $i \in \N$. By continuity, $\ov V^i$ is $\F$-predictable. Moreover, as $g_i$ is a bounded continuous function vanishing in a neighborhood of the origin, there exists $0<\delta_i\leq 1$ such that for each $n$
	\begin{equation*}
	\ov V^i= \int_0^\cdot \int_{\R^d} g_i(x) \, \ov F^{\ov \P_n}_s(dx)\,ds
	= \int_0^\cdot \int_{\{|x|\geq \delta_i\}} g_i(x) \, \ov F^{\ov \P_n}_s(dx)\,ds
	\ \ \ov \P_n\mbox{-a.s.}.
	\end{equation*}
	Thus, we can argue as in Lemma~\ref{le:B-enlarge} to conclude that for $c_i:=\sup_x |g_i(x)|$ and $\widetilde{\cK}_i:=\nicefrac{c_i \cK}{\delta_i^2}$, we have $\ov V^{i} \in \mbox{Lip}_{\widetilde{\cK}_i} \ \ov \P_n$-a.s. for all $n$ and then also $\ov V^{i} \in \mbox{Lip}_{\widetilde{\cK}_i} \ \ov \P_0$-a.s.. Therefore, by applying Lemma~\ref{le:canonical-decomp-limit-enlarge} (ii), it remains to show that $\ov W^i:=\ov U^i- \ov V^i$ is a $\ov \P_0$-$\ov \F$-(local)-martingale. But this follows directly when applying the same arguments as in (i), but with respect to $\ov W^i$.
\end{proof}
\begin{remark}\label{rem:M-enlarge-P-0-NOT-PURE-DISCONT} Whereas $\ov M^c$ is a continuous $\overline{\F}$-martingale under $\overline{\P}_0$, it is not necessarily true that $\ov M^d$ is a purely discontinuous $\ov \P_0$-$\ov \F$-martingale. We recall that our goal in this section is to provide a necessary and sufficient condition for $\ov M^d$  to be a purely discontinuous martingale under $\ov \P_0$, see Proposition~\ref{prop:purely-disc-enl}.
\end{remark}
%
%
%
%
%
%
%
%
\begin{lemma}
	\label{le:compensator-enlarge}
	Let $\Theta \subseteq \R^d\times \S^d_+\times \cL$ satisfy Condition~(B) and 
	let $(\ov \P_n)_{n \in \N}\subseteq \ov \fP^{ac}_{sem}(\Theta)$ be a sequence converging weakly to some law $\ov \P_0\in \fM_1(\ov \Omega)$. Then the $\ov \P_0$-$\ov \F$-compensator $\nu^{\ov \P_0}(dx,dt)$ of the measure $\mu^{\ov X}(dx,dt)$ associated to the jumps of $\ov X$ is absolutely continuous in the sense that $\nu^{\ov \P_0}(dx,dt)$ satisfies a disintegration $\nu^{\ov \P_0}(dx,dt)=\ov F^{\ov \P_0}_t
	(dx)\,dt \ \, \ov \P_0$-a.s..
\end{lemma}
\begin{proof}
	A careful inspection of the proof of \cite[Proposition~II.2.9, p.77]{JacodShiryaev.03} shows that the absolute continuity of $\nu^{\overline{\P}_0}$ is equivalent to the absolute continuity of the process $\ov A:= (1 \wedge |x|^2) * \nu^{\overline{\P}_0}$. 
	
	Recall the countable family $ \cC^+(\R^d):=\{g_i \mid i \in \N\}$ introduced at the beginning of Subsection~\ref{subsec:enlarg}. In particular, for each $m \in \N$,  $g_{2m}$ is a continuous function on $\R^d$  satisfying for  all $x \in \R^d$ that $0\leq g_{2m}(x)\leq |x|^2 \wedge 1 $ and that
	$$
	g_{2m}(x) = \begin{cases}
	0       & \quad \text{if } |x| \leq \frac{1}{2m};\\
	|x|^2 \wedge 1  & \quad \text{if } |x| > \frac{1}{m}.\\
	\end{cases}
	$$
	%
	Then by definition, we have for each $n$ that
	\begin{equation*}
	\ov V^{2m} = \int_0^\cdot \int_{\R^d} g_{2m}(x)\, \ov F^{\ov \P_n}_s(dx)\,ds \quad \ov \P_n\mbox{-a.s.}
	\end{equation*}
	Thus for each $m$, by the same argument as in Lemma~\ref{le:M-enlarge} (ii) we obtain that 
	$\ov V^{2m} \in \mbox{Lip}_\cK \ \ov \P_0$-a.s. Therefore,  Fatou's lemma and Lemma~\ref{le:M-enlarge} (ii) yields
	\begin{equation*}
	\E^{\ov \P_0}[\ov A_T] 
	\leq 
	\liminf_{m \to \infty} 
	\E^{\ov \P_0}[\ov V^{2m}_T]\leq \cK T,
	\end{equation*}
	in particular, $\ov A_t<\infty \ \ov \P_0$-a.s. for all $t\in [0,T]$. Thus, by dominated convergence, $\ov A_t=\lim\limits_{m \to \infty} \ov V^{2m}_t \ \ov \P_0$-a.s. for each $t \in [0,T]$, hence also 
	\begin{equation*}
	\ov A= \sup_{m \in \N} \ov V^{2m} \ \ \ov \P_0\mbox{-a.s.}
	\end{equation*}
	As $(\ov V^{2m})_{m \in \N} \subseteq \mbox{Lip}_\cK$, we obtain that $\ov A \in \mbox{Lip}_\cK$, which implies its absolute continuity $\ov \P_0$-a.s..
\end{proof}
%
%
\begin{lemma}\label{le:quad-var-second-modified-enlarge}
	Let $\Theta \subseteq \R^d\times \S^d_+\times \cL$ satisfy Condition~(B) and 
	let $(\ov \P_n)_{n \in \N}\subseteq \ov \fP^{ac}_{sem}(\Theta)$ be a sequence converging weakly to some law $\ov \P_0\in \fM_1(\ov \Omega)$. Then, the following hold true:\\
	(i) $\ov C$ is the quadratic covariation of $\ov M^c$ under $\ov \P_0$-$\ov \F$.\\
	(ii) $\ov{[M]}$ is the quadratic covariation of $\ov M:=\ov M^c + \ov M^d$ under $\ov \P_0$-$\ov \F$.\\ 
	(iii) $\ov{\widetilde{C}}$ is the modified second characteristic of $\ov X$ under  $\ov \P_0$-$\ov \F$.
\end{lemma}
\begin{proof}
	To obtain (i), denote by $\langle \ov M^c\rangle^{\ov\P_0}$ the quadratic variation of $\ov M^c$ under $\ov \P_0$-$\ov \F$. As $\ov C$ is the quadratic variation of $\ov M^c$ under each $\ov \P_n$, we deduce from \cite[Corollary~VI.6.29, p.385]{JacodShiryaev.03} that $(\ov M^c, \ov C) =(\ov M^c,\langle \ov M^c\rangle^{\ov\P_0})$ in law under $\ov \P_0$. This implies that also $(\ov M^c,\ov M^c, \ov C) =(\ov M^c,\ov M^c,\langle \ov M^c\rangle^{\ov\P_0})$ in law under $\ov \P_0$. By applying \cite[Theorem VI.6.22, p.383]{JacodShiryaev.03},
	we see that $(\ov M^c,\ov M^c, \ov C, \int \ov M^c\,d\ov M^c)=(\ov M^c,\ov M^c, \langle \ov M^c\rangle^{\ov\P_0}, \int \ov M^c\,d\ov M^c)$ in law under $\ov \P_0$. As a consequence, we obtain componentwise for any $1\leq i,j\leq d$ that
	\begin{align*}
	& \ \ov C^{ij} -\ov M^{c,i} \ov M^{c,j} -\int \ov M^{c,i}\,d\ov M^{c,j} - \int \ov M^{c,j}\,d\ov M^{c,i}\\
	= & \
	\langle \ov M^c\rangle^{\ov\P_0, ij} - \ov M^{c,i} \ov M^{c,j} -\int \ov M^{c,i}\,d\ov M^{c,j} - \int \ov M^{c,j}\,d\ov M^{c,i} \\
	= & \ 0
	\end{align*}
	in law under $\ov \P_0$, hence by definition of the quadratic variation,  $\ov C$ coincide with $\langle \ov M^c\rangle^{\ov\P_0}$.
	
	The proof of (ii) is identical to the one for (i), hence it remains to prove (iii).
	
	To obtain (iii), we know from Lemma~\ref{le:B-enlarge} that $\ov{\widetilde{C}}$ is of integrable variation under $\ov \P_0$. Moreover, we know from (ii) that $\ov{[M]}$ is the quadratic variation of $M:=\ov M^c + \ov M^d$ under $\ov \P_0$. Hence by definition of the  modified second characteristic, it remains to show that $\ov{[M]}-\ov{\widetilde{C}}$ is a $\ov\P_0$-$\ov \F$-martingale, which we obtain by applying the same arguments as in Lemma~\ref{le:M-enlarge} once we derived that the sequence $(\ov{[M]}_t-\ov{\widetilde{C}}_t\,|\,\ov \P_n)$ is uniformly integrable for each $t\in [0,T]$. To see this, 
	observe that due to Condition~(B) 
	\begin{align*}
	\E^{\ov \P_n}\big[|\ov{[M]}_t-\ov{\widetilde{C}}_t|^2\big]
	=\E^{\ov \P_n}\big[|[\ov{[M]}-\ov{\widetilde{C}}]_t|\big]
	&\leq\E^{\ov \P_n}\big[\sum_{0\leq s\leq t} |\Delta \ov{[M]}_s|^2\big]\\
	&=\E^{\ov \P_n}\big[\int_0^t \int_{\R^d} |h(x)|^4\,\ov F^{\ov \P_n}_s(dx)\,ds\big]\\
	&\leq K \cK t
	\end{align*}
	uniformly for all $n$, which implies the desired uniform integrability.
\end{proof}
\begin{proof}[Proof of Proposition~\ref{prop:X-semim-cont-char}] This is simply 
	the summary of Lemma~\ref{le:canonical-decomp-limit-enlarge}--Lemma~\ref{le:quad-var-second-modified-enlarge},
	noticing that the absolute continuity of the second characteristic $\ov C^{\ov \P_0}$ (which is not necessarily equal to $\ov C$!) of $\ov X$ under $\ov\P_0$-$\ov\F$ follows from the absolute continuity of the  modified second characteristic $\ov{\widetilde{C}}$ and the third characteristic.
\end{proof}
%
%
\subsection{Proof of Proposition~\ref{prop:purely-disc-enl} and Corollary~\ref{co:closedness-P-Theta-enlarge}}\label{subsec:proof-purely-disc-and-closedness-enlarge}
The goal of this subsection is to prove Proposition~\ref{prop:purely-disc-enl} and Corollary~\ref{co:closedness-P-Theta-enlarge}. We divide the proof into several lemmas. To keep the notation short, we provide the proof of Proposition~\ref{prop:purely-disc-enl} in the one-dimensional case, i.e. $\Omega=\D([0,T],\R)$. The extension to the multi-dimensional case is straightforward, which we explain right after the proof in Remark~\ref{rem:1-d-to-Multi-d}. 
\begin{lemma}\label{le:quad-var-conv-Md-enlarge}
	Let $\Theta \subseteq \R^d\times \S^d_+\times \cL$ satisfy Condition~(B) and 
	let $(\ov \P_n)_{n \in \N}\subseteq \ov \fP^{ac}_{sem}(\Theta)$ be a sequence converging weakly to some law $\ov \P_0\in \fM_1(\ov \Omega)$. Then, 
	the sequence $\ov \P_n\circ (\ov X,\ov M^d, [\ov M^d])^{-1}$ converges weakly to 
	$\ov \P_0\circ (\ov X,\ov M^d, [ \ov M^d])^{-1}$
\end{lemma}
\begin{proof}
	First, recall that by Proposition~\ref{prop:X-semim-cont-char} we know that $\ov M^d$ is a $\ov \P_0$-$\ov\F$-martingale, hence $[\ov M^d]$ is well-defined also under the limit law $\ov \P_0$. Moreover, by \cite[Corollary~VI.6.29, p.385]{JacodShiryaev.03}, the sequence $\ov \P_n \circ (\ov M^d, [\ov M^d])^{-1}$ converges weakly to the law $\ov\P_0 \circ (\ov M^d, [\ov M^d])^{-1}$. From the Skorokhod representation theorem, we obtain  $\D([0,1],\R^2)$-valued random variables $(( z^{n,\ov M^d},y^n))_{n \in \N}$, $( z^{0,\ov M^d},y^0)$ on $([0,1], \cB([0,1]), \lambda)$ such that $(z^{n,\ov M^d},y^n)$ converges to $(z^{0,\ov M^d},y^0)$ in $\D([0,1],\R^2)$ pointwisely and for each $n \geq 0$
	\begin{equation*}
	\lambda \circ (z^{n,\ov M^d},y^n)^{-1} = \overline{\P}_n \circ (\ov M^d,[\ov M^d])^{-1}.
	\end{equation*}

	Now, as
	$(z^{n,\ov X})$
	converges pointwise to $z^{0,\ov X}$,
	there is for each
	$t$ a sequence $(t_n)$ converging to $t$ such that $\Delta z^{n,\ov X}_{t_n}$ converges to $\Delta z^{0,\ov X}_{t}$, see \cite[Proposition~VI.2.1, p.337]{JacodShiryaev.03}. By continuity of the truncation function $h(\cdot)$, 
	also
	$(h(\Delta z^{n,\ov X}_{t_n}), h^2(\Delta z^{n,\ov X}_{t_n}))$  converges to $(h(\Delta z^{0,\ov X}_{t}), h^2(\Delta z^{0,\ov X}_{t}))$. Moreover, note that $(h(\Delta \ov X_{t_n}), h^2(\Delta \ov X_{t_n}))=(\Delta \ov M^d_{t_n}, \Delta [\ov M^d]_{t_n}) \ \ov \P_n$-a.s., therefore we have that $(h(\Delta z^{n,\ov X}_{t_n}), h^2(\Delta z^{n,\ov X}_{t_n}))= (\Delta z^{n,\ov M^d}_{t_n}, \Delta y^n_{t_n}) \ \lambda$-a.s. for each $n$, and from Proposition~\ref{prop:X-semim-cont-char} we know that $(h(\Delta \ov X_{t}),h^2(\Delta \ov X_{t})) = (\Delta \ov M^d_{t}, \Delta [\ov M^d]_{t})\ \ov \P_0$-a.s., hence $(h(\Delta z^{0,\ov X}_{t}),h^2(\Delta z^{0,\ov X}_{t}))= (\Delta z^{0,\ov M^d}_{t}, \Delta y^0_t)$. Therefore,  we obtain that $(\Delta z^{n,\ov M^d}_{t_n}, \Delta y^n_{t_n})$ converges to $(\Delta z^{0,\ov M^d}_{t}, \Delta y^0_t) \ \lambda$-a.s.
	We can then conclude by applying \cite[Proposition~VI.2.2(b), p.338]{JacodShiryaev.03}   that
	\begin{equation*}
	(z^{n,\ov X}, z^{n,\ov M^d},y^n)  \mbox{ converges to } (z^{0,\ov X},z^{0,\ov M^d}, y^0) \mbox{ in  } \D([0,T], \R^3) \ \ \lambda\mbox{-a.s.,}
	\end{equation*}
	which implies the desired result. 
\end{proof}
%
%
From the previous Lemma, there exists a sequence of $\D([0,T],\R)$-valued random variables $y^n$ converging pointwise to $y^0$, satisfying
\begin{equation*}
\ov \P_n \circ [\ov M^d]^{-1}=\lambda \circ (y^n)^{-1} \mbox{ for each } n \in \N_0. 
\end{equation*}

For the rest of this section, we fix for any $\delta>0$ 
a continuous function $\psi_\delta: \R \to [0,1]$ such that  $(\psi_\delta)_{\delta}$ is a nonincreasing sequence of functions when $\delta$ tends to zero satisfying
\begin{equation*}
\psi_\delta(x) = \begin{cases}
1       & \quad \text{if } |x| \leq \delta/2;\\
0  & \quad \text{if } |x| > \delta.\\
\end{cases}
\end{equation*}
Moreover, for each $n\geq 0$, we introduce
\begin{align*}
\beta^{n,\delta} &:= \sum_{s \leq \cdot}\big(1 - \psi_\delta(\Delta z^{n,\ov X}_s)\big)\,h^2(\Delta z^{n,\ov X}_s),\\
\gamma^{n,\delta} &:= \sum_{s \leq \cdot}\psi_\delta(\Delta z^{n, \ov X}_s)\,h^2(\Delta z^{n,\ov X}_s).
\end{align*}
Assume that the conditions of Proposition~\ref{prop:purely-disc-enl} hold true. Under each $\ov \P_n$, $\ov M^d$ is the purely discontinuous  martingale part in the canonical representation of the $\ov \P_n$-$\ov \F$-semimartingale $\ov X$, satisfying $\Delta \ov M^d_t=h(\Delta \ov X_t) \ \ov \P_n$-a.s.. Hence, using Lemma~\ref{le:quad-var-conv-Md-enlarge} we have for each $n\geq 1$ 
\begin{equation}\label{eq:quadratic-var-decomp}
y^n= \sum_{0\leq s\leq \cdot} h^2(\Delta z^{n,\ov X}_s)= \beta^{n,\delta} + \gamma^{n,\delta} \quad \lambda\mbox{-a.s.}
\end{equation}
However, $\ov M^d$ is not necessarily a purely discontinuous martingale under $\ov \P_0$. In fact, as by Proposition~\ref{prop:X-semim-cont-char}\, $\Delta \ov M^d=h(\Delta \ov X) \ \ov \P_0$-a.s., $\ov M^d$ is a purely discontinuous martingale if and only if  equation \eqref{eq:quadratic-var-decomp} also holds true  for $n=0$.
%
%
\begin{lemma}\label{le:limit-beta-n-delta-in-n}
	For any $\delta>0$, we have:\\
	(i) $\beta^{n,\delta}$ converges pointwise to $\beta^{0,\delta}$ in $\D([0,T],\R)$;\\
	(ii) $y^n-\beta^{n,\delta}$  converges pointwise to $y^0-\beta^{0,\delta}$ in $\D([0,T],\R)$.
\end{lemma}
\begin{proof}
	Notice that $(1-\psi_\delta(x))h^2(x)$ is a continuous function vanishing in a neighborhood of the origin, hence by \eqref{eq:jumps-away-0-cont}, the convergence of $(z^{n,\ov X})$ to $z^{0,\ov X}$ implies (i).
	
	To see (ii), we can use (i) and argue as in Lemma~\ref{le:quad-var-conv-Md-enlarge} to see that for each $t$ there exists a sequence $(t_n)$ converging to $t$ such that $(\Delta y^n_{t_n})$ converges to $\Delta y^0_t$ and $(\Delta \beta^{n,\delta}_{t_n})$ converges to $\Delta \beta^{0,\delta}_t$. Hence (ii) now follows directly from \cite[Proposition~VI.2.2(a), p.338]{JacodShiryaev.03}.
\end{proof}
From now on, we fix a sequence $(\delta_m)\subseteq (0,\infty)$ which converges to 0.
\begin{lemma}\label{le:limit-beta-0-delta-in-delta}
	The following hold true.
	
	\vspace*{0.15cm}
	\noindent
	(i) $\lim\limits_{m \to \infty}\beta^{0,\delta_m}=\sum_{0\leq s \leq \cdot} h^2(\Delta z^{0,\ov X}_s)$ in $\D([0,T],\R)$.\\
	(ii)  $y^{0,cont}:=\lim\limits_{m \to \infty}y^0-\beta^{0,\delta_m}=y^0-\sum_{0\leq s \leq \cdot} h^2(\Delta z^{0,\ov X}_s)$ in $\D([0,T],\R)$.
\end{lemma}
\begin{proof}
	To obtain (i), as each  $\beta^{0,\delta_m}$ and $\sum_{0\leq s \leq \cdot} h^2(\Delta z^{0,\ov X}_s)$ are nonnegative, nondecreasing functions starting in the origin, it suffices to verify that the conditions of \cite[Theorem~VI.2.15, p.342]{JacodShiryaev.03} are satisfied. To that end, we want to show that for $t \in [0,T]$
	\begin{equation*}
	\lim_{m \to \infty}\beta^{0,\delta_m}_t = \sum_{s \leq t}h^2(\Delta z^{0,\ov X}_s), \quad \mbox{and} \quad \lim_{m \to \infty} \sum_{s \leq t}|\Delta \beta^{0,\delta_m}_s|^2 = \sum_{s \leq t}|h(\Delta z^{0,\ov X}_s)|^4.
	\end{equation*}
	But this follows directly from the definition by applying the monotone convergence theorem, as $\lim_{\delta \rightarrow 0} \psi_\delta(x) = 0$ for all nonzero $x$. 
	
	For (ii), argue as in Lemma~\ref{le:limit-beta-n-delta-in-n} and apply \cite[Proposition~VI.2.2(a), p.338]{JacodShiryaev.03}. 
\end{proof}
%
%
%
Let us summarize what we have shown so far. 
By Lemma~\ref{le:quad-var-conv-Md-enlarge} and as $\Delta [\ov M^d] = h^2(\Delta \ov X) \ \ov \P_0$-a.s., we get $\Delta y^0 =h^2(\Delta z^{0,\ov X}) \ \lambda$-a.s.. Hence, we conclude that $y^{0,cont}$ has continuous paths. Moreover, $\ov M^d$ is a purely discontinuous $\ov \P_0$-$\ov \F$-martingale if and only if $y^{0,cont}=0 \ \lambda$-a.s.  In fact, 
decomposing 
\begin{equation*}
\ov M^d= \ov M^{d,c} + \ov M^{d,d} 
\end{equation*}
into its continuous and purely discontinuous martingale part under $\ov \P_0$, we see that $\ov \P_0 \circ (\langle \ov M^{d,c} \rangle)^{-1} = \lambda \circ (y^{0,cont})^{-1}$. Being a nonnegative, nondecreasing function starting in zero, we obtain $y^{0,cont}=0 \ \lambda$-a.s. if and only if $\E^{\lambda}[y^{0,cont}_T]=0$. Therefore, it remains to show that
\begin{equation}\label{eq:char2-Md-purely-disc-enlarge}
\E^{\lambda}\big[y^{0,cont}_T\big]=0 \ \Longleftrightarrow \ \lim\limits_{\delta\downarrow0}\limsup_{n \to \infty} \E^{\ov \P_n}\Big[\int_0^T\int_{\{|x|\leq \delta\}} |x|^2\, \ov F^{\ov \P_n}_t(dx)\,dt\Big] =0.
\end{equation}
%
%
%
%
In the subsequent proofs, $K>0$ will be a constant whose value may change from line to line. Moreover, we recall the constant $\cK<\infty$ defined in the Condition~(B). 
\begin{lemma}\label{le:limit-gamma-n-delta-in-n}
	For each $m \in \N$, the following hold true.
	\begin{align*}
	\E^{\lambda}\big[y^0_T-\beta^{0,\delta_m}_T\big]&=\lim\limits_{n \to \infty} \E^{\lambda}\big[\gamma^{n,\delta_m}_T\big]\\
	&=\lim\limits_{n \to \infty} \E^{\ov \P_n}\Big[\int_0^T \int_{\R} \psi_{\delta_m}(x)\,h^2(x)\,\ov F^{\ov \P_n}_s(dx)\,ds\Big].
	\end{align*}
\end{lemma}
\begin{proof}
	Fix $m \in \N$. We know from Lemma~\ref{le:limit-beta-n-delta-in-n} that $\gamma^{n,\delta_m}$ converges pointwise to $y^0-\beta^{0,\delta_m}$ in $\D([0,T],\R)$. By Proposition~\ref{prop:X-semim-cont-char}, $\ov X$ is a $\ov \P_0$-$\ov \F$-semimartingale with absolutely continuous characteristics, in particular $\ov \P_0[\Delta \ov X_t\neq 0]=0$ for every $t \in [0,T]$. As a consequence, we can conclude that $\lambda[\Delta y^0_t\neq 0]=0$ for all $t \in [0,T]$. Therefore, by \cite[VI.2.3, p.339]{JacodShiryaev.03}, 
	\begin{equation*}
	\gamma^{n,\delta_m}_T \mbox{ converges pointwise to } y^0_T-\beta^{0,\delta_m}_T \  \ \lambda\mbox{-a.s. when } n \to \infty.
	\end{equation*} 
	To obtain the first equality, it remains to show that the sequence $(\gamma^{n,\delta_m}_T)_{n \in \N}$ is uniformly integrable with respect to $\lambda$. By the de la Vall\'ee-Poussin theorem, it suffices to show the boundedness in $L^2(\lambda)$ of that sequence. Define the nondecreasing process
	\begin{equation*}
	\ov A^{m}_t:= \sum_{0\leq s \leq t} \psi_{\delta_m}(\Delta \ov X_s)\,h^2(\Delta \ov X_s), \quad t\in [0,T].
	\end{equation*}
	Then, for each $n\in \N$, we have $\ov \P_n\circ (\ov A^{m}_T)^{-1}= \lambda \circ (\gamma^{n,\delta_m}_T)^{-1}$, hence we need to check that $\sup_{n\in \N} \E^{\ov \P_n}[(\ov A^m_T)^2]<\infty$. To that end, by the product-rule, write
	\begin{equation*}
	(\ov A^m_T)^2 = 2\int_0^T \ov A^m_{s-} \, d\ov A^m_s + [\ov A^m]_T.
	\end{equation*} 
	Moreover, $\ov A^{m}$ has 
	\begin{equation*}
	\ov A^{m,\ov \P_n}:= \int_0^\cdot \int_{\R} \psi_{\delta_m}(x)\,h^2(x)\,\ov F^{\ov \P_n}_s(dx)\,ds
	\end{equation*}
	as its $\ov \P_n$-$\ov \F$-compensator.
	%
	As $\Theta$ satisfies Condition~(B), we obtain from \eqref{eq:bound-characteristics} that
	\begin{align*}\E^{\ov \P_n}\big[[\ov A^m]_T\big]
	&=\E^{\ov \P_n}\Big[\sum_{0\leq s \leq T} \psi^2_{\delta_m}(\Delta \ov X_s)\,h^4(\Delta \ov X_s)\Big]\\
	&\leq K\, \E^{\ov \P_n}\Big[\sum_{0\leq s \leq T} \psi_{\delta_m}(\Delta \ov X_s)\,h^2(\Delta \ov X_s)\Big]\\
	&=K\E^{\ov \P_n}\big[\ov A^{m,\ov \P_n}_T\big]\\
	&\leq K \cK T, 
	\end{align*}
	where $K$ 
	is a constant only depending on $h$.
	Similarly, using \cite[Theorem~I.3.17(iii), p.32]{JacodShiryaev.03}, 
	\begin{align*}
	\E^{\overline{\P}_n}\Big[\int_0^T \ov A^m_{s-} \, d\ov A^m_s\Big] = \E^{\overline{\P}_n}\Big[\int_0^T \ov A^m_{s-} \, d\ov A^{m,\ov \P_n}_s\Big] 
	&\leq \E^{\overline{\P}_n}[\ov A^m_T \ov A^{m,\ov \P_n}_T] \\
	&\leq K \cK T\,\E^{\overline{\P}_n}[\ov A^m_T] \\
	&\leq (K \cK\,T)^2.
	\end{align*}
	Therefore, we conclude that indeed $\sup_{n\in \N} \E^{\ov \P_n}[(\ov A^m_T)^2]<\infty$, hence the first equality holds. 
	
	The second equality follows simply from the fact that for each $n\in \N$
	\begin{equation*}
	\E^{\lambda}\big[\gamma^{n,\delta_m}_T\big]=\E^{\ov \P_n}[\ov A^m_T]=\E^{\ov \P_n}[\ov A^{m,\ov \P_n}_T] =\E^{\ov \P_n}\Big[\int_0^T \int_{\R} \psi_{\delta_m}(x)\,h^2(x)\,\ov F^{\ov \P_n}_s(dx)\,ds\Big].
	\end{equation*}
\end{proof}
Now we are able to prove Proposition~\ref{prop:purely-disc-enl}.
\begin{proof}[Proof of Proposition~\ref{prop:purely-disc-enl}]
	Recall that part (i) was already proven in Proposition~\ref{prop:X-semim-cont-char}, hence it remains to show part (ii).
	
	To see (ii), assume for the first direction that there is $u>0$ such that
	\begin{equation*}
	\lim\limits_{\delta\downarrow0}\limsup_{n \to \infty} \E^{\ov \P_n}\Big[\int_0^T\int_{\{|x|\leq \delta\}} |x|^2\, \ov F^{\ov \P_n}_t(dx)\,dt\Big]\geq u>0.
	\end{equation*}
	As by definition, $\psi_{\delta_m}\geq \mathbbm{1}_{\{|x|\leq \delta_m/2 \}},$ we get from Lemma~\ref{le:limit-gamma-n-delta-in-n} that for each big enough $m$
	\begin{equation*}
	\E^\lambda\big[y^0_T- \beta^{0,\delta_m}_T\big] \geq \limsup_{n \to \infty }\E^{\ov \P_n}\Big[\int_0^T\int_{\{|x|\leq \delta_m/2\}} |x|^2\, \ov F^{\ov \P_n}_t(dx)\,dt\Big]\geq u.
	\end{equation*}
	Now, a similar argument as in Lemma~\ref{le:limit-gamma-n-delta-in-n} yields
	\begin{equation*}
	\E^{\lambda}\big[(y^n_T)^2\big]
	\leq 2(K \cK T)^2 + K\cK T
	\end{equation*}
	uniformly for each $n$, so the sequence $(y^n_T)_{n \in \N}$ is uniformly integrable with respect to $\lambda$. As $y^n$ converges pointwisely to $y^0$ in $\D([0,T],\R)$ and $\lambda[\Delta y^0_t\neq 0]=0$, we  have $y^n_T$ converging to $y^0_T$ $\lambda$-a.s., and  by uniform integrability also in $L^1(\lambda)$. In particular, we see that $\E^{ \lambda}[y^0_T]<\infty$. Therefore, by dominated convergence, we can conclude that
	\begin{equation}\label{eq:dom-conv-y-0-cont}
	\E^\lambda[y^{0,cont}_T]
	=\lim\limits_{m \to \infty} \E^{\lambda}[y^{0}_T-\beta^{0,\delta_m}_T]\geq u,
	\end{equation}
	hence by the discussion before \eqref{eq:char2-Md-purely-disc-enlarge}, $\ov M^d$ is not a purely discontinuous martingale.
	
	Conversely, suppose that
	\begin{equation}\label{eq:jump-cond-0-enlarge}
	\lim_{\delta\downarrow 0}\limsup_{n \to \infty }\E^{\ov \P_n}\Big[\int_0^T\int_{\{|x|\leq \delta\}} |x|^2\, \ov F^{\ov \P_n}_t(dx)\,dt\Big]=0.
	\end{equation}
	From the fact that $\psi_{\delta_m} \leq \mathbbm{1}_{\{|x| \leq \delta_m\}}$ and by Lemma~\ref{le:limit-gamma-n-delta-in-n}, we have for each big enough $m$ 
	\begin{equation*}
	\E^{\lambda}\big[y^0_T-\beta^{0,\delta_m}_T\big] \leq \limsup_{n \rightarrow \infty} \E^{\overline{\P}_n}\Big[\int_0^T \int_{\{|x| \leq \delta_m \}} |x|^2 \, \ov F^{\ov \P_n}_s(dx) \, ds\Big].
	\end{equation*}
	Therefore, as in \eqref{eq:dom-conv-y-0-cont}, we obtain by dominated convergence and by \eqref{eq:jump-cond-0-enlarge} that
	\begin{align*}
	\E^\lambda\big[y^{0,cont}_T\big] 
	&=\lim\limits_{m\to \infty} \E^{\lambda}\big[y^0_T - \beta^{0,\delta_m}_T\big]\\
	& \leq \lim\limits_{m\to \infty} \limsup_{n \rightarrow \infty} \E^{\overline{\P}_n}\Big[\int_0^T \int_{\{|x| \leq \delta_m \}} |x|^2 \, \ov F^{\ov \P_n}_s(dx) \, ds\Big]\\
	&=0.
	\end{align*}
	Thus, $y^{0, cont}=0 \ \lambda$-a.s., hence by the discussion before \eqref{eq:char2-Md-purely-disc-enlarge}, we have that $\ov M^d$ is indeed a purely discontinuous martingale.
\end{proof}
\begin{remark}\label{rem:1-d-to-Multi-d}
	The proof of Proposition~\ref{prop:purely-disc-enl} was given in dimension $d=1$. However, we see from its proof that it can be easily adapted for the multidimensional case. Indeed, clearly, $ \ov M^d:=(\ov M^{d,1},\dots,\ov M^{d,d})$ is a purely discontinuous martingale if and only if each component $\ov M^{d,i}$ is. Therefore, one can apply the above proof for each of its component $\ov M^{d,i}$. Then, 
	\begin{equation*}
	\lim_{\delta\downarrow 0}\limsup_{n \to \infty }\E^{\ov \P_n}\Big[\int_0^T\int_{\{|x|\leq \delta\}} |x|^2\, \ov F^{\ov \P_n}_t(dx)\,dt\Big]=0
	\end{equation*}
	implies that each of its component $\ov M^{d,i}$ is a purely discontinuous martingale, hence so is $\ov M^d$. On the other hand,
	if 
	\begin{equation*}
	\lim_{\delta\downarrow 0}\limsup_{n \to \infty }\E^{\ov \P_n}\Big[\int_0^T\int_{\{|x|\leq \delta\}} |x|^2\, \ov F^{\ov \P_n}_t(dx)\,dt\Big]>0,
	\end{equation*}
	then there exists at least one $j\in \{1,\dots,d\}$ such that  $\ov M^{d,j}$ is not a purely discontinuous martingale, hence $\ov M^d$ is not one either.
\end{remark}
In the rest of this subsection, we provide the proof of Corollary~\ref{co:closedness-P-Theta-enlarge}, which we present directly in the multidimensional case.
\begin{proof}[Proof of Corollary~\ref{co:closedness-P-Theta-enlarge}]
	By Proposition~\ref{prop:X-semim-cont-char}, $\ov \P_0 \in \ov \fP^{ac,w}_{sem}$. In particular, $\ov X$ has canonical representation
	\begin{equation*}
	\ov X = \ov X_0 + \ov B + \ov M^c + \ov M^d + \ov J \quad \ov \P_0\mbox{-a.s.}
	\end{equation*}
	Moreover, as $\Theta \subseteq \R^d\times \S^d_+\times \cL$ is satisfying Condition~(J), we deduce from Proposition~\ref{prop:purely-disc-enl} that $\ov M^d$ is a $\ov \P_0$-$\ov \F$-purely discontinuous martingale, which implies that $\ov C$ is the second characteristic of $\ov X$ under $\ov \P_0$-$\ov \F$ and $\ov \P_0 \in \ov \fP^{ac}_{sem}$.
	
	Now, assume for the rest of the proof that in addition, $\Theta\subseteq \R^d \times\S^d_+ \times \cL$ is closed, convex. Let $(\ov b^{\ov \P_0}, \ov c^{\ov \P_0}, \ov F^{\ov \P_0})$ be the $\ov \P_0$-$\ov\F$-differential characteristics of $\ov X$. It remains to show that $(\ov b^{\ov \P_0}, \ov c^{\ov \P_0}, \ov F^{\ov \P_0}) \in \Theta \ \ov \P_0\times dt$-a.s. To that end, let us recall the additive, positive homogeneous bijection $\ov \varphi: \Theta \to \ov \varphi(\Theta)\subseteq \R^d\times \S^d_+\times\R^\N$ defined in \eqref{eq:def-ov-varphi}.  Due to \eqref{eq:char-in-Theta} it is equivalent to show that $(\ov b, \ov c, \ov v) \in \mbox{cl}(\ov \varphi(\Theta)) \ \ov \P_0\times dt$-a.s.. 
	Clearly, $\mbox{cl}(\ov \varphi(\Theta))$ is convex and closed.
	Therefore, 
	as $\ov\eta:=(\ov b, \ov c, \ov v) \in \mbox{cl}(\ov \varphi(\Theta)) \ \, \ov \P_n\times dt$-a.s., we have for each $m$ that $m \int_t^{t+\frac{1}{m}} \ov \eta_s\,ds \in \mbox{cl}(\ov \varphi(\Theta))\ \,\P_n\times dt$-a.s., for all $n \in \N$. In other words, for each $n \in \N$,
	\begin{equation*}
	\E^{\ov \P_n}\Big[\int_0^T \, \mathbbm{1}_{\big\{m\big(\ov B_{t+\frac{1}{m}} - \ov B_t,\, \ov C_{t+\frac{1}{m}} - \ov C_t,\, (\ov V^i_{t + \frac{1}{m}} - \ov V^i_t)_{i \in \N}\big) \in \mbox{cl}(\ov \varphi(\Theta))\big\}}\,dt\Big]=T.
	\end{equation*} 
	In terms of the Skorokhod representation, it means that
	\begin{equation*}
	\E^\lambda\Big[\int_0^T \mathbbm{1}_{\big\{m\big(z^{n,\ov B}_{t+\frac{1}{m}} - z^{n,\ov B}_t,\, z^{n,\ov C}_{t+\frac{1}{m}} - z^{n,\ov C}_t,\, (z^{n,\ov V^i}_{t + \frac{1}{m}} - z^{n,\ov V^i}_t)_{i \in \N}\big) \in \mbox{cl}(\ov \varphi(\Theta))\big\}}\, dt\Big] = T.
	\end{equation*}
	Now, letting $n$ tends to infinity, by applying the dominated convergence theorem, the closedness of  $\mbox{cl}(\ov \varphi(\Theta))$ yields
	\begin{equation*}
	\E^\lambda\Big[\int_0^T \mathbbm{1}_{\big\{m\big(z^{0,\ov B}_{t+\frac{1}{m}} - z^{0,\ov B}_t,\, z^{0,\ov C}_{t+\frac{1}{m}} - z^{0,\ov C}_t,\, (z^{0,\ov V^i}_{t + \frac{1}{m}} - z^{0,\ov V^i}_t)_{i \in \N}\big) \in \mbox{cl}(\ov \varphi(\Theta))\big\}}\, dt\Big] = T,
	\end{equation*}
	which implies that $m(\ov B_{t+\frac{1}{m}} - \ov B_t, \ov C_{t+\frac{1}{m}} - \ov C_t, (\ov V^i_{t + \frac{1}{m}} - \ov V^i_t)_{i \in \N}) \in \mbox{cl}(\ov \varphi(\Theta))$, $\overline{\P}_0 \times dt$-a.e. for all $m \in \N$. Now letting $m$ tend to infinity, we obtain that $\ov \eta_t = (\ov b_t, \ov c_t, \ov v_t) \in \mbox{cl}(\ov \varphi(\Theta))$ holds $\overline{\P}_0 \times dt$-a.e., hence we are done.
\end{proof}
%
\subsection{Tightness of $\ov \fP^{ac}_{sem}(\Theta)(\Gamma_0)$} \label{subsec:tightness-enlarge}
The goal of this subsection is to prove tightness of $\ov \fP^{ac}_{sem}(\Theta)(\Gamma_0)$.
More precisely:
\begin{proposition}\label{prop:tightness-enlarge}
	Let $\Theta \subseteq \R^d\times \S^d_+\times \cL$ satisfy Condition~(B) and $\Gamma_0\subseteq\fM_1(\R^d)$ be tight. Then the set $\ov \fP^{ac}_{sem}(\Theta)(\Gamma_0)$ defined as 
	in Subsection~\ref{subsec:enlarg} is tight.
\end{proposition}  
In fact, a careful inspection of the proof will lead to the conclusion that with the same arguments 
we  also obtain the tightness of $\fP^{ac}_{sem}(\Theta)(\Gamma_0)$ on the original space $\fM_1(\Omega)$ (when imposing the same conditions as Proposition~\ref{prop:tightness-enlarge}), see Corollary~\ref{co:tightness-P}.

We divide the proof of Proposition~\ref{prop:tightness-enlarge} into several lemmas. We will use the notation introduced in Section~\ref{subsec:enlarg}.
We recall the constant $\cK<\infty$ defined in the Condition~(B). 
%
%
\begin{lemma}\label{le:tight-B-C}
	Let $\Theta \subseteq \R^d\times \S^d_+\times \cL$ satisfy Condition~(B) and $(\ov \P_n)_{n \in \N}\subseteq \ov \fP^{ac}_{sem}(\Theta)$. Then, the sequences $(\ov \P_n \circ \ov B^{-1})_{n \in \N}$, $(\ov \P_n \circ \ov C^{-1})_{n \in \N}$, $(\ov \P_n \circ \ov{\widetilde{C}}^{-1})_{n \in \N}$  of laws on  $\Omega^c_d$
	and $\Omega^c_{d^2}$ are tight.   
\end{lemma}
\begin{proof}
	By definition, we have $\ov \P_n[\ov B_0=0]=1$ for each $n$. Moreover, due to $\Theta$ satisfying Condition~(B), we have $\ov \P_n$-a.s. that for any $0\leq s\leq t\leq T$
	\begin{equation*}
	|\ov B_t-\ov B_s| \leq \int_s^t |\ov b_r|\,dr \leq \cK |t-s|.
	\end{equation*}
	Consequently, by Markov's inequality, we have for each $\varepsilon>0$ that
	\begin{equation*}
	\lim\limits_{\delta\downarrow0}\sup_{n \in \N}\ov \P_n\Big[\sup_{t-s\leq \delta;0\leq s\leq t\leq T} |\ov B_t-\ov B_s|\geq\varepsilon\Big] =0.
	\end{equation*} 
	By applying \cite[Theorem~7.3, p.82]{Billingsley.99}, we obtain the tightness of $(\ov \P_n \circ \ov B^{-1})_{n \in \N}$. Replacing $\ov B$ by $\ov C$ 
	in the above proof, we conclude the tightness also for the sequence 
	$(\ov \P_n \circ \ov C^{-1})_{n \in \N}$. Moreover, using \cite[II.2.18, p.79]{JacodShiryaev.03}, the same arguments yields tightness of  $(\ov \P_n \circ \ov{\widetilde{C}}^{-1})_{n \in \N}$.
\end{proof}
%
\begin{lemma}\label{le:tight-M}
	Let $\Theta \subseteq \R^d\times \S^d_+\times \cL$ satisfy Condition~(B) and let $(\ov \P_n)_{n \in \N}\subseteq \ov \fP^{ac}_{sem}(\Theta)$. Then, both  sequences $(\ov \P_n \circ (\ov M^c)^{-1})_{n \in \N}$ and $(\ov \P_n \circ (\ov M^d)^{-1})_{n \in \N}$  of laws on $\Omega^c_d$ and $\Omega$ are tight. 
\end{lemma}
\begin{proof}
	For each $n \in \N$, we have by definition that both $\ov M^c$ and $\ov M^d$ are $\ov \P_n$-$\ov \F$-(locally) square integrable martingales with $\ov \P_n[\ov M^c_0=0]=\ov \P_n[\ov M^d_0=0]=1$.
	By Lemma~\ref{le:tight-B-C}, the sequence $(\ov \P_n \circ \ov C^{-1})_{n \in \N}$ of laws on $\Omega^c_{d^2}$ is tight. As by definition, $\langle \ov M^{c} \rangle = \ov C$ under each $\ov \P_n$, we deduce the tightness of  $(\ov \P_n \circ (\ov M^c)^{-1})_{n \in \N}$ 
	directly from \cite[Corollary~3, p.29]{Rebolledo.79}.
	
	Now, we define the process
	\begin{equation*}
	\ov G^{d,\ov \P_n}
	:= \sum_{j=1}^d \int_0^\cdot \int_{\R^d} |h^j(x)|^2 \, \ov F_t^{\ov \P_n}(dx)\,dt.
	\end{equation*}
	Since $\Theta$ satisfies Condition~(B), we can apply the same arguments as in Lemma~\ref{le:tight-B-C} to obtain the tightness of the sequence $(\ov \P_n \circ (\ov G^{d,\ov \P_n})^{-1})_{n \in \N}$ as laws on $\Omega^c_1$ or equivalently the $C$-tightness, if we think of them as laws on $\D([0,T],\R)$. As a consequence, we get the tightness of the sequence of laws $(\ov \P_n \circ (\ov M^d)^{-1})_{n \in \N}$ directly from \cite[Theorem~VI.4.13, p.358]{JacodShiryaev.03}.
\end{proof}
For any adapted process $\ov Z$ defined on $(\ov \Omega, \ov \cF,\ov \F)$ and any law $\ov \P$ defined on that filtered measurable space, we write $(\ov Z,\ov \P)$ when considering $(\ov \Omega, \ov \cF,\ov \F, \ov \P)$ as its stochastic basis. Moreover, we refer to \cite[p.377]{JacodShiryaev.03} for the standard notion for a sequence of processes to be Predictably Uniformly Tight (P-UT).
\begin{lemma}\label{le:P-UT}
	Let $\Theta \subseteq \R^d\times \S^d_+\times \cL$ satisfy Condition~(B) and let $(\ov \P_n)_{n \in \N}\subseteq \ov \fP^{ac}_{sem}(\Theta)$. Then, the sequence $(\ov X,\ov \P_n)_{n \in \N}$ is 
	(P-UT).
\end{lemma}
\begin{proof}
	Fix any $t \in [0,T]$ and $i=1,\dots,d$.
	Due to Condition~(B), the variation of $B^i$ satisfies
	\begin{equation*}
	\sup_{n}\E^{\ov \P_n}\big[\mbox{Var}(\ov B^i)_t\big]
	\leq\sup_{n}\E^{\ov \P_n}\Big[\int_0^t |\ov b_s|\, ds\Big] \leq \cK t,
	\end{equation*}
	hence tightness of  $(\ov \P_n\circ(\mbox{Var}(\ov B^{i})_t)^ {-1})_{n\in \N}$ follows directly from the Markov inequality.
	
	Now, the  modified second characteristic $\ov{\widetilde{C}}$ satisfies in each component $1\leq k,l\leq d$ 
	\begin{equation*}
	\overline{\widetilde{C}}^{kl}=\ov C^{kl} + \int_0^\cdot \int_{\R^d} h^{k}(x) h^l(x) \, \ov F^{\ov \P_n}_s(dx)\,ds \quad \ov \P_n\mbox{-a.s.}, 
	\end{equation*}
	hence we can argue as above to obtain tightness of the sequence $(\P_n\circ(\overline{\widetilde{C}}^{ii}_t)^{-1})_{n\in \N}$.
	
	Next, we see that for the variation of $\ov J$, we have
	\begin{align*}
	\E^{\ov \P_n}\big[\mbox{Var}(\ov J)_t\big]
	&=
	\E^{\ov \P_n}\Big[\int_0^t \int_{\R^d}|x-h(x)|\, \mu^{\ov X}(dx)\,ds\Big]\\
	&=\E^{\ov \P_n}\Big[\int_0^t \int_{\R^d} |x-h(x)|\, \ov F_s^{\ov \P_n}(dx)\,ds\Big],
	\end{align*}
	hence in the same way as above, we get also the tightness of the sequence $(\ov\P_n\circ \mbox{Var}(\ov J)_t^{-1})_{n\in \N}$. As under each $\ov \P_n, \ov X$ is a $\ov \P_n$-$\ov \F$-semimartingale with canonical representation \eqref{eq:can-dec-X-enlarged}, we conclude that the sequence $(\ov X,\ov \P_n)_{n \in \N}$ is (P-UT) directly from \cite[Theorem~6.15, p.380]{JacodShiryaev.03}.
\end{proof}
\begin{lemma}\label{le:tight-X}
	Let $\Theta \subseteq \R^d\times \S^d_+\times \cL$ satisfy Condition~(B), $\Gamma_0\subseteq\fM_1(\R^d)$ be tight and let $(\ov \P_n)_{n \in \N}\subseteq \ov \fP^{ac}_{sem}(\Theta)(\Gamma_0)$.  Then, the  sequence $(\ov \P_n \circ \ov X^{-1})_{n \in \N}$ is tight. 
\end{lemma}
\begin{proof}
	As $\ov X$ is a $\ov \P_n$-$\ov \F$-semimartingale for each $n$, it suffices to verify that the conditions in \cite[Theorem~VI.4.18, p.359]{JacodShiryaev.03} are satisfied. 
	
	By assumption,  $(\ov \P_n \circ (\ov X_0)^{-1})_{n \in \N}\subseteq \Gamma_0$ is tight. By Lemma~\ref{le:P-UT}, $(\ov X, \ov \P_n)_{n \in \N}$ is P-UT, hence by \cite[Theorem~6.16, p.380]{JacodShiryaev.03}
	\begin{equation*}
	\lim\limits_{a\uparrow \infty} \sup_{n \in \N} \ov \P_n\big[\nu^{\ov \P_n}([0,T]\times\{x:\, |x|>a\})>\varepsilon\big]=0.
	\end{equation*}
	Consider the following increasing process
	\begin{equation*}
	\ov D^{\ov \P_n}:= \sum_{i=1}^d \big[\Var(\ov B^i) + \ov C^{ii}\big] + \int_0^\cdot\int_{\R^d} |x|^2\wedge 1\,\ov F_s^{\ov \P_n}(dx)\,ds  
	\end{equation*}
	Using the same argument as in Lemma~\ref{le:tight-B-C}, we obtain the C-tightness of the sequence $(\ov \P_n\circ (\ov D^{\ov \P_n})^{-1})_{n\in \N}$.
	Therefore, by \cite[Remark~VI.4.20, pp.359--360]{JacodShiryaev.03}, the conditions of \cite[Theorem~VI.4.18, p.359]{JacodShiryaev.03} are satisfied, hence we get the result.
	%
\end{proof}
\begin{lemma}\label{le:tight-J-U-V}
	Let $\Theta \subseteq \R^d\times \S^d_+\times \cL$ satisfy Condition~(B) and let $(\ov \P_n)_{n \in \N}\subseteq \ov \fP^{ac}_{sem}(\Theta)$. 
	Then the following hold true.\\
	(i) The  sequence of laws $(\ov \P_n \circ \ov J^{-1})_{n \in \N}$ on $\Omega$ is tight. \\
	(ii)  The  sequence of laws $(\ov \P_n \circ (\ov U^{i})^{-1})_{n \in \N}$ on $\Omega_1$ is tight for each $i\in \N$.\\
	(iii) The  sequence of laws $(\ov \P_n \circ (\ov V^{i})^{-1})_{n \in \N}$ on $\Omega^c_1$ is tight for each $i\in \N$. 
	\\
	(iv) The  sequence of laws $(\ov \P_n \circ \ov{[M]}^{-1})_{n \in \N}$ on $\Omega_{d^2}$ is tight.
\end{lemma}
\begin{proof}
	Using the notation introduced in  \eqref{eq:jumps-away-0-cont}, we have $\ov J=I^{x-h(x)}(\ov X-\ov X_0)\ \ov \P_n$-a.s. for each $n$, and for each $i \in \N$, $\ov U^i=I^{g_i(x) (|x|^2\wedge 1)}(\ov X-\ov X_0)\ \ov \P_n$-a.s.. Therefore, (i) and (ii) follows directly from the tightness of $(\ov \P_n \circ (\ov X-\ov X_0)^{-1})_{n \in \N}$, see Lemma~\ref{le:tight-X}. To see that (iii) holds, recall that under each $\ov \P_n$, we have for each $i \in \N$ that
	\begin{equation*}
	\ov V^i= \int_0^\cdot \int_{\R^d} g_i(x)\, \ov F_s^{\ov \P_n}(dx)\,ds \quad \ov \P_n\mbox{-a.s.}.
	\end{equation*}
	As $g_i(x)$ is a bounded function vanishing in a neighborhood of the origin, we can argue as in Lemma~\ref{le:tight-B-C} to obtain (iii). 
	Finally, (iv) follows by \cite[Theorem~II.2, p.28]{Rebolledo.79} from tightness of $(\ov \P_n \circ \ov{\widetilde{C}}^{-1})_{n \in \N}$, see Lemma~\ref{le:tight-B-C}.
\end{proof}
%
%
Now we are able to prove Proposition~\ref{prop:tightness-enlarge}.
\begin{proof}[Proof of Proposition~\ref{prop:tightness-enlarge}] 
	By an application of Prohorov's theorem, it suffices to show that any sequence $(\ov \P_n)\subseteq \ov \fP^{ac}_{sem}(\Theta)(\Gamma_0)$ is tight. To that end, fix such a sequence.  From all previous lemmas, we have established the tightness of all the following families
	\begin{align*}
	& \ (\overline{\P}_n \circ \ov X^{-1})_{n \in \N}, (\overline{\P}_n \circ \ov  B^{-1})_{n \in \N}, (\overline{\P}_n \circ {(\ov M^c)}^{-1})_{n \in \N}, (\overline{\P}_n \circ {(\ov M^d)}^{-1})_{n \in \N},\\
	& \ (\overline{\P}_n \circ {\ov J}^{-1})_{n \in \N},
	(\overline{\P}_n \circ \ov C^{-1})_{n \in \N},
	(\overline{\P}_n \circ \ov{[M]}^{-1})_{n \in \N},
	(\overline{\P}_n \circ \ov{\widetilde C}^{-1})_{n \in \N},\\
	& \  (\overline{\P}_n \circ {(\ov U^i)}^{-1})_{n \in \N}, (\overline{\P}_n \circ {(\ov V^i)}^{-1})_{n \in \N}, \quad i \in \N.
	\end{align*}
	Fix $\varepsilon>0$, and let $K(\ov X)\subseteq \Omega$ be a compact set such that $\sup_{n \in \N}\ov \P_n[\ov X \notin K(\ov X)]\leq \frac{\varepsilon}{10}$. The same way,  define the compact sets $K(\ov B)\subseteq \Omega^c_d, \dots,  K(\ov{\widetilde C}) \subseteq \Omega^c_{d^2}$. Moreover, choose for each $i \in \N$ compact sets $K(\ov U^i)\subseteq \Omega_1,K(\ov V^i)\subseteq \Omega^c_1$ such that
	\begin{equation*}
	\max\Big\{\sup_{n \in \N}\ov \P_n[\ov U^i \notin K(\ov U^i)],\,\sup_{n \in \N}\ov \P_n[\ov V^i \notin K(\ov V^i)]\Big\}\leq \frac{\varepsilon}{10}\frac{1}{2^i}
	\end{equation*}
	By Tychonoff's theorem, the set 
	$$\ov K:=K(\ov X)\times K(\ov B)\times\dots \times K(\ov{\widetilde{C}})\times\Pi_{i\in \N} [K(\ov U^i)\times K(\ov V^i)] \subseteq \ov \Omega$$
	is compact. Moreover, we have for each $n \in \N$ that
	\begin{align*}
	\ov \P_n[\ov K^c]
	&\leq  \ov \P_n[\ov X \notin K(\ov X)] + \dots +\ov \P_n[\ov{\widetilde{C}} \notin K(\ov{\widetilde{C}})] 
	\\
	& \ \ \ \,
	+ \sum_{i=1}^\infty \big( \ov \P_n[\ov U^i \notin K(\ov U^i)] + \ov \P_n[\ov V^i \notin K(\ov V^i)]\big)\\ &\leq \varepsilon.
	\end{align*}
\end{proof}
In fact, we observe that by the same arguments as above, we also get tightness of $\fP^{ac}_{sem}(\Theta)(\Gamma_0)$ in the original space $\fM_1(\Omega)$.
\begin{corollary}\label{co:tightness-P}
	Let $\Theta \subseteq \R^d\times \S^d_+\times \cL$ satisfy Condition~(B) and $\Gamma_0\subseteq\fM_1(\R^d)$ be tight. Then $ \fP^{ac}_{sem}(\Theta)(\Gamma_0)$ is tight.
\end{corollary}
\begin{proof}
	By an application of Prohorov's theorem, it suffices to show that any sequence $(\P_n)\subseteq \fP^{ac}_{sem}(\Theta)$ is tight. For any such a sequence,  we can apply exactly the same argument as in Lemma~\ref{le:tight-X} to obtain the tightness result.
\end{proof}
%
%
%
%
\section{Proof of Theorem~\ref{thm:purely-disc}, Theorem~\ref{thm:compactness}, and Theorem~\ref{thm:compactness2}}\label{sec:prf-thm-semi}
The goal of this section is to prove Theorem~\ref{thm:purely-disc} and Theorem~\ref{thm:compactness}. The strategy of their proofs is the following. In the last section, we stated and proved corresponding results in an enlarged space. We need to find a way how to go back and forth from the original space $\Omega=\D([0,T],\R^d)$ to the enlarged space $\ov \Omega$ to conclude the results in the original space $\Omega$. On the one hand, to get from $\Omega$ to $\ov \Omega$, we introduce for any  measure $\P \in \fP_{sem}^{ac}(\Theta)$ a corresponding measure $\ov \P \in \ov \fP_{sem}^{ac}(\Theta)$ which is simply the probability measure induced by the canonical representation of $X$ under $\P$. On the other hand, to get back from $\ov \Omega$ to $\Omega$, we consider for any  $\ov \P \in \ov \fP_{sem}^{ac}(\Theta)$ the corresponding push forward measure $\P:=\ov \P \circ \ov X^{-1}$, which will turn out to be in $\fP_{sem}^{ac}(\Theta)$.

Let us start with a fixed probability measure $\P \in \fP^{ac}_{sem}$. 
Consider the canonical representation of $X$ under $\P$-$\F^\P_+$ is given by
\begin{equation*}
X=X_0 + B^{\F^\P_+} + M^{c,\F^\P_+} + M^{d,\F^\P_+} + J,
\end{equation*}
where $J:=\sum_{0\leq s \leq \cdot} [\Delta X_s -h(\Delta X_s)]$, $B^{\F^\P_+}=\int_0^\cdot b^{\F^\P_+}_s\,ds$ and $M^{c,\F^\P_+}, M^{d,\F^\P_+}$ denotes the continuous and purely discontinuous local martingale part. Moreover, denote by $C^{\F^\P_+}$ and $F^{\F^\P_+}(dx)\,ds$ the second and third $\P$-$\F^\P_+$-characteristic of $X$, by $[M^{\F^\P_+}]$ the quadratic variation of $(M^{c,\F^\P_+} + M^{d,\F^\P_+})$, by $\widetilde{C}^{\F^\P_+}$ the  modified second $\P$-$\F^\P_+$-characteristic of $X$, and write
\begin{align*}
U^i&:= \int_0^\cdot \int_{\R^d} g_i(x)\, \mu^{X}(dx,ds),\\ 
V^{i,\F^\P_+}&:= \int_0^\cdot \int_{\R^d} g_i(x)\, F^{\F^\P_+}(dx)\,ds.
\end{align*}
We can define the  map $\Psi^{\P}:\Omega \to \ov \Omega$ by
\begin{align*}
\omega \mapsto \Big(&X(\omega),B^{\F^\P_+}(\omega),M^{c,\F^\P_+}(\omega),M^{d,\F^\P_+}(\omega),
J(\omega), C^{\F^\P_+}(\omega),[M^{\F^\P_+}],\widetilde{C}^{\F^\P_+},\\& \ \  (U^{i}(\omega),V^{i,\F^\P_+}(\omega))_{i \in \N}\Big),
\end{align*}
which in measurable with respect to the Borel $\sigma$-field, completed by $\P$.
Then,  the measure 
\begin{equation}\label{eq:def:P-to-enlarge-P}
\ov \P:= \P \circ (\Psi^{\P})^{-1}
\end{equation}
is an element of $\ov \fP^{ac}_{sem}$. We used the canonical representation of $X$ with respect to $\F^\P_+$ to guarantee that for every $\omega$, each summand has c\`adl\`ag paths and continuous paths, respectively, and not only $\P$-a.s., so that $\Psi^\P$ is well-defined. However, as the characteristics of $X$ do not depend on the choice of $\F$ or $\F^\P_+$, we conclude that $\P \in \fP^{ac}_{sem}(\Theta)$ implies $\ov \P \in \ov\fP^{ac}_{sem}(\Theta)$, i.e. $\ov \P$ preserves the structure of $\P$.
%
\begin{remark}\label{rem:from-P-to-P-enlarge}
	By construction, we have
	\begin{equation*}
	\ov \P \circ \ov X^{-1}= \P \circ (\ov X \circ\Psi^{\P})^{-1}=\P=\P\circ X^{-1}.
	\end{equation*}
	This implies e.g. that for each $\delta>0$, we have
	\begin{align*}
	\E^{\ov \P}\Big[\int_0^T\int_{\{|x|\leq \delta\}} |x|^2\, \ov F^{\ov \P}_t(dx)\,dt\Big]
	&=\E^{\ov \P}\Big[\sum_{0 \leq t \leq T} |\Delta \ov X_t|^2 \, \mathbbm{1}_{\{0<\Delta |\ov X_t|\leq  \delta \}}\Big]\\
	&=\E^{ \P}\Big[\sum_{0 \leq t \leq T} |\Delta X_t|^2 \, \mathbbm{1}_{\{0<\Delta | X_t|\leq  \delta\}}\Big]\\
	&=\E^{ \P}\Big[\int_0^T\int_{\{|x|\leq \delta\}} |x|^2\, F^{ \P}_t(dx)\,dt\Big],
	\end{align*}
	which will play a crucial role in the proof of Theorem~\ref{thm:purely-disc}.
\end{remark}

On the other hand, the natural candidate connecting the set $\ov \fP_{sem}^{ac}(\Theta)$ with $\fP_{sem}^{ac}(\Theta)$ seems to be for any  $\ov \P \in \ov \fP_{sem}^{ac}(\Theta)$ its corresponding push forward measure $\P:=\ov \P \circ \ov X^{-1}$. The following positive answer is stated in such a way that it is compatible with any limit law $\ov \P_0$ of sequences $(\ov\P_n)_{n \in \N}\subseteq \ov \fP^{ac}_{sem}(\Theta)$ with $\Theta$ satisfying Condition~(B), see Proposition~\ref{prop:X-semim-cont-char}.
\begin{lemma}\label{le:push-forward}
	Let $\ov \P \in \ov\fP^{ac,w}_{sem}$ satisfying
	\begin{align*}
	& \bullet \ \E^{\ov \P}\Big[\int_0^T \big[|\ov b^{\ov \P}_s|  + |\ov c^{\ov \P}_s| +\int_{\R^d} |h(x)|^2\,\ov F^{\ov \P}_s(dx)\big]\, ds\Big]<\infty,\\
	&\bullet \ \mbox{both $\ov M^c$, $\ov M^d$ are $\ov \P$-$\ov \F$-martingales.}
	\end{align*}
	Then, the corresponding pushforward measure
	\begin{equation}\label{eq:def:from-P-enlarge-to-P}
	\P:= \ov \P \circ (\ov X)^{-1}
	\end{equation} 
	is an element in $\fP^{ac}_{sem}$. Moreover, if the $\ov \P$-$\ov \F$-differential characteristics $(\ov b^{\ov \P}, \ov c^{\ov \P}, \ov F^{\ov \P})$ are taking values in some set $\Theta \subseteq \R^d \times \S^d_+ \times \cL$ which is closed, convex and satisfy Condition~(B), then the corresponding pushforward measure $\P:= \ov \P \circ (\ov X)^{-1}$ is 
	an element in $\fP_{sem}^{ac}(\Theta)$.
\end{lemma} 
\begin{proof}
	By  Lemma~\ref{le:semimart-char-smaller-filtr}, $\ov X$ is a $\ov \P$-$\ov\F^{\ov X}$-semimartingale with absolutely continuous characteristics, and the local martingale part  $\ov M^{\ov X, \ov \P}$ of  the canonical representation  
	\begin{equation*}
	\ov X = \ov X_0 + \ov B^{\ov X, \ov \P} + \ov M^{\ov X, \ov \P} + \sum_{0 \leq s \leq \cdot} [\Delta \ov X_s- h(\Delta \ov X_s)]  \quad \ov \P\mbox{-a.s.}
	\end{equation*}
	of $\ov X$ under $\ov \P$-$\ov\F^{\ov X}$ is a  $\ov \P$-$\ov\F^{\ov X}$-martingale.
	Define the map $\Phi:\Omega \to \ov \Omega$, by $\omega \mapsto (\omega,0,0,0,\dots)$. We claim that
	\begin{equation}\label{eq:can-decomp-pushforward}
	\ov X_0 \,\circ\, \Phi\, + \ov B^{\ov X, \ov \P} \,\circ\, \Phi \, + \ov M^{\ov X, \ov \P} \,\circ \,\Phi \, + \sum_{0 \leq s \leq \cdot} [\Delta \ov X_s- h(\Delta \ov X_s)] \ \circ \,\Phi\, 
	\end{equation}
	is the canonical representation of $X$ under $\P$-$\F$. To see this, observe that each summand in
	\eqref{eq:can-decomp-pushforward} is $\F$-adapted, as $\Phi^{-1}(A) \in \cF_t$ for all $A \in \overline{\cF}^{\ov X}_t$ for each $t \in [0,T]$.
	In view of \cite[Exercise 1.5.6, p.44]{StroockVaradhan.79}, the $\overline{\F}^{\ov X}$-adapted process $\ov B^{\ov X, \ov \P}$ admits a representation 
	\begin{equation*}
	\ov B^{\ov X, \ov \P}_t = \Lambda(t,\ov X_{t_1}, \ov X_{t_2}, ... ), \ \quad t\in [0,T],
	\end{equation*}
	where $\Lambda$ is a measurable function defined on the product space $[0,T] \times (\R^d)^\N$ and $0 \leq t_1 < t_2 <...$ is a sequence in $[0,T]$. This implies that for all $\overline{\omega} \in \overline{\Omega}$ and $t \in [0,T]$, it holds that
	\begin{equation*} 
	\ov B^{\ov X, \ov \P}_t(\ov{\omega}) = \ov B^{\ov X, \ov \P}_t \circ \Phi \circ \ov X(\ov{\omega}).
	\end{equation*}
	In particular, we obtain that
	\begin{equation} \label{eq:law-P-ov-P}
	\mbox{the law of  $B^{\ov X, \ov \P} \circ \Phi$ under $\P$ \ }=  \mbox{ \ the law of  $B^{\ov X, \ov \P}$ under $\ov \P$,} 
	\end{equation}
	hence $B^{\ov X, \ov \P} \circ \Phi$ is absolutely continuous $\P$-a.s.
	Of course, \eqref{eq:law-P-ov-P}  also holds with respect to the other summands 
	in \eqref{eq:can-decomp-pushforward}.
	In particular, $\ov M^{\ov X, \ov \P} \,\circ \,\Phi $ is a $\P$-$\F$-martingale  and \eqref{eq:can-decomp-pushforward} is $\P$-a.s. equal to $\ov X \circ \,\Phi=X$, hence \eqref{eq:can-decomp-pushforward} is indeed the canonical representation of $X$ under $\P$-$\F$.
	
	Next, from the canonical representation, we see that
	$B^{\ov X, \ov \P} \circ \Phi$ is the first characteristic of $X$ under $\P$-$\F$, which we argued above to be $\P$-a.s. absolutely continuous.
	Denote by $\nu^{\P}(dx,dt)$ the 
	third characteristic of $X$ under $\P$-$\F$. Applying 
	$\eqref{eq:law-P-ov-P}$ to the 
	third characteristic $\nu^{\ov X, \ov \P}$ of $\ov X$ under $\ov \P$-$\ov \F^{\ov X}$ yields that for any $g_i \in \cC^+(\R^d)$
	\begin{align}\label{eq:law-P-ov-P-compensator}
	& \ \mbox{\,the law of  $g_i(x) \ast  \nu^{\ov X, \ov \P}\, \circ \Phi$ under $\P$ \ }  \\
	=  & \ \mbox{\,the law of  $g_i(x) \ast   \nu^{\ov X, \ov \P}$ under $\ov \P$.}  \nonumber
	\end{align}
	As a consequence, we have for each $g_i \in \cC^+(\R^d)$ that
	\begin{equation*}
	g_i(x) \ast \nu^{\P}(dx,dt)=g_i(x)  \ast  \nu^{\ov X, \ov \P} \circ \Phi,
	\end{equation*}
	which implies that $\nu^{\P}$ satisfies a disintegration $\nu^{\P}(dx,dt)=F^{\P}_t(dx)\,dt \ \P$-a.s. For the second characteristic $C^\P$ of $X$ under $\P$-$\F$, observe that $X= \ov X \circ \Phi$ implies the same for the quadratic variation, namely $[X]= [\ov X] \circ \Phi$. As the second characteristic is the continuous part of the quadratic variation as finite variation process (see e.g. \cite[Proposition~6.6]{NeufeldNutz.13a}), we obtain that $C^{\P}= \ov C^{\ov X,\ov \P} \circ \Phi$,   where $\ov C^{\ov X,\ov \P}$ denotes the second characteristic of $\ov X$ under $\ov \P$-$\ov \F^{\ov X}$ (which coincides with the one with respect to $\ov \P$-$\ov \F$). Therefore, applying $\eqref{eq:law-P-ov-P}$ as above, but with respect to the second characteristic, yields 
	that $C^\P$ is absolutely continuous $\P$-a.s. We conclude that $\P \in \fP^{ac}_{sem}$.
	
	Now, for the rest of the proof, assume that in addition $\ov \P \in \ov\fP^{ac}_{sem}(\Theta)$ for some  $\Theta \subseteq \R^d \times \S^d_+ \times \cL$ which is closed, convex and satisfies Condition~(B). From the above arguments, we see that $X$ is a $\P$-$\F$ semimartingale with differential characteristics
	\begin{equation*}
	(b^\P,c^\P,F^\P)=(\ov b^{\ov X, \ov \P}\circ \Phi, \ov c^{\ov X, \ov \P} \circ \Phi, \ov F^{\ov X, \ov \P}\circ \Phi),
	\end{equation*}
	where $(\ov b^{\ov X,\P},\ov c^{\ov X,\P}, \ov F^{\ov X,\P})$ denotes the  $\ov \P$-$\ov \F^{\ov X}$-differential characteristics of $\ov X$.
	To see that the differential characteristics are taking values in $\Theta \ \P \times dt$-a.s., we recall the map $\ov\varphi:\R^d\times \S^d_+\times \cL \to \R^d\times \S^d_+\times \R^\N$ defined in \eqref{eq:def-ov-varphi} (with corresponding function $\varphi$, see \eqref{eq:map}).  By definition of the map $\ov\varphi$, \eqref{eq:law-P-ov-P-compensator} and Lemma~\ref{le:semimart-char-smaller-filtr}, we obtain that
	\begin{align*}
	&\, \E^\P\Big[\int_0^T \mathbbm{1}_{\{\overline{\varphi}((b^\P_t,c^\P_t, F^\P_t))\in \mbox{cl}(\overline{\varphi}(\Theta))\}} \, dt\Big]\\
	=&\, \E^\P\Big[\int_0^T \mathbbm{1}_{\{(\ov b^{\ov X, \ov \P}_t\circ \Phi, \ov c^{\ov X, \ov \P}_t \circ \Phi, \varphi(\ov F^{\ov X, \ov \P}_t\circ \Phi))\in \mbox{cl}(\overline{\varphi}(\Theta))\}} \, dt\Big] \\
	=&\, \E^{\overline{\P}}\Big[\int_0^T \mathbbm{1}_{\{(\ov b^{\ov X, \ov \P}_t, \ov c^{\ov X, \ov \P}_t, \varphi(\ov F^{\ov X, \ov \P}_t)) \in \mbox{cl}(\overline{\varphi}(\Theta))\}} \, dt\Big] \\
	=&\, \E^{\overline{\P}}\Big[\int_0^T \mathbbm{1}_{\{\ov  \varphi((\ov b^{\ov X, \ov \P}_t, \ov c^{\ov X, \ov \P}_t, \ov F^{\ov X, \ov \P}_t)) \in \mbox{cl}(\overline{\varphi}(\Theta))\}} \, dt\Big] \\
	=&\, T.
	\end{align*}
	Using the property \eqref{eq:char-in-Theta} of the map $\ov \varphi$, we conclude that $\P:=\ov \P \circ (\ov X)^{-1} \in \fP^{ac}_{sem}(\Theta)$.
	%
\end{proof}
Let us continue with the following Lemma.
\begin{lemma}\label{le:identify-limit-P-0}
	Let $(\P_n)\subseteq \fP^{ac}_{sem}$ and consider the corresponding sequence  $(\ov\P_n)\subseteq \ov \fP^{ac}_{sem}$ on the enlarged space $\ov \Omega$ defined by \eqref{eq:def:P-to-enlarge-P}. Assume that both sequences $(\P_n)$ and $(\ov\P_n)$ converge to some law $\P_0 \in \fM_1(\Omega)$ and $\ov \P_0 \in \fM_1(\ov\Omega)$, respectively. Then, we have
	\begin{equation*}
	\P_0= \ov \P_0 \circ \ov X^{-1}.
	\end{equation*}
\end{lemma}
\begin{proof}
	To see this, observe first that by definition \eqref{eq:def:P-to-enlarge-P}, we have for any $n$ that
	\begin{equation*}
	\ov \P_n \circ \ov X^{-1}= \P_n \circ (\ov X \circ\Psi^{\P})^{-1}=\P_n.
	\end{equation*}
	Therefore, for any bounded continuous function $g \in C_b(\Omega)$, we obtain that
	\begin{equation*}
	\int_{\Omega} g \, d\P_{n} = \int_{\overline{\Omega}} g \circ \ov X \, d\overline{\P}_{n}, 
	\end{equation*}
	for every $n$. Consequently, as $\overline{\P}_{n}$ converges to $\overline{\P}_0$ weakly and as $\ov X: \overline{\Omega} \rightarrow \Omega$ is  continuous, 
	\begin{align*}
	\lim_{n \rightarrow \infty} \int_{\Omega}g \, d\P_{n} = \lim_{n \rightarrow \infty}\int_{\overline{\Omega}} g \circ \ov X \, d\overline{\P}_{n} = \int_{\overline{\Omega}} g \circ \ov X \, d\overline{\P}_{0} = \int_{\Omega}g \, d(\overline{\P}_0 \circ \ov X^{-1}).
	\end{align*}
	As $g \in C_b(\Omega)$ was arbitrary, we conclude that $\overline{\P}_0 \circ \ov X^{-1}$ is the weak limit of 
	$(\P_{n})_{n \in \N}$.
\end{proof}
Due to Lemma~\ref{le:identify-limit-P-0}, we can identify the structure of limit laws of sequences in $\fP^{ac}_{sem}(\Theta)$.
\begin{proposition}\label{prop:limit-P-P-enlarge}
	Let $\Theta\subseteq \R^d \times \S^d_+ \times \cL$ satisfy Condition (B) and let  $(\P_n)_{n\in \N} \subseteq \fP^{ac}_{sem}(\Theta)$ converging to some law $\P_0 \in \fM_1(\Omega)$.  
	Then the following holds true:
	\begin{enumerate}
		\item [1)]  We obtain that $\P_0 \in \fP^{ac}_{sem}$.
	\end{enumerate}
	If in addition, $\Theta$ is closed, convex and satisfies Condition~(J). Then the following holds true:
	\begin{enumerate}
		\item[2)] We obtain that $\P_0 \in \fP^{ac}_{sem}(\Theta)$. In particular, $\fP^{ac}_{sem}(\Theta)$ is closed.
	\end{enumerate}
\end{proposition}
%
%

%
%
\begin{proof} 
	%
	By Prohorov's theorem, $(\P_n)_{n \in \N}$ is tight. Since $X_0:\Omega \to \R^d$ is continuous,  $(\P_n \circ X_0^{-1})_{n \in \N}\subseteq \fM_1(\R^d)$ is tight, too. Consider the corresponding sequence $(\ov \P_n)_{n \in \N} \in \ov \fP_{sem}^{ac}(\Theta)$ on the enlarged space $\ov \Omega$ defined in \eqref{eq:def:P-to-enlarge-P}. By definition, $(\ov \P_n \circ \ov X_0^{-1})_{n \in \N}\subseteq \fM_1(\R^d)$ is tight. Therefore, by Proposition~\ref{prop:tightness-enlarge}, we have tightness of $(\ov \P_n)_{n \in \N}$, hence there exists a  subsequence $(\overline{\P}_{n_k})_{k \in \N}$ which converges to some $\overline{\P}_0 \in \fM_1(\ov\Omega)$.
	
	To prove 1), we can apply Proposition~\ref{prop:X-semim-cont-char}  to conclude that $\ov \P_0 \in \ov \fP^{ac,w}_{sem}$ satisfying both
	\begin{align*}
	& \bullet \ \E^{\ov \P_0}\Big[\int_0^T \big[|\ov b^{\ov \P_0}_s|  + |\ov c^{\ov \P_0}_s| +\int_{\R^d} |h(x)|^2\,\ov F^{\ov \P_0}_s(dx)\big]\, ds\Big]<\infty,\\
	&\bullet \ \mbox{$\ov M^c$, $\ov M^d$ are $\ov \P_0$-$\ov \F$-martingales.}
	\end{align*}
	
	By Lemma~\ref{le:push-forward}, we obtain that 
	$\ov \P_0\circ \ov X^{-1} \in \fP^{ac}_{sem}$. Moreover, we know from Lemma~\ref{le:identify-limit-P-0} that $\P_0= \ov \P_0 \circ \ov X^{-1}$, hence  $\P_0\in \fP^{ac}_{sem}$.
	
	To show 2), assume from now on that $\Theta$ is closed, convex and satisfies Condition~(J). Then by Corollary~\ref{co:closedness-P-Theta-enlarge}, $\ov \P_0 \in \overline{\fP}_{sem}^{ac}(\Theta)$. Therefore, we obtain the desired results directly from Lemma~\ref{le:push-forward} and Lemma~\ref{le:identify-limit-P-0}.


	%
\end{proof}
%
%
Next, we present the proof of Theorem~\ref{thm:purely-disc} and Corollary~\ref{co:closedness-M-d-law}, characterizing when a limit law of purely discontinuous martingales remains a purely discontinuous martingale law.
\begin{proof}[Proof of Theorem~\ref{thm:purely-disc}]
	By Proposition~\ref{prop:limit-P-P-enlarge}, $\P_0 \in \fP^{ac}_{sem}$, hence it remains to prove that $\P_0$ is a martingale law and
	the characterization 
	to be a purely discontinuous martingale law.
	
	\textbf{Step 1:} We show that  the canonical process $X$ on $\Omega$ is a $\P_0$-$\F$-martingale.
	
	To see this, 
	consider the corresponding sequence $(\ov \P_n)_{n \in \N} \in \ov \fP^{ac}_{m,d}(\Theta)$ on the enlarged space $\ov \Omega$ defined by \eqref{eq:def:P-to-enlarge-P}. By the same argument as in the beginning of the proof of Proposition~\ref{prop:limit-P-P-enlarge}, we obtain the existence
	of a  subsequence $(\ov \P_{n_k})_{k \in \N}$ which converges weakly to some law $\ov \P_0 \in \fM_1(\ov \Omega)$. We deduce from Proposition~\ref{prop:X-semim-cont-char} that the first coordinate $\ov X$ is a $\ov \P_0$-$\ov \F$-semimartingale having absolutely continuous characteristics and canonical representation 
	\begin{equation*}
	\ov X = \ov X_0 + \ov B + \ov M^c + \ov M^d + \ov J \quad \ \ov \P_0\mbox{-a.s.}
	\end{equation*}
	Moreover, by assumption, $\E^{\ov \P_0}[|\ov X_0|]<\infty$.
	
	
	Now, we claim that $\ov X$ is a $\ov \P_0$-$\ov \F$-martingale. Since we know by Proposition~\ref{prop:X-semim-cont-char} that $\ov X$ is a $\ov \P_0$-$\ov\F$-semimartingale with the above canonical representation, where both $\ov M^c$ and $\ov M^d$ are $\ov \P_0$-$\ov\F$-martingales, it remains to show that 
	\begin{equation*}
	\mbox{$\ov B + \ov J$ is a $\ov \P_0$-$\ov \F$-martingale.}
	\end{equation*}
	From the Skorokhod representation, we have that both $(z^{n,\ov B})$ and $(z^{n,\ov J})$ converges pointwise to $z^{0,\ov B}$ and $z^{0,\ov J}$, respectively. As each $z^{n,\ov B}$ and $z^{0,\ov B}$ are continuous, we deduce from \cite[Proposition~VI.2.2 a), p.338]{JacodShiryaev.03} that $z^{n,\ov B}+z^{n,\ov J}$ converges to $z^{0,\ov B}+z^{0,\ov J}$. 
	Moreover, we know from the proof of Lemma~\ref{le:M-enlarge} that for each $t\in [0,T]$, both sequences $(\ov M^c_t\,|\, \ov \P_n)$ and $(\ov M^d_t\,|\, \ov \P_n)$ are uniformly integrable, and the same holds true by \cite[Lemma~5.2]{NeufeldNutz.13b} for the sequence $(\ov X_t-\ov X_0\,|\, \ov \P_n)$. Therefore, due to the canonical representation of $\ov X$ under each $\ov \P_n$, we can also conclude that the sequence $(\ov B_t +\ov J_t\,|\, \ov \P_n)$ is uniformly integrable for each $t\in [0,T]$.
	As a consequence, we can apply the same proof as in Lemma~\ref{le:M-enlarge}(i), but with respect to $\ov B + \ov J$ instead of $\ov M^d$ to derive that $\ov B + \ov J$ is a $\ov \P_0$-$\ov \F$-martingale. 
	
	Consider the smaller filtration $\ov \F^{\ov X}\subseteq \ov \F$ generated by $\ov X$. As $\ov X$ is a $\ov \P_0$-$\ov \F$ martingale being $\ov \F^{\ov X}$-adapted, it is  also a $\ov \P_0$-$\ov \F^{\ov X}$ martingale by the tower property of the conditional expectation. Therefore, by construction, we obtain that under the push forward measure $\ov \P_0 \circ \ov X^{-1}$, the canonical process $X$ on the original space $\Omega$ is a $\ov \P_0 \circ \ov X^{-1}$-$\F$-martingale. Moreover, applying Lemma~\ref{le:identify-limit-P-0} yields
	\begin{equation}\label{eq:proof-thm-purely-disc-limit-meas}
	\P_0 = \ov \P_0  \circ \ov X^{-1},
	\end{equation}
	hence we get the desired result of Step 1.
	
	\textbf{Step 2:} We want to characterize when $X$ is a $\P_0$-$\F$-purely discontinuous martingale, i.e. 
	\begin{equation*}
	\P_0 \in \fP^{ac}_{m,d} \Longleftrightarrow \ \lim\limits_{\delta\downarrow0}\limsup_{n \to \infty} \E^{\P_n}\Big[\int_0^T\int_{\{|x|\leq \delta\}} |x|^2\, F^{\P_n}_t(dx)\,dt\Big] =0.
	\end{equation*}
	To see this, observe that from the first step, 
	$X$ is a $\P_0$-$\F$-martingale and by \eqref{eq:proof-thm-purely-disc-limit-meas}
	\begin{align}\label{eq:proof-purely-disc-char-0}
	& \ \mbox{ $X$ is a $\P_0$-$\F$-purely discont. martingale}\\
	\Longleftrightarrow & \ \mbox{ $\ov X$ is a $\ov \P_0$-$\ov \F^{\ov X}$--purely discont. martingale.} \nonumber
	\end{align} 
	We obtained in Step 1 that $\ov X$ is a $\ov \P_0$-$ \ov \F$- and $\ov \P_0$-$\ov \F^{\ov X}$-martingale. 
	Moreover, the quadratic variation $[\ov X]$ is $\ov \F^{\ov X}$-adapted, hence so is its continuous part, which implies that the second characteristic of $\ov X$ under $\ov \P_0$-$ \ov \F$ and $\ov \P_0$-$\ov \F^{\ov X}$ are the same. From the fact that a martingale is a purely discontinuous one if and only if its second characteristic vanishes, we obtain that
	\begin{align}\label{eq:proof-purely-disc-char-1}
	& \ \mbox{ $\ov X$ is a $\ov \P_0$-$\ov \F^{\ov X}$-purely discont. martingale}\\
	\Longleftrightarrow  & \ \mbox{ $\ov X$ is a $\ov\P_0$-$\ov \F$-purely discont. martingale.} \nonumber
	\end{align}
	Now, as $\ov X$ is a purely discontinuous $\ov \P_n$-$\ov \F$-martingale, we have that $\ov M^c= 0 \ \ov \P_n$-a.s. for all $n \in \N$. Using the Skorokhod representation, this means that $z^{n,\ov M^c}=0 \ \lambda$-a.s. for each $n$. As $z^{n,\ov M^c}$ converges pointwise to $z^{0,\ov M^c}$, we also have that $z^{0,\ov M^c}=0 \ \lambda$-a.s., which means that $\ov M^c=0 \,\ov \P_0$-a.s. Therefore, as from Step 1 we know that $\ov X = \ov X_0+\ov M^c + \ov M^d + (\ov B + \ov J) \ \ov \P_0$-a.s., with $(\ov B + \ov J)$ being a  $\ov \P_0$-$\ov \F$-purely discontinuous martingale, we conclude that
	\begin{align}\label{eq:proof-purely-disc-char-2}
	& \  \mbox{ $\ov X$ is a $\ov \P_0$-$\ov \F$-purely discont. martingale}\\
	\Longleftrightarrow & \ \mbox{ $\ov M^d$ is a $\ov\P_0$-$\ov \F$-purely discont. martingale.}
	\nonumber
	\end{align}
	In Section~\ref{sec:enlarge}, we proved in Proposition~\ref{prop:purely-disc-enl} that
	\begin{align}\label{eq:proof-purely-disc-char-3}
	& \ \ov M^d \mbox{is a $\ov \P_0$-$\ov \F$-purely discontinuous martingale} \nonumber\\
	\Longleftrightarrow & \  \lim\limits_{\delta\downarrow0}\limsup_{n \to \infty} \E^{\ov \P_n}\Big[\int_0^T\int_{\{|x|\leq \delta\}} |x|^2\, \ov F^{\ov \P_n}_t(dx)\,dt\Big] =0.
	\end{align}
	Finally, by definition of the corresponding measures $(\ov \P_n)_{n \in \N}$ on the enlarged space, we deduce from Remark~\ref{rem:from-P-to-P-enlarge} that
	\begin{align}\label{eq:proof-purely-disc-char-4}
	& \ \lim\limits_{\delta\downarrow0}\limsup_{n \to \infty} \E^{\ov \P_n}\Big[\int_0^T\int_{\{|x|\leq \delta\}} |x|^2\, \ov F^{\ov \P_n}_t(dx)\,dt\Big] =0 \nonumber\\ \ \Longleftrightarrow & \  \lim\limits_{\delta\downarrow0}\limsup_{n \to \infty} \E^{\P_n}\Big[\int_0^T\int_{\{|x|\leq \delta\}} |x|^2\,  F^{\P_n}_t(dx)\,dt\Big] =0.
	\end{align}
	The desired characterization now follows from the equivalence of the statement \eqref{eq:proof-purely-disc-char-0}--\eqref{eq:proof-purely-disc-char-4}.
\end{proof}
%
\begin{proof}[Proof of Corollary~\ref{co:closedness-M-d-law}]
	For each $\P_n$, the purely discontinuous martingale part $M^{d,\P_n}$ has $\P_n$-$\F$-differential characteristics $(0,0,F^{\F,\P_n,M})$ with $F^{\F,\P_n,M}=F^{\P_n}\circ h^{-1}$, where $F^{\P_n}$ denotes  the third $\P_n$-$\F$-differential characteristic of the canonical process $X$. In particular, $F^{\F,\P_n,M}=F^{\P_n}$ on $\{|x|\leq \delta\}$ for any small enough $\delta>0$. Denote by $\F^{M}\subseteq \F^X$ the filtration generated by $M^{d,\P_n}$. We deduce from Lemma~\ref{le:semimart-char-smaller-filtr} and Remark~\ref{rem:semimart-char-smaller-filtr}  that $M^{d,\P_n}$ is also a $\P_n$-$\F^{M}$-purely discontinuous martingale with corresponding differential characteristics $(0,0,F^{\F^{M},\P_n,M})$ satisfying for any small enough $\delta>0$ 
	\begin{align*}
	& \ \E^{\P_n}\Big[\int_0^T \int_{\{|x|\leq \delta\}} |x|^2\, F^{\F^{M},\P_n,M}_t(dx)\,dt\Big]\\
	= & \ \E^{\P_n}\Big[\int_0^T \int_{\{|x|\leq \delta\}} |x|^2\, F^{\F,\P_n,M}_t(dx)\,dt\Big]\\
	= & \ \E^{\P_n}\Big[\int_0^T \int_{\{|x|\leq \delta\}} |x|^2\, F^{\P_n}_t(dx)\,dt\Big].
	\end{align*}
	Therefore, the result follows directly from Theorem~\ref{thm:purely-disc}.
\end{proof}
We continue with the proof of Theorem~\ref{thm:compactness}, which provides a necessary and sufficient condition for $\fP^{ac}_{sem}(\Theta)(\Gamma_0)$ to be compact.
\begin{proof}[Proof of Theorem~\ref{thm:compactness}]
	We already have proved in Proposition~\ref{prop:limit-P-P-enlarge} that Condition~(J) implies closedness of the set $\fP^{ac}_{sem}(\Theta)$.
	
	For the other direction, i.e. to see that closedness of $\fP^{ac}_{sem}(\Theta)$ implies that Condition~(J) holds whenever $\Theta$ additionally satisfies the condition in Definition~\ref{def:box}, we assume  that Condition~(J) fails and want to conclude that then  also closedness of $\fP^{ac}_{sem}(\Theta)$ fails. 
	
	By definition,  if Condition~(J) fails, then there exists $\varepsilon>0$ and a sequence $(F^m)_{m \in \N}\subseteq \proj_F(\Theta)$ such that for each $m \in \N$,
	\begin{equation*}
	\int_{\{|x|\leq \frac{1}{m}\}} |x|^2 \, F^m(dx)\geq \varepsilon.
	\end{equation*}
	Denote by $\widehat{c}:=\max\{|c| \in \proj_c(\Theta)\}$. Due to $\Theta$ satisfying the condition in Definition~\ref{def:box}, there exists for each $m$ an element $b^m \in \R^d$ such that $(b^m, \widehat c, F^{m}) \in \Theta$. Denote by $\P^m$ the law where the canonical process $X$ is a  L\'evy process with corresponding L\'evy triplet $(b^m, \widehat c, F^{m})$. By definition, each $\P^m \in \fP^{ac}_{sem}(\Theta)$ with $X_0=0 \ \P^m$-a.s.. We derive from Corollary~\ref{co:tightness-P} that the sequence $(\P^m)$ is tight, hence there exists a subsequence $(\P^{m_k})_{k \in \N}$ which by Proposition~\ref{prop:limit-P-P-enlarge} converges weakly to some law $\P_0 \in \fP^{ac}_{sem}$. Our goal is to show that $\P_0 \notin \fP^{ac}_{sem}(\Theta)$ which would mean that the closedness of $\fP^{ac}_{sem}(\Theta)$ fails.
	
	To that end, consider the corresponding sequence $(\ov \P^k) \subseteq \ov\fP^{ac}_{sem}(\Theta)$ of $(\P^{m_k})$ on the enlarged space $\ov \Omega$ defined by \eqref{eq:def:P-to-enlarge-P}. Due to Proposition~\ref{prop:tightness-enlarge}, $(\ov \P^k)$ is tight, hence there exists a subsequence $(\ov \P^{k_u})_{u \in \N}$ which by Proposition~\ref{prop:X-semim-cont-char} converges weakly to some law $\ov \P_0 \in \ov\fP^{ac,w}_{sem}$.
	
	Recall $\ov c$ being the differential of $\ov C$, which is the quadratic variation of $\ov M^c$ under $\ov \P_0$. Morevover, denote by $\ov c^{\ov \P_0,\ov M^d}$ the differential of the quadratic variation of the continuous martingale part of $\ov M^d$ under $\ov \P_0$-$\ov \F$ (note that under $\ov \P_0$-$\ov \F$, $\ov M^d$ is not necessarily a purely discontinuous martingale). Then, the differential $\ov c^{\ov \P_0}$ of the second characteristic of $\ov X$ under $\ov \P_0$-$\ov \F$ satisfies
	\begin{equation*}
	\ov c^{\ov \P_0}=\ov c + \ov c^{\ov \P_0,\ov M^d} \quad \ov \P_0\times dt\mbox{-a.s.}.
	\end{equation*}
	Observe that by construction of the  sequence $(\ov \P^{k_u})$,we have
	\begin{align*}
	& \ \lim\limits_{\delta\downarrow0}\limsup_{u \to \infty} \E^{\ov \P_{k_u}}\Big[\int_0^T\int_{\{|x|\leq \delta\}} |x|^2\, \ov F^{\ov \P_{k_u}}_t(dx)\,dt\Big] \\
	= & \ 
	\lim\limits_{\delta\downarrow0}\limsup_{u \to \infty} \int_0^T\int_{\{|x|\leq \delta\}} |x|^2\, \ov F^{m_{k_{u}}}(dx)\,dt\\
	\geq & \
	\varepsilon.
	\end{align*} 
	Therefore, we deduce from Proposition~\ref{prop:purely-disc-enl} that $\ov M^d$ is not a $\ov\P_0$-$\ov \F$-purely discontinuous martingale, which implies that $(\ov \P_0\times dt)[\ov c^{\ov \P_0,\ov M^d}>0]>0$.
	
	Now, we claim that $\ov c = \hat{c} \ \, \ov\P_0\times dt$-a.s. To see this, we know from the Skorokhod representation that the sequence $(z^{k_u,\ov C})_{u \in \N}$ converges uniformly to $z^{0,\ov C}$. But by definition of the sequence $(\ov \P^{k_u})$, we have $\lambda$-a.s. that $z^{k_u,\ov C}_t = \widehat c t, \ t\in [0,T],$ for all $u$. Therefore, we also have $\lambda$-a.s. that $z^{0,\ov C}_t = \widehat c t, \ t\in [0,T],$ which implies that indeed $\ov c = \hat{c} \ \ov\P_0\times dt$-a.s.
	
	As the second characteristic of $\ov X$  under $\ov\P$-$\ov \F^{\ov X}$ coincide with the one under $\ov\P$-$\ov \F$, we obtain that the second differential characteristic $\ov c^{\ov X,\ov \P_0}$ under $\ov\P$-$\ov \F^{\ov X}$ satisfies
	\begin{equation*}
	\ov c^{\ov X,\ov \P_0}= \ov c^{\ov \P_0}=\widehat c + \ov c^{\ov \P_0,\ov M^d} \ \ov \P_0\times dt\mbox{-a.s.}
	\end{equation*}
	Recall that $\widehat{c}:=\max\{|c| \in \proj_c(\Theta)\}$. Due to the positive definiteness, we conclude that 
	\begin{equation*}
	(\ov \P_0\times dt)\big[\ov c^{\ov X,\ov \P_0} \notin \proj_c(\Theta)\big]\geq (\ov \P_0\times dt)\big[\ov c^{\ov \P_0,\ov M^d}>0\big]>0.
	\end{equation*}
	As $\P_0= \ov \P_0 \circ \ov X^{-1}$, we can argue as in Lemma~\ref{le:push-forward} to see that $\P_0 \in \fP^{ac}_{sem}$, but $\P_0 \notin \fP^{ac}_{sem}(\Theta)$, since
	the second differential characteristics $c^{\P_0}$  satisfies
	\begin{equation*}
	(\P_0\times dt)\big[ c^{\P_0} \notin \proj_c(\Theta)\big]>0.
	\end{equation*}
	
	Summarizing, we have shown that if Condition~(J) fails, we can find a sequence of probability measures in $\fP^{ac}_{sem}(\Theta)$ converging to some element $\P_0$ which is not an element in $\fP^{ac}_{sem}(\Theta)$. Hence closedness of $\fP^{ac}_{sem}(\Theta)$ fails as desired. 
	
	Finally, for the compactness criterion, observe that the compactness assumption on $\Gamma_0 \subseteq \fM_1(\R^d)$ implies tightness of $\fP^{ac}_{sem}(\Theta)(\Gamma_0)$, see Corollary \ref{co:tightness-P}. Moreover, $\Gamma_0$ being closed implies that any sequence $(\P_n)\subseteq \fP^{ac}_{sem} (\Theta)(\Gamma_0)$ converging to some $\P_0 \in \fM_1(\Omega)$ satisfies $\P_0\circ \ov X_0^{-1}\in \Gamma_0$. Hence, the compactness characterization follows directly from the closedness one proved above. 
\end{proof}
%
%
Next, we provide the proof of Theorem~\ref{thm:compactness2} which is our second compactness criterion for the set $\fP_{sem}^{ac}(\Theta)(\Gamma_0)$ of semimartingale laws.
\begin{proof}[Proof of Theorem~\ref{thm:compactness2}]
	$1)\Longrightarrow 2)\colon$ Let $\Gamma_0 \subseteq \fM_1(\R^d)$ be a compact subset of distributions on $\R^d$. As by assumption, $\Theta$ satisfies Condition~(B), we obtain from Corollary~\ref{co:tightness-P} that $\fP_{sem}^{ac}(\Theta)(\Gamma_0)$ is tight. Hence, it remains to show that $\fP_{sem}^{ac}(\Theta)(\Gamma_0)$ is closed. To that end, let $(\P_n)_{n \in \N} \subseteq \fP_{sem}^{ac}(\Theta)(\Gamma_0)$  converge to some $\P_0 \in \fM_1(\Omega)$. We need to show that $\P_0 \in \fP^{ac}_{sem}(\Theta)(\Gamma_0)$. Consider the corresponding sequence $(\ov\P_n)_{n \in \N}\subseteq \ov \fP^{ac}_{sem}(\Theta)(\Gamma_0)$ on the enlarged space $\ov \Omega$ defined by \eqref{eq:def:P-to-enlarge-P}. 
	By Proposition~\ref{prop:tightness-enlarge}, $\ov \fP_{sem}^{ac}(\Theta)(\Gamma_0)$ is tight, hence there exists a subsequence $(\ov\P_{n_k})_{k \in \N}$ which converges weakly to some law $\ov \P_0 \in \fM_1(\ov \Omega)$. Observe that closedness of $\Gamma_0$ assures that $\ov \P_0 \circ \ov X_0^{-1} \in \Gamma_0$, whereas Proposition~\ref{prop:X-semim-cont-char} implies that $\ov \P_0 \in \ov\fP^{ac,w}_{sem}$. In particular, $\overline{\widetilde{C}}$ is the modified second characteristic of $\ov X$ under each $\ov \P_n$-$\ov \F$ and also under $\ov \P_0$-$\ov \F$. 
	Following the arguments of the proof of Corollary~\ref{co:closedness-P-Theta-enlarge}, we see that
	\begin{equation*}
	\E^{\ov \P_{n_k}}\Big[\int_0^T \, \mathbbm{1}_{\big\{m\big(\ov B_{t+\frac{1}{m}} - \ov B_t,\, \overline{\widetilde{C}}_{t+\frac{1}{m}} - \overline{\widetilde{C}}_t,\, (\ov V^i_{t + \frac{1}{m}} - \ov V^i_t)_{i \in \N}\big) \in \mbox{cl}(\ov \varphi(u(\Theta)))\big\}}\,dt\Big]=T.
	\end{equation*} 
	In terms of the Skorokhod representation, it means that
	\begin{equation*}
	\E^\lambda\Big[\int_0^T \mathbbm{1}_{\big\{m\big(z^{{n_k},\ov B}_{t+\frac{1}{m}} - z^{{n_k},\ov B}_t,\, z^{{n_k},\overline{\widetilde{C}}}_{t+\frac{1}{m}} - z^{{n_k},\overline{\widetilde{C}}}_t,\, (z^{{n_k},\ov V^i}_{t + \frac{1}{m}} - z^{{n_k},\ov V^i}_t)_{i \in \N}\big) \in \mbox{cl}(\ov \varphi(u(\Theta)))\big\}}\, dt\Big] = T.
	\end{equation*}
	Observe that $(z^{{n_k},\overline{\widetilde{C}}})$ converges to  $(z^{0,\overline{\widetilde{C}}})$ in $\Omega^c_{d^2} \ \lambda$-a.s. Therefore, when letting $k$ tend to infinity, 
	we get that
	\begin{equation*}
	\E^\lambda\Big[\int_0^T \mathbbm{1}_{\big\{m\big(z^{0,\ov B}_{t+\frac{1}{m}} - z^{0,\ov B}_t,\, z^{0,\overline{\widetilde{C}}}_{t+\frac{1}{m}} - z^{0,\overline{\widetilde{C}}}_t,\, (z^{0,\ov V^i}_{t + \frac{1}{m}} - z^{0,\ov V^i}_t)_{i \in \N}\big) \in \mbox{cl}(\ov \varphi(u(\Theta)))\big\}}\, dt\Big] = T.
	\end{equation*}
	This demonstrates that $m(\ov B_{t+\frac{1}{m}} - \ov B_t, \overline{\widetilde{C}}_{t+\frac{1}{m}} - \overline{\widetilde{C}}_t, (\ov V^i_{t + \frac{1}{m}} - \ov V^i_t)_{i \in \N}) \in \mbox{cl}(\ov \varphi(u(\Theta)))$, $\overline{\P}_0 \times dt$-a.e. for all $m \in \N$. Now letting $m$ tend to infinity, we obtain that $(\ov b_t, \overline{\widetilde{c}}_t, \ov v_t) \in \mbox{cl}(\ov \varphi(u(\Theta)))$ holds $\overline{\P}_0 \times dt$-a.e., where $\overline{\widetilde{c}}_t:=\limsup_{n\to \infty} n( \overline{\widetilde{C}}_t-\overline{\widetilde{C}}_{(t-\frac{1}{n})\vee 0}), \,  t\in[0,T]$. Moreover, due to the assumption 1) that $u(\Theta)$ is  closed, we can argue as in \eqref{eq:char-in-Theta} to obtain the relation
	\begin{equation*}
	(\ov b^{\ov \P_0}, \overline{\widetilde{c}}^{\ov \P_0}, \ov F^{\ov \P_0}) \in u(\Theta) \ \ \ov \P_0\times dt\mbox{-a.s.} \ \ \Longleftrightarrow \ \ (\ov b, \overline{\widetilde{c}}, \ov v) \in \mbox{cl}(\ov \varphi(u(\Theta)))  \ \ \ov \P_0\times dt\mbox{-a.s.,}
	\end{equation*}
	This and the definition of the function $u$ ensure that $(\ov b^{\ov \P_0}, \overline{c}^{\ov \P_0}, \ov F^{\ov \P_0}) \in \Theta \ \ov \P_0\times dt\mbox{-a.s.}$. Therefore, Lemma~\ref{le:push-forward} and Lemma~\ref{le:identify-limit-P-0} imply that $\P_0 =\ov \P_0 \circ \ov X^{-1} \in \fP^{ac}_{sem}(\Theta)(\Gamma_0)$.
	
	$2) \Longrightarrow 3)\colon$ Since $\fP_L(\Theta)\subseteq \fP_{sem}^{ac}(\Theta)(\{\delta_0\})$, tightness of $\fP_L(\Theta)$ follows from the assumption $2)$ that $\fP_{sem}^{ac}(\Theta)(\{\delta_0\})$ is compact. Therefore, it remains to show that $\fP_L(\Theta)$ is closed. To that end, let $(\P_n)_{n \in \N} \subseteq \fP_{sem}^{ac}(\Theta)(\Gamma_0)$  converge to some $\P_0 \in \fM_1(\Omega)$. Due to Assumption 2), $\fP_{sem}^{ac}(\Theta)(\{\delta_0\})$ is compact, hence $\P_0 \in \fP_{sem}^{ac}(\Theta)(\{\delta_0\})$. Thus, it remains to show that $\P_0 \in \fP_L$. Now, since each $\P_n \in \fP_L$ is a L\'evy law, \cite[VII.2.6, p.395]{JacodShiryaev.03}  assures that  indeed $\P_0 \in \fP_L$.
	
	$3) \Longrightarrow 4)\colon$
	Let $((b^n,c^n,F^n))_{n \in \in \N} \subseteq \Theta$. We need to show that there exists a subsequence $((b^{n_k},c^{n_k},F^{n_k}))_{k\in  \N}$ and  $(b^0,c^0,F^0) \in \Theta$ such that for all $f \in C^2_b(\R^d)$ the subsequence of functions $((\fL^{(b^{n_k},c^{n_k},F^{n_k})}f))_{k \in \N}$ converges pointwise to the function $(\fL^{(b^0,c^0,F^0)}f)$. 
	To that end, fix any $f \in C^2_b(\R^d)$. 
	For each $n \in \N$  let $\P_n \in \fP_{L}(\Theta)$ be the L\'evy law with corresponding L\'evy triplet  $(b^n,c^n,F^n)$. Due to assumption 3), $\fP_{L}(\Theta)$ is compact, hence  there exists a subsequence $(\P_{n_k})_{k \in \N}$ which converges to some L\'evy law $\P_0$ with  L\'evy triplet $(b^0,c^0,F^0)\in \Theta$. We claim that the corresponding subsequence of functions $((\fL^{(b^{n_k},c^{n_k},F^{n_k})}f))_{k \in \N}$ converges pointwise to the corresponding function $(\fL^{(b^0,c^0,F^0)}f)$. To see this, denote by $\widetilde{c}^n$ the modified second L\'evy triplet for each $n \in \N_0$. Then, observe that for each $n \in \N_0$, we have for all $x \in \R^d$ that 
	\begin{align*}
	& \ (\fL^{(b^n,c^n,F^n)}f)(x)\\
	= & \ \textstyle\sum\limits_{i = 1}^d b^{n,i} \frac{\partial f}{\partial x^i}(x) + \frac{1}{2}\textstyle\sum\limits_{i,j=1}^d c^{n,ij}\frac{\partial^2 f}{\partial x^i \partial x^j}(x) \\
	& \ + \int_{\R^d} \Big(f(x + y) -f(x) - \textstyle\sum\limits_{i =1}^d\frac{\partial f}{\partial x^i}(x)h^i(y)\Big) \, F^n(dy) \\
	=& \  \textstyle\sum\limits_{i = 1}^d b^{n,i} \frac{\partial f}{\partial x^i}(x) + \frac{1}{2}\textstyle\sum\limits_{i,j=1}^d \Big(c^{n,ij} + \int_{\R^d}h^i(y)h^j(y) \, F^n(dy)\Big)\frac{\partial^2 f}{\partial x^i \partial x^j}(x) \\
	& \ + \int_{\R^d} \Big(f(x + y) -f(x) - \textstyle\sum\limits_{i =1}^d\frac{\partial f}{\partial x^i}(x)h^i(y) - \frac{1}{2}\textstyle\sum\limits_{i,j=1}^d\frac{\partial^2 f}{\partial x^i \partial x^j}(x) h^i(y)h^j(y)\Big) \, F^n(dy) \\
	=& \  \textstyle\sum\limits_{i = 1}^d b^{n,i} \frac{\partial f}{\partial x^i}(x) + \frac{1}{2}\textstyle\sum\limits_{i,j=1}^d \widetilde{c}^{n,ij}\frac{\partial^2 f}{\partial x^i \partial x^j}(x) +\int_{\R^d} R_x(y) \, F^n(dy), 
	\end{align*}
	where $R_x(y) = f(x + y) -f(x) - \sum_{i =1}^d\frac{\partial f}{\partial x^i}(x)h^i(y) - \frac{1}{2}\sum_{i,j=1}^d\frac{\partial^2 f}{\partial x^i \partial x^j}(x) h^i(y)h^j(y)$. Since $h$ is a bounded continuous function on $\R^d$ with $h(y)=y$ in a neighborhood of the origin and $f \in C^2_b(\R^d)$, the basic property of Taylor expansion guarantees that $R_x$ is a bounded continuous function on $\R^d$ satisfying $R_x(y) = o(|y|^2)$ as $y \rightarrow 0$. 
	Thus, since the subsequence of L\'evy laws $(\P_{n_k})_{k \in \N}$ converges weakly to the L\'evy law $\P_0$, \cite[Theorem~VII.2.9, p.396]{JacodShiryaev.03} guarantees that the sequence $((b^{n_k}, \widetilde c^{n_k},F^{n_k}))_{k \in \N}$ converges to $(b^{0}, \widetilde{c}^{0},F^{0})$ and the sequence $(\int_{\R^d}
	R_x(y)\, F^{n_k}(dy))_{k \in \N}$ converges to $\int_{\R^d}
	R_x(y)\, F^{0}(dy)$. This, together with the above representation for  $(\fL^{(b^n,c^n,F^n)}f)(x)$, $n \in \N_0$, demonstrates that the subsequence $(\fL^{(b^{n_k},c^{n_k},F^{n_k})}f)_{k \in \N}$ indeed converges pointwise to $(\fL^{(b^{0},c^{0},F^{0})}f)$. 
	
	$4) \Longrightarrow 5)\colon$ Fix any $x \in \R^d$ and define the function $f_x\colon \R^d \to \mathbb C$ by  $f_x(z):=e^{ix\cdot z}$. Observe that for each L\'evy triplet $(b,c,F)$, we have
	\begin{equation*}
	\psi^{(b,c,F)}(x)= \mathrm{i}x\cdot b - \frac{1}{2}x \cdot c \cdot x + \int_{\R^d}(\mathrm{e}^{\mathrm{i}x \cdot y} - 1 - \mathrm{i}x\cdot h(y) )\, F(dy)= (L^{(b,c,F)}f_x)(0).
	\end{equation*}
	Therefore, we see that 5) follows from assumption 4) applied to both the real and imaginary part of $f_x$.
	
	$5) \Longrightarrow 1)\colon$ 
	Let $((b^n,\widetilde{c}^n,F^n))_{n \in \N}\subseteq u(\Theta)$  be a sequence which converges to some $(\mathfrak{b},\widetilde{\mathfrak{c}},\mathfrak{F})\in \R^d \times \S^d_+ \times \cL$. we need to show that there exists $(b^0, c^0, F^0)\in \Theta$ such that $u(b^0,c^0,F^0)=(\mathfrak{b},\widetilde{\mathfrak{c}},\mathfrak{F})$. To that end, for each $n\in \N$, let $(b^n,c^n,F^n)\in \Theta$ such that $u(b^n,c^n,F^n)= (b^n,\widetilde{c}^n,F^n)$. Moreover, for each $n \in \N$, let $\mu_n$ be the infinitely divisible distribution with L\'evy triplet $(b^n,c^n,F^n)$ and characteristic function $\varphi_{\mu_n}(x):=e^{\psi^{(b^n,c^n,F^n)}(x)}$. Consider the sequence of functions $(\psi^{(b^n,c^n,F^n)})_{n\in \N}$. Due to assumption 5), $\{\psi^{(b,c,F)}: (b,c,F) \in \Theta\}$ is sequentially compact  for the topology of pointwise convergence. Therefore, there exists a subsequence of functions $(\psi^{(b^{n_k},c^{n_k},F^{n_k})})_{k \in \N}$ and a L\'evy triplet $(b^0,c^0,F^0) \in \Theta$ such that the subsequence of functions $(\psi^{(b^{n_k},c^{n_k},F^{n_k})})_{k \in \N}$ converges pointwise to the function $\psi^{(b^0,c^0,F^0)}$. This implies that the subsequence of characteristic functions $(e^{\psi^{(b^{n_k},c^{n_k},F^{n_k})}})_{k \in \N}$ converges pointwise to the characteristic function  $e^{\psi^{(b^{0},c^{0},F^{0})}}$ of the infinitely divisible distribution $\mu_0$ with L\'evy triplet $(b^0,c^0,F^0)$. By L\'evy's continuity theorem, this implies that the subsequence $(\mu_{n_k})_{k \in \N}$ converges weakly to $\mu_0$. Therefore, \cite[Theorem~VII.2.9, p.396]{JacodShiryaev.03} and the definition of the function $u$ in \eqref{eq:Def-func-u} assure that $(u(b^{n_k},c^{n_k},F^{n_k}))_{k \in \N}$ also converges to $u(b^{0},c^{0},F^{0})$. Since by assumption $(u(b^{n_k},c^{n_k},F^{n_k}))_{k \in \N}$ also converges to $(\mathfrak{b},\widetilde{\mathfrak{c}},\mathfrak{F})$, we obtain that indeed, $u(b^0,c^0,F^0)=(\mathfrak{b},\widetilde{\mathfrak{c}},\mathfrak{F})$.
\end{proof}
We finish this section with the proof of Example~\ref{ex:cond-J} showing that in general, Condition~(J) is not necessary for $\fP^{ac}_{sem}(\Theta)$ to be closed.
\begin{proof}[Proof of Example~\ref{ex:cond-J}]
	Let $d=1$, and consider the set $\Theta\subseteq \R\times [0,\infty)\times \cL$ defined by
	\begin{equation*}
	\Theta:=\Big\{(b,c,F) \Big| \, \supp(F)\subseteq\{|x|\leq 1\},\ b=0, \ c+ \int_{\R} |x|^2\, F(dx)=1\Big\}.
	\end{equation*}
	Clearly, $\Theta$ is closed, convex and satisfies Condition~(B). We claim that $\fP^{ac}_{sem}(\Theta)(\delta_0)$ is compact. By 1) in Theorem~\ref{thm:compactness2}, it suffices to show that $u(\Theta)\subseteq \R^d \times \S^d_+ \times \cL$ is closed. To that end, observe that
	\begin{equation}
	u(\Theta)=\{0\}\times \{1\}\times \big\{F \in \cL \, \big| \, \supp(F)\subseteq\{|x|\leq 1\}\big\}.
	\end{equation}
	Therefore, it remains to show that  $\big\{F \in \cL \, \big| \, \supp(F)\subseteq\{|x|\leq 1\}\big\}\subseteq \cL$ is closed.
	Let $(F_n)_{n\in \N}\subseteq  \big\{F \in \cL \, \big| \, \supp(F)\subseteq\{|x|\leq 1\}\big\}$ converge to some $F_0 \in \cL$. We need to show that $\supp(F_0)\subseteq\{|x|\leq 1\}$. Consider a sequence of functions $(f_k)\subseteq C_b(\R)$ with values in $[0,1]$ such that $f_k=0$ on $\{|x| \leq 1\}$ for each $k$ and $(f_k)$ increasingly converges pointwise to $\mathbbm{1}_{\{|x|>1\}}$. Then, by the monotone convergence theorem
	\begin{equation*}
	\int_{\R^d} \mathbbm{1}_{\{|x|>1\}} \,F_0(dx)= \lim_{k \to \infty} \int_{\R^d} f_k(x) \,F_0(dx).
	\end{equation*}
	Since each $f_k$ is a bounded continuous function vanishing in a neighborhood of the origin, we obtain  for each $k \in \N$ that
	\begin{equation*}
	\int_{\R^d} f_k(x) \,F_0(dx) 
	=
	\lim_{n \to \infty} \int_{\R^d} f_k(x) \,F_n(dx).
	\end{equation*}
	Therefore, as each $F_n$ satisfies $\supp(F_n)\subseteq\{|x|\leq 1\}$ and each $f_k$ satisfies $f_k=0$ on $\{|x| \leq 1\}$, we obtain that
	\begin{equation*}
	\int_{\R^d} \mathbbm{1}_{\{|x|>1\}} \,F_0(dx)=\lim_{k \to \infty}  \lim_{n \to \infty} \int_{\R^d} f_k(x) \,F_n(dx)=0.
	\end{equation*}
	This implies that $\supp(F_0)\subseteq\{|x|\leq 1\}$.
\end{proof}
\section{Semimartingale Optimal Transport}\label{sec:optimal-transport} 
The goal of this section is to prove Theorem~\ref{thm:main}, which provides the existence of a minimizer of the primal optimal transport problem introduced in \eqref{eq:primal}, and  also present a corresponding duality result. We follow Tan and Touzi \cite{TanTouzi.11}, and Mikami and Thieullen \cite{MikamiThieullen.06}.

Recall the optimal transport problem \eqref{eq:primal} defined by 
\begin{equation*}
V(\mu_0,\mu_1):= \inf_{\P \in \fP_\Theta(\mu_0,\mu_1)} \fJ(\P):= \inf_{\P \in \fP_\Theta(\mu_0,\mu_1)}\E^\P \Big[\int_0^1 L(t,X,b^\P_t,c^\P_t,F^\P_t)\,dt\Big],
\end{equation*}
where we write $\fP_\Theta\equiv\fP^{ac}_{sem}(\Theta)$ to shorten the notation.
\begin{remark}\label{rem:compact-mu-0-mu-1}
	Under the assumption on $\Theta$ stated in Theorem~\ref{thm:main}, we have as a consequence of Theorem~\ref{thm:compactness} that $\fP_\Theta(\mu_0,\mu_1)$ is compact. Indeed, tightness is clear as $\fP_\Theta(\mu_0,\mu_1)\subseteq \fP_\Theta$. Closedness follows from the observation that for any sequence $(\P_n)\subseteq \fP_{\Theta}$ converging to some $\P_0 \in \fP_\Theta$, we have the convergence of $(\P_n \circ X_0^{-1})_{n \in \N}$ to $\P_0 \circ X_0^{-1}$, and as $\P_0$ has no fixed time of discontinuity, we also have  the convergence of $(\P_n \circ X_1^{-1})_{n \in \N}$ to $\P_0 \circ X_1^{-1}$.
\end{remark} 

We start with a useful lemma which allows us to consider the optimal transport problem introduced in \eqref{eq:primal}, but  on the enlarged space introduced in Section~\ref{sec:enlarge}, and give its relation to the original one. To that end, recall the function $\ov \varphi:\R^d \times \S^d_+ \times \cL \to \R^d \times \S^d_+ \times \R^\N$ introduced in \eqref{eq:def-ov-varphi} (with corresponding function $\varphi$). We have seen that $\ov \varphi$ is an additive, positive homogeneous map being a  bijection onto its image, see Section~\ref{subsec:enlarg}. 

Define the corresponding function $\ov L: [0,1] \times \Omega \times \ov \varphi(\Theta) \to [0,\infty)$ by setting
\begin{equation*}
\ov L(t,\omega,\ov \varphi(b,c,F)):= L(t,\omega, b,c,F),
\end{equation*}
and define for any $\ov \P \in \fP_{\Theta}(\mu_0,\mu_1)$ the associated transportation cost
\begin{equation*}
\ov \fJ(\ov \P):=\E^{\ov \P} \Big[\int_0^1 \ov L\big(t,\ov X,\ov \varphi (b^{\ov \P}_t,c^{\ov \P}_t,F^{\ov \P}_t)\big)\,dt\Big]= \E^{\ov \P} \Big[\int_0^1 \ov L\big(t,\ov X,\ov b_t, \ov c_t, \ov v_t\big)\,dt\Big],
\end{equation*}
where $\ov b, \ov c, \ov v$ are defined as in Section~\ref{subsec:enlarg}. 
\begin{lemma}\label{le:from-P-to-P-enlarge-and-viceversa-J} Under the conditions of Theorem~\ref{thm:main}, the following hold true:
	
	\vspace*{0.15cm}
	\noindent
	(i) For any $\P \in \fP_{\Theta}(\mu_0,\mu_1)$, let $\ov \P \in \ov \fP_{\Theta}(\mu_0,\mu_1)$ be the corresponding  measure on the enlarged space $\ov \Omega$ defined in \eqref{eq:def:P-to-enlarge-P}. Then, we have $\fJ(\P)=\ov \fJ(\ov\P)$.
	
	\vspace*{0.1cm}
	\noindent
	(ii) Conversely, for any $\ov \P \in \ov \fP_{\Theta}(\mu_0,\mu_1)$, let $\P:=\ov \P \circ \ov X^{-1} \in \fP_{\Theta}(\mu_0,\mu_1)$ be the push forward measure defined in \eqref{eq:def:from-P-enlarge-to-P}. Then, we have $\ov \fJ(\ov \P)\geq \fJ(\P)$.
\end{lemma}
\begin{proof}
	To see part (i), let $\P \in \fP_{\Theta}(\mu_0,\mu_1)$ and let $\ov \P:=\P\circ (\Psi^\P)^{-1} \in \ov \fP_{\Theta}(\mu_0,\mu_1)$, where $\Psi^{\P}: \Omega \to \ov \Omega$ is defined just above \eqref{eq:def:P-to-enlarge-P}. Moreover, recall the set $\cC^+(\R^d):=\{g_i\,|\,i \in \N\}$ and the processes $\ov b, \ov c$ and $\ov v:=(\ov v^1,\ov v^2,\dots)$ defined 
	in Section~\ref{subsec:enlarg}. Then,
	\begin{align*}
	\fJ(\P)
	&=\E^\P \Big[\int_0^1 L(t,X,b^\P_t,c^\P_t,F^\P_t)\,dt\Big]\\
	&=\E^{\P} \Big[\int_0^1 \ov L(t,\ov X,\ov b^{\ov \P}_t,\ov c^{\ov \P}_t,(\int_{\R^d} g_i(x)\,\ov F^{\ov\P}_t(dx))_{i \in \N})\circ \Psi^{\P}\,dt\Big]\\
	&=\E^{\P} \Big[\int_0^1 \ov L(t,\ov X,\ov b_t,\ov c_t,\ov v_t)\circ \Psi^{\P}\,dt\Big]\\
	&=\E^{\ov \P} \Big[\int_0^1 \ov L(t,\ov X,\ov b_t,\ov c_t,\ov v_t)\,dt\Big]\\
	&= \ov \fJ(\ov\P).
	\end{align*}
	
	For part (ii), let $\ov \P  \in \ov \fP_{\Theta}(\mu_0,\mu_1)$ and $\P:=\ov \P \circ \ov X^{-1}\in\fP_{\Theta}(\mu_0,\mu_1)$ be the push forward measure defined in  Lemma~\ref{le:push-forward}. Denote by $(\ov b^{\ov\P},\ov c^{\ov\P},\ov F^{\ov\P})$ the $\ov \P$-$\ov \F$-differential characteristics of $\ov X$. Due to Assumption~\ref{ass:cost-function-L} and its definition, the function $\ov L(t,\omega,\cdot,\cdot,\cdot)$ is convex on $\ov \varphi(\Theta)$ for each $t,\omega$. Denote by 
	$\ov\F^{\ov X, \ov \P}_+$  
	the usual $\ov \P$-augmentation of $\ov\F^{\ov X}$. Using Fubini's theorem, Jensen inequality and the definition of $\ov \P$ yields
	\begin{align*}
	& \ \ov \fJ(\ov \P)\\
	= & \ \int_0^1 \E^{\ov \P} \Big[ \ov L(t,\ov X,\ov b_t,\ov c_t,\ov v_t)\,\Big]\,dt\\
	\geq & \ \int_0^1 \E^{\ov \P} \Big[ \ov L\big(t,\ov X,\E^{\ov\P}\big[(\ov b_t,\ov c_t,\ov v_t)\,\big| \, \ov\cF^{\ov X, \ov \P}_{t+}\big] \big)\Big]\,dt\\
	= & \ \int_0^1 \E^{\ov \P} \Big[ \ov L\big(t,\ov X,\E^{\ov\P}\big[\, \big(\ov b^{\ov\P}_t,\ov c^{\ov\P}_t,(\int_{\R^d}g_i(x)\,\ov F^{\ov \P}_t(dx)\,)_{i \in \N} \big)\,\big| \, \ov\cF^{\ov X, \ov \P}_{t+}\big] \big)\Big]\,dt.
	\end{align*} 
	By Lemma~\ref{le:semimart-char-smaller-filtr},  the differential characteristics $(\ov b^{\ov X, \ov \P}, \ov c^{\ov X, \ov \P}, \ov F^{\ov X, \ov \P})$ of $\ov X$ under $\ov \P$-$\ov \F^{\ov X}$ (which are the same under $\ov \P$-$\ov \F^{\ov X, \ov \P}_+$) are  optional projections of the differential characteristics of $\ov X$ under $\ov \P$-$\ov \F$. Therefore, we obtain from the definition of  $\ov \varphi$ defined in \eqref{eq:def-ov-varphi} that
	\begin{align*}
	& \ \int_0^1 \E^{\ov \P} \Big[ \ov L\big(t,\ov X,\E^{\ov\P}\big[\, \big(\ov b^{\ov\P}_t,\ov c^{\ov\P}_t,(\int_{\R^d}g_i(x)\,\ov F^{\ov \P}_t(dx) )_{i \in \N}\,\big)\,\big| \, \ov\cF^{\ov X, \ov \P}_{t+}\big] \big)\Big]\,dt \\
	= & \  \int_0^1 \E^{\ov \P} \Big[ \ov L\big(t,\ov X,\ov b^{\ov X, \ov \P}_t,\ov c^{\ov X, \ov \P}_t,(\int_{\R^d}g_i(x)\,\ov F^{\ov X, \ov \P}_t(dx) )_{i \in \N}\,\big)\,\Big]\,dt \\
	= & \ \int_0^1 \E^{\ov \P} \Big[ \ov L\big(t,\ov X, \ov\varphi( \ov b^{\ov X, \ov \P}_t,\ov c^{\ov X, \ov \P}_t,\ov F^{\ov X, \ov \P}_t) \big)\,\Big]\,dt. 
	\end{align*}
	We know from the proof of Lemma~\ref{le:push-forward} that
	\begin{align*}
	& \ \mbox{ the law of  $\ov\varphi(\ov b^{\ov X, \ov \P},\ov c^{\ov X, \ov \P},\ov F^{\ov X, \ov \P})$ under $\ov \P$}\\
	= & \  \mbox{ the law of  $\ov\varphi(\ov b^{\ov X, \ov \P},\ov c^{\ov X, \ov \P},\ov F^{\ov X, \ov \P}) \circ \Phi$ under $\P$ \ }
	\end{align*}
	and that the differential characteristics of $X$ under $\P$-$\F$ satisfy
	\begin{equation*}
	(b^\P, c^\P, F^\P) = (\ov b^{\ov X, \ov \P}  \circ \Phi,\ov c^{\ov X, \ov \P}  \circ \Phi, \ov F^{\ov X, \ov \P}  \circ \Phi).
	\end{equation*}
	Therefore, we conclude that
	\begin{align*}
	\int_0^1 \E^{\ov \P} \Big[ \ov L\big(t,\ov X, \ov\varphi( \ov b^{\ov X, \ov \P}_t,\ov c^{\ov X, \ov \P}_t,\ov F^{\ov X, \ov \P}_t)\,\big)\,\Big]\,dt
	&= \int_0^1 \E^{\P} \Big[  L(t,X, b^{\P}_t, c^{\P}_t, F^{\P}_t)\,\Big]\,dt\\
	& = \fJ(\P).
	\end{align*}
\end{proof}
Now we are able to prove the existence of a minimizer $\widehat{\P} \in \fP_{\Theta}(\mu_0,\mu_1)$ for
\eqref{eq:primal}.
\begin{lemma}\label{le:lsc-value-fct}
	The function
	\begin{equation*}
	\fM_1(\R^d)\times \fM_1(\R^d) \to [0,\infty], \quad \quad (\mu_0,\mu_1)\mapsto V(\mu_0,\mu_1)
	\end{equation*}
	is lower semicontinuous. As a consequence, there exists $\widehat{\P} \in \fP_{\Theta}(\mu_0,\mu_1)$ satisfying
	\begin{equation*}
	\fJ(\widehat \P)= \inf_{\P \in \fP_\Theta(\mu_0,\mu_1)} \fJ(\P).
	\end{equation*} 
\end{lemma}
\begin{proof}
	We follow  \cite[Lemma~3.13]{TanTouzi.11}, which goes back to the arguments in \cite[Lemma~3.1]{MikamiThieullen.06}.
	
	Let $(\mu^n_0)_{n \in \N}$ and $(\mu^n_1)_{n \in \N}$ be two sequences in $\fM_1(\R^d)$ converging weakly to $\mu_0$ and $\mu_1$, respectively. We need to show that
	\begin{equation*}
	\liminf_{n \rightarrow \infty} V(\mu^n_0,\mu^n_1) \geq V(\mu_0,\mu_1).
	\end{equation*}
	Without loss of generality, assume that $\liminf_{n \rightarrow \infty} V(\mu^n_0,\mu^n_1) < \infty$. Then, after choosing a subsequence if necessary, we can assume that the sequence $(V(\mu^n_0,\mu^n_1))_{n \in \N}$ is 
	bounded, and for each $n \in \N$, there exists a probability measure $\P_n \in \fP_\Theta(\mu^n_0,\mu^n_1)$ such that
	\begin{equation}\label{eq:minimizer-procedure}
	0 \leq \fJ(\P_n) - V(\mu^n_0,\mu^n_1) \leq \frac{1}{n}.
	\end{equation}
	Thanks to Lemma~\ref{le:from-P-to-P-enlarge-and-viceversa-J}, we find for each $n \in \N$  a corresponding measure $\overline{\P}_n \in \ov \fP_\Theta(\mu^n_0,\mu^n_1)$ on the enlarged space satisfying $\overline{\fJ}(\overline{\P}_n) = \fJ(\P_n)$.
	\

	Note that the sequences $(\mu^n_0)_{n \in \N}$ and $(\mu^n_1)_{n \in \N}$ are tight as they converge weakly to $\mu_0$ and $\mu_1$, respectively. Consequently, Proposition~\ref{prop:tightness-enlarge} implies tightness of the sequence $(\overline{\P}_n)_{n \in \N}$. Moreover, we deduce from Corollary~\ref{co:closedness-P-Theta-enlarge} and the arguments in Remark~\ref{rem:compact-mu-0-mu-1} that any limit law $\overline{\P}_0$  of a converging subsequence $(\overline{\P}_{n_k})_{k\in \N}\subseteq(\overline{\P}_n)_{n \in \N}$ is an element of $\ov \fP_\Theta(\mu_0,\mu_1)$.
	\
	
	Now, as the cost function satisfy Assumption~\ref{ass:cost-function-L}, we can follow exactly the arguments of \cite[Lemma~3.9]{TanTouzi.11}, which go  back to \cite{Mikami.02}, to derive the lower semicontinuity of the map
	\begin{equation*}
	\ov\fP_\Theta \to [0,\infty], \quad \ov \P \mapsto \ov \fJ(\ov  \P).
	\end{equation*}
	Moreover,  by Lemma~\ref{le:from-P-to-P-enlarge-and-viceversa-J}, we can find a probability measure $\P_0 \in \fP_\Theta(\mu_0,\mu_1)$ such that $\fJ(\P_0) \leq \overline{\fJ}(\overline{\P}_0)$. Hence, we  obtained the lower semicontinuity due to the following inequalities
	\begin{align*}
	\liminf_{n \rightarrow \infty} V(\mu^n_0,\mu^n_1) = \liminf_{n \rightarrow \infty} \fJ(\P_n) = \liminf_{n \rightarrow \infty} \overline{\fJ}(\overline{\P}_n) \geq \overline{\fJ}(\overline{\P}_0) 
	&\geq \fJ(\P_0)\\
	&\geq V(\mu_0,\mu_1).
	\end{align*}
	
	To obtain the existence of a minimizer $\widehat{\P} \in \fP_{\Theta}(\mu_0,\mu_1)$, choose $(\mu^n_0,\mu^n_1)=(\mu_0,\mu_1)$ and follow the arguments used above from equation \eqref{eq:minimizer-procedure} on to derive the result.
\end{proof}
In the rest of this section, we prove the duality result for the value function $V(\mu_0,\mu_1)$ stated in Theorem~\ref{thm:main}. We will use classical convex duality arguments, which require that $V(\mu_0,\mu_1)$ is lower semicontinuous and convex. Whereas the lower semicontinuity of $V(\mu_0,\mu_1)$ was already shown in Lemma~\ref{le:lsc-value-fct}, we can argue exactly the same way as in \cite[Lemma~3.15]{TanTouzi.11} to obtain the convexity of the map
\begin{equation*} 
\fM_1(\R^d)\times \fM_1(\R^d) \to [0,\infty], \quad (\mu_0,\mu_1) \mapsto V(\mu_0,\mu_1).
\end{equation*} 
Before starting the proof of the duality result, recall the dual function  
\begin{equation*} 
\mathcal{V}(\mu_0,\mu_1) := \sup_{\lambda_1 \in C_b(\R^d)}\Big\{ \int_{\R^d} \lambda^{\lambda_1}_0(x)\,\mu_0(dx) - \int_{\R^d} \lambda_1(x)\,\mu_1(dx)\Big\},
\end{equation*}
where
\begin{equation*} 
\lambda_0^{\lambda_1}(x) := \inf_{\P \in \fP_\Theta(\delta_{x})} \E^\P\Big[\int_0^1 L(t,X,b^\P_t,c^\P_t,F^\P_t)\,dt + \lambda_1(X_1)\Big].
\end{equation*}
By arguing exactly as in the proof of \cite[Lemma~3.5]{TanTouzi.11} (using \cite[Theorem~2.1]{NeufeldNutz.13b} for the conditioning and pasting of probability measures), we get immediately the following result.
\begin{lemma}\label{le:meas.selec}
	Let the cost function $L$ satisfy Assumption~\ref{ass:cost-function-L}. Then for any $\lambda_1 \in C_b(\R^d)$, the function $\lambda^{\lambda_1}_0$  is measurable with respect to the Borel $\sigma$-field on $\R^d$ completed by $\mu_0$, and
	\begin{equation*}
	\int_{\R^d} \lambda^{\lambda_1}_0(x)\,\mu_0(dx) = \inf_{\P \in \fP_\Theta(\mu_0)} \E^\P\Big[\int_0^1 L(t,X,b^\P_t,c^\P_t,F^\P_t)\,dt + \lambda_1(X_1)\Big].
	\end{equation*}
	In particular, the integral $\int_{\R^d} \lambda^{\lambda_1}_0(x)\,\mu_0(dx)$ is well-defined.
\end{lemma}
To keep the notation short, denote $\mu(\phi) := \int_{\R^d} \phi(x)\, \mu(dx)$ for all $\mu \in \fM_1(\R^d)$, $\phi \in L^1(\mu).$
%
%
\begin{proof}[Proof of Theorem~\ref{thm:main}] In Lemma~\ref{le:lsc-value-fct}, we already have proved the existence of a minimizer $\widehat \P \in \fP_{\Theta}(\mu_0,\mu_1)$ of the primal optimal transport problem \eqref{eq:primal}, whenever $V(\mu_0,\mu_1)<\infty$.
	
	To obtain the duality result, we follow the argument of \cite[p.9]{TanTouzi.11} and \cite[Theorem~2.1]{MikamiThieullen.06}. For any fixed initial distribution $\mu_0$, observe that if $V(\mu_0,\mu_1)=\infty$ for every $\mu_1 \in \fM_1(\R^d)$, then $\fJ(\P)=\infty$ for all $\P \in \fP_{\Theta}(\mu_0)$, which by definition of the dual function and Lemma~\ref{le:meas.selec} implies that also $\cV(\mu_0,\mu_1)=\infty$. In this case, the duality result holds true trivially. 
	
	Now, consider the case where the function $\mu_1 \mapsto V(\mu_0,\mu_1)$ is not always equal to infinity. Denote by $\fM_{f,s}(\R^d)$ the space of all finite signed measures on $\R^d$ equipped with the coarsest topology making the maps $\mu \mapsto \mu(\phi)$ continuous for every $\phi \in C_b(\R^d)$. Then, 
	the subspace topology on $\fM_1(\R^d)$ coincides with the usual weak topology on it, see \cite[Chapter~8]{Bogachev.07volII}. By extending $V(\mu_0,\cdot)$ from $\fM_1(\R^d)$ to $\fM_{f,s}(\R^d)$ setting $V(\mu_0,\mu_1) = \infty$ for all $\mu_1 \in \fM_{f,s}(\R^d)\setminus \fM_1(\R^d)$, we retain its lower semicontinuity and convexity also on the bigger space $\fM_{f,s}(\R^d)$.
	
	Now, recall that the dual space $\fM_{f,s}(\R^d)^*$ of $\fM_{f,s}(\R^d)$ is defined by 
	\begin{equation*}
	\fM_{f,s}(\R^d)^* = \big\{\mu \mapsto  \int_{\R^d} \phi(x) \, \mu(dx) \mid \phi \in C_b(\R^d)\big\},
	\end{equation*}
	see e.g. \cite[Lemma 3.2.3, p.65]{DeuschelStroock.89}, and the  Legendre transform
	$g:C_b(\R^d) \to (-\infty,\infty]$ of $V(\mu_0,\cdot)$, which is  defined by
	\begin{equation*}
	g(\lambda_1) := \sup_{\mu_1 \in \fM_{f,s}(\R^d)} \big\{\mu_1(\lambda_1) - V(\mu_0,\mu_1)\big\}.
	\end{equation*}
	We can apply the classical convex duality result \cite[Theorem 2.2.15, p.55]{DeuschelStroock.89} to get
	\begin{equation}\label{eq:convex duality}
	V(\mu_0,\mu_1)= \sup_{\lambda_1 \in C_b(\R^d)} \big\{\mu_1(\lambda_1) - g(\lambda_1) \big\}.
	\end{equation}
	Moreover, using the definition of $V(\mu_0,\mu_1)$ and Lemma~\ref{le:meas.selec} leads to the following characterization of the Legendre transform
	\begin{align*}
	g(-\lambda_1) &:= \sup_{\mu_1 \in \fM_{f,s}(\R^d)} \big\{\mu_1(-\lambda_1) - V(\mu_0,\mu_1)\big\} \\
	& =\sup_{\mu_1 \in \fM_1(\R^d)} \big\{\mu_1(-\lambda_1) - V(\mu_0,\mu_1)\big\} \\
	& = - \inf_{\mu_1 \in \fM_1(\R^d)}\inf_{\P \in \fP_{\Theta}(\mu_0,\mu_1)} \big\{\E^\P[\lambda_1(X_1)] + \fJ(\P) \big\} \\
	& = - \inf_{\P \in \fP_{\Theta}(\mu_0)} \big\{\E^\P[\lambda_1(X_1)] + \fJ(\P)\big\} \\
	& = - \mu_0(\lambda^{\lambda_1}_0).
	\end{align*}
	Therefore, the classical convex duality result  \eqref{eq:convex duality} becomes as desired
	\begin{align*}
	V(\mu_0,\mu_1) = \sup_{-\lambda_1 \in C_b(\R^d)} \{\mu_1(-\lambda_1) + \mu_0(\lambda^{\lambda_1}_0)\} 
	&=\sup_{\lambda_1 \in C_b(\R^d)} \{\mu_0(\lambda^{\lambda_1}_0) - \mu_1(\lambda_1) \} \\
	&=\mathcal{V}(\mu_0,\mu_1).
	\end{align*}
\end{proof}
%
%

\appendix
\section{Appendix}\label{sec:appendix}
\setcounter{theorem}{0}
\setcounter{equation}{0}
For the appendix, let $(\Omega, \cF,\F,\P)$ be any filtered probability space. Consider a subfiltration $\G\subseteq \F$ 
and denote by $\cP(\G)$ the corresponding $\G$-predictable $\sigma$-field. For any  process $Y$, denote by $^oY$ the optional projection of $Y$ with respect to $\G^\P_+$, i.e. the usual augmentation of $\G$. The reason why we consider the usual augmentation is that the optional projection with respect to $\G$ might not exist, if $\G$ does not satisfy the usual conditions. We recall that by \cite[Proposition~2.2]{NeufeldNutz.13a}, an $\F$-adapted process $X$ having c\`adl\`ag paths is a $\F$-semimartingale if and only if it is a $\F^\P_+$-semimartingale, and the characteristics associated with these filtrations are the same. If in addition $X$ is $\G$-adapted, the same holds true  with respect to $\G$ (but of course, the characteristics may vary between $\F$ and $\G$). 

We present the following useful lemma which identifies the characteristics of a semimartingale $X$ when considering a smaller filtration. The lemma is not stated in full generality, but in such a way that it fits the framework needed for this paper.
\begin{lemma}\label{le:semimart-char-smaller-filtr}
	Let $(\Omega, \cF,\F,\P)$ be a filtered probability space, let $X$ be a stochastic process with c\`adl\`ag paths which is a $\P$-$\F$-semimartingale having absolutely continuous characteristics $(b^{\F}_t \,dt,c^{\F}_t \,d t,F^{\F}_t \,dt)$, and let $\G\subseteq\F$ be a subfiltration such that $X$ is $\G$-adapted. Moreover, assume that the  canonical representation of $X$ under $\F$,
	\begin{equation*}
	X= X_0 + \int_0^\cdot b^{\F}_s \,ds + M^{\F} + \sum_{0 \leq s \leq \cdot} [\Delta X_s -  h(\Delta X_s)], 
	\end{equation*}
	satisfies $\E^\P\Big[\int_0^T |b^\F_s|\, ds\Big]<\infty$ and the local martingale part $M^{\F}$ is  a true $\P$-$\F$-martingale.
	%
	Then the following hold:
	
	\vspace*{0.15cm}
	\noindent
	1) $X$ is a $\P$-$\G$-semimartingale having absolutely continuous characteristics with differential characteristics of the form
	\begin{equation*}
	(b^{\G},c^{\G},F^{\G}):=(^ob^{\F},c^{\F}, {^o}F^{\F}),
	\end{equation*}
	where
	${^o}F^{\G}$ is defined by setting for any $\cF\otimes \cB([0,T])\otimes \cB(\R^d)$-measurable function $W$
	\begin{equation*}
	\int_{\R^d} W(t,x)\,{^o}F^{\F}_t(dx):= {^o}\Big(\int_{\R^d} W(\cdot,x)\,F_{\cdot}^{\F}(dx)\Big)_t, \quad \ t\in[0,T].
	\end{equation*}
	Moreover, the local martingale part $M^{\G}$ in the canonical representation of $X$ under $\P$-$\G$ is a true $\P$-$\G$-martingale.
	
	\vspace*{0.15cm}
	\noindent
	2) Furthermore, if $\P\times dt$-a.s., $(b^{\F},c^{\F},F^{\F})$ are taking values in some  $\Theta\subseteq \R^d \times \S^d_+ \times \cL$ which is closed, convex
	and satisfies Condition~(B), 
	then also $(b^{\G},c^{\G},F^{\G}) \in \Theta \ \P\times dt$-a.s.
\end{lemma}
\begin{proof}
	W.l.o.g.,  assume that $(\Omega,\cF)=(\D([0,T],\R^d), \cB(\D([0,T],\R^d)))$ and let $X$ be the canonical process. 
	
	Let $(B^\F, C^\F, \nu^\F)$ be the $\F$-charactertistics of $X$.  The second characteristic can be defined as the continuous part of the finite variation process $[X]$, see e.g. \cite[Proposition~6.6]{NeufeldNutz.13a}. As by assumption, $X$ is $\G$-adapted, so is the quadratic variation process $[X]$ and hence also its continuous part. Therefore, $C^{\G}=C^{\F}$, in particular $C^{\G}$ is absolutely continuous $\P$-a.s.,  and $C^\F$ can be chosen to be $\G$-predictable. 
	%
	Consider the canonical representation
	\begin{equation*}
	X=X_0+\int_0^\cdot b^{\F^\P_+}_s\,ds + M^{\F^\P_+} +\sum_{0\leq s \leq \cdot}[\Delta X_s -h(\Delta X_s)] 
	\end{equation*}
	of the $\P$-$\F^\P_+$-semimartingale $X$, where $M^{\F^\P_+}$ is a true $\P$-$\F^{\P}_+$-martingale. Therefore, by the tower property, its optional projection ${^o}M^{\F^\P_+}$ (with respect to $\G^\P_+$) is a  $\P$-$\G^{\P}_+$-martingale. Moreover, by the integrability condition imposed on $B^\F$, it is straightforward to verify that the process $Z:= {^o}B^{\F^\P_+}-\int_0^\cdot {^o}b^{\F^\P_+}_s\,ds$ is a $\P$-$\G_+^\P$-martingale, too. By assumption, $X$ is $\G$-adapted, so $X$ is a $\G$- (and hence also $\G_+^\P$)-semimartingale with $\G_+^\P$-canonical representation
	\begin{equation*}
	X =X_0 + \int_0^\cdot {^o}b^{\F^\P_+}_s\,ds + \big({^o}M^{\F^\P_+} + Z\big) + \sum_{0\leq s \leq \cdot}[\Delta X_s -h(\Delta X_s)]\quad \P\mbox{-a.s.}.
	\end{equation*}
	This implies that $B^{\G^\P_+}=\int_0^\cdot {^o}b^{\F^\P_+}_s\,ds$, hence
	by \cite[Proposition~2.2]{NeufeldNutz.13a}, we obtain $b^\G= {^o}b^\F$. Moreover, $M^\G= \big({^o}M^{\F^\P_+} + Z\big) \, \P$-a.s., so it is a $\P$-$\G$-martingale.
	
	For the third characteristic, we have by definition that $^oF^\F(dx)\,dt$ is a predictable random measure with respect to $\G^\P_+$. Moreover, 
	we obtain for any nonnegative $\cP(\G^\P_+)\otimes \cB(\R^d)$-measurable function $W$ by Fubini's theorem that
	\begin{equation*}
	\E^\P\Big[\int_0^T \int_{\R^d}W(t,x)^oF^\F(dx)\,dt\Big]=\E^\P\Big[\int_0^T \int_{\R^d}W(t,x) \mu^X(dx,dt)\Big],
	\end{equation*} 
	which implies that the $^oF^\F(dx)\,dt$ is the $\P$-$\G^\P_+$-compensator of 
	$\mu^X(dx,dt)$,
	see \cite[Theorem~II.1.8, p.66]{JacodShiryaev.03}. Therefore, we conclude from \cite[Proposition~2.2]{NeufeldNutz.13a} that $F^\G(dx)={^o}F^\F(dx)$.
	
	Finally, for the second part, assume from now on that  $(b^{\F},c^{\F},F^{\F})$ are taking values in some $\Theta\subseteq \R^d \times \S^d_+ \times \cL$ which satisfies Condition~(B) and is closed, convex. Recall the function $\ov \varphi$ defined in \eqref{eq:def-ov-varphi}. By the characterization \eqref{eq:char-in-Theta}, we know that 
	\begin{align*}
	& \ \ (b^{\G},c^{\G},F^{\G}) \in \Theta \ \ \P\times dt\mbox{-a.s.}\\
	\Longleftrightarrow  & \ \ \big(b^{\G},c^{\G},(\int_{\R^d} g_i(x)\,F^{\G}(dx))_{i \in \N}\big) \in \mbox{cl}(\ov\varphi(\Theta))\ \ \P\times dt\mbox{-a.s.}
	\end{align*}
	where $\cC^+(\R^d)=\{g_i\,| \, i \in \N\}$, see Section~\ref{subsec:enlarg}. 
	Now by assumption, we know that
	\begin{equation*}
	\big(b^{\F},c^{\F},(\int_{\R^d} g_i(x)\,F^{\F}(dx))_{i \in \N}\big) \in \mbox{cl}(\ov\varphi(\Theta)) \ \ \P\times dt\mbox{-a.s.}
	\end{equation*}
	Moreover, due to the first part, we have
	\begin{align*}
	&\ (b^{\G},c^{\G},(\int_{\R^d} g_i(x)\,F^{\G}(dx))_{i \in \N}\big)\\
	=& \
	\big(\E^\P[b^{\F}\,|\,\G^\P_{+}],\E^\P[c^{\F}\,|\,\G^\P_{+}],(\E^\P[\int_{\R^d} g_i(x)\,F^{\F}(dx)\,|\,\G^\P_{+}])_{i \in \N}\big)\\
	=& \ \E^\P \big[ \big(b^{\F},c^{\F},(\int_{\R^d} g_i(x)\,F^{\F}(dx))_{i \in \N}\big)\,\big|\,\G^\P_{+}\big].
	\end{align*}
	Thus, as $\mbox{cl}(\ov\varphi(\Theta))$ is convex and closed, 
	\begin{equation*}
	(b^{\G},c^{\G},(\int_{\R^d} g_i(x)\,F^{\G}(dx))_{i \in \N}\big) \in \mbox{cl}(\ov\varphi(\Theta)) \ \ \ \P\times dt\mbox{-a.s..}
	\end{equation*}
\end{proof}
%
%
\begin{remark}\label{rem:semimart-char-smaller-filtr}
	Observe that in the setting of Lemma~\ref{le:semimart-char-smaller-filtr}, we have for any $\delta>0$ that
	\begin{equation*}
	\E^\P\Big[\int_0^T \int_{\{|x|\leq \delta\}} |x|^2\, F^{\F}_t(dx)\,dt\Big]=\E^\P\Big[\int_0^T \int_{\{|x|\leq \delta\}} |x|^2\, F^{\G}_t(dx)\,dt\Big].
	\end{equation*}
\end{remark}

%
%
%
%
%
%
%
%

\newcommand{\dummy}[1]{}



\begin{thebibliography}{10}
	
	\bibitem{BeiglbockHenryLaborderePenkner.11}
	M.~Beiglb{\"o}ck, P.~Henry-Labord{\`e}re, and F.~Penkner.
	\newblock Model-independent bounds for option prices: a mass transport
	approach.
	\newblock {\em Finance Stoch.}, 17(3):477--501, 2013.
	
	\bibitem{BeiglbockHenryLabordereTouzi.15}
	M.~Beiglb{\"o}ck, P.~Henry-Labord{\`e}re, and N.~Touzi.
	\newblock Monotone martingale transport plans and skorokhod embedding.
	\newblock {\em Stochastic Process. Appl.}, 127(9):3005--3013, 2017.
	
	\bibitem{BeiglbockJuillet.12}
	M.~Beiglb{\"o}ck and N.~Juillet.
	\newblock On a problem of optimal transport under marginal martingale
	constraints.
	\newblock {\em Ann. Probab.}, 
	44(1):42--106, 2016.
	
	\bibitem{BeiglbockNutzTouzi.16}
	M.~Beiglb{\"o}ck, M.~Nutz, and N.~Touzi.
	\newblock Complete duality for martingale optimal transport on the line.
	\newblock {\em Ann. Probab.}, 
	45(5):3038--3074, 2017.
	
	\bibitem{BertsekasShreve.78}
	D.~P. Bertsekas and S.~E. Shreve.
	\newblock {\em Stochastic Optimal Control. The Discrete-Time Case.}
	\newblock Academic Press, New York, 1978.
	
	\bibitem{Billingsley.99}
	P.~Billingsley.
	\newblock {\em Convergence of probability measures}.
	\newblock Wiley Series in Probability and Statistics: Probability and
	Statistics. John Wiley \& Sons Inc., New York, second edition, 1999.
	
	\bibitem{Bogachev.07volII}
	V.~I. Bogachev.
	\newblock {\em Measure Theory. {V}ol. {II}}.
	\newblock Springer-Verlag, Berlin, 2007.
	
	\bibitem{CampiLaachirMartini.14}
	L.~Campi, I.~Laachir, and C.~Martini.
	\newblock Change of numeraire in the two-marginals martingale transport
	problem.
	\newblock {\em Finance Stoch.}, 
	21(2):471--486, 2017.
	
	\bibitem{CruzeiroLassalle.15}
	A.~B. Cruzeiro and R.~Lassalle.
	\newblock Weak calculus of variations for functionals of laws of
	semi-martingales.
	\newblock {\em Preprint, arXiv: 1501.05134}, 2015.
	
	\bibitem{DeuschelStroock.89}
	J.-D. Deuschel and D.~W. Stroock.
	\newblock {\em Large deviations}, volume 137 of {\em Pure and Applied
		Mathematics}.
	\newblock Academic Press, Inc., Boston, MA, 1989.
	
	\bibitem{DolinskySoner.12}
	Y.~Dolinsky and H.~M. Soner.
	\newblock Martingale optimal transport and robust hedging in continuous time.
	\newblock {\em Probab. Theory Related Fields}, 160(1--2):391--427, 2014.
	
	\bibitem{DolinskySoner.15}
	Y.~Dolinsky and H.~M. Soner.
	\newblock Martingale optimal transport in the {S}korokhod space.
	\newblock {\em Stochastic Process. Appl.}, 125(10):3893--3931, 2015.
	
	\bibitem{FahimHuang.16}
	A.~Fahim and Y.-J. Huang.
	\newblock Model-independent superhedging under portfolio constraints.
	\newblock {\em Finance Stoch.}, 20(1), 2016.
	
	\bibitem{GalichonHenryLabordereTouzi.11}
	A.~Galichon, P.~Henry-Labord{\`e}re, and N.~Touzi.
	\newblock A stochastic control approach to no-arbitrage bounds given marginals,
	with an application to lookback options.
	\newblock {\em Ann. Appl. Probab.}, 24(1):312--336, 2014.
	
	\bibitem{GuoTanTouzi.15}
	G.~Guo, X.~Tan, and N.~Touzi.
	\newblock Tightness and duality of martingale transport on the skorokhod space.
	\newblock {\em Stochastic Process. Appl.}, 
	127(3):927--956, 2017.
	
	\bibitem{HenryLabordereOblojSpoidaTouzi.16}
	P.~Henry-Labord{\`e}re, J.~Ob{\l}{\'o}j, P.~Spoida, and N.~Touzi.
	\newblock The maximum maximum of a martingale with given {$n$} marginals.
	\newblock {\em Ann. Appl. Probab.}, 26(1):1--44, 2016.
	
	\bibitem{HenryLabordereTouzi.16}
	P.~Henry-Labord{\`e}re and N.~Touzi.
	\newblock An explicit martingale version of the one-dimensional {B}renier
	theorem.
	\newblock {\em Finance Stoch.}, 20(3):635--668, 2016.
	
	\bibitem{Hobson.98}
	D.~Hobson.
	\newblock Robust hedging of the lookback option.
	\newblock {\em Finance Stoch.}, 2(4):329--347, 1998.
	
	\bibitem{Hobson.11}
	D.~Hobson.
	\newblock The {S}korokhod embedding problem and model-independent bounds for
	option prices.
	\newblock In {\em Paris-{P}rinceton {L}ectures on {M}athematical {F}inance
		2010}, volume 2003 of {\em Lecture Notes in Math.}, pages 267--318. Springer,
	Berlin, 2011.
	
	\bibitem{JacodMeminMetivier.83}
	J.~Jacod, J.~M{\'e}min,  and M.~M{\'e}tivier.
	\newblock On tightness and stopping times.
	\newblock {\em Stochastic Process. Appl.}, 14(2):109--146, 1983.
	
	\bibitem{JacodShiryaev.03}
	J.~Jacod and A.~N. Shiryaev.
	\newblock {\em Limit Theorems for Stochastic Processes}.
	\newblock Springer, Berlin, 2nd edition, 2003.
	
	\bibitem{Jakubowski.97}
	A.~Jakubowski.
	\newblock The almost sure {S}korokhod representation for subsequences in
	nonmetric spaces.
	\newblock {\em Teor. Veroyatnost. i Primenen.}, 42(1):209--216, 1997.
	
	\bibitem{Jakubowski.97b}
	A.~Jakubowski.
	\newblock A non-{S}korohod topology on the {S}korohod space.
	\newblock {\em Electron. J. Probab.}, 2:1--21, 1997.
	
	\bibitem{Kantorovich.42}
	L.~Kantorovitch.
	\newblock On the translocation of masses.
	\newblock {\em C. R. (Doklady) Acad. Sci. URSS (N.S.)}, 37:199--201, 1942.
	
	\bibitem{Kellerer.84}
	H.~Kellerer.
	\newblock Duality theorems for marginal problems.
	\newblock {\em Z. Wahrsch. Verw. Gebiete}, 67(4):399--432, 1984.
	
	\bibitem{LepeltierMarchal.76}
	J.-P. Lepeltier and B.~Marchal.
	\newblock Probl\`eme des martingales et \'equations diff\'erentielles
	stochastiques associ\'ees \`a un op\'erateur int\'egro-diff\'erentiel.
	\newblock {\em Ann. Inst. H. Poincar\'e Sect. B (N.S.)}, 12(1):43--103, 1976.
	
	
	\bibitem{MeyerZheng.84}
	P.-A. Meyer and W.~A. Zheng.
	\newblock Tightness criteria for laws of semimartingales.
	\newblock {\em Ann. Inst. H. Poincar\'e Probab. Statist.}, 20(4):353--372,
	1984.
	
	\bibitem{Mikami.02}
	T.~Mikami.
	\newblock Optimal control for absolutely continuous stochastic processes and
	the mass transportation problem.
	\newblock {\em Electron. Comm. Probab.}, 7:199--213 (electronic), 2002.
	
	\bibitem{Mikami.15}
	T.~Mikami.
	\newblock Two end points marginal problem by stochastic optimal transportation.
	\newblock {\em SIAM J. Control Optim.}, 53(4):2449--2461, 2015.
	
	\bibitem{MikamiThieullen.06}
	T.~Mikami and M.~Thieullen.
	\newblock Duality theorem for the stochastic optimal control problem.
	\newblock {\em Stochastic Process. Appl.}, 116(12):1815--1835, 2006.
	
	\bibitem{NeufeldNutz.13a}
	A.~Neufeld and M.~Nutz.
	\newblock Measurability of semimartingale characteristics with respect to the
	probability law.
	\newblock {\em Stochastic Process. Appl.}, 124(11):3819--3845, 2014.
	
	\bibitem{NeufeldNutz.13b}
	A.~Neufeld and M.~Nutz.
	\newblock Nonlinear {L\'e}vy processes and their characteristics.
	\newblock {\em Trans. Amer. Math. Soc.}, 
	369(1):69--95, 2017.
	
	\bibitem{NeufeldNutz.15}
	A.~Neufeld and M.~Nutz.
	\newblock Robust utility maximization with {L\'e}vy processes.
	\newblock {\em Math. Finance}, 
	28(1):82--105, 2018.
	
	\bibitem{Obloj.04}
	J.~Ob{\l}{\'o}j.
	\newblock The {S}korokhod embedding problem and its offspring.
	\newblock {\em Probab. Surv.}, 1:321--390, 2004.
	
	\bibitem{Rebolledo.79}
	R.~Rebolledo.
	\newblock La m\'ethode des martingales appliqu\'ee \`a l'\'etude de la
	convergence en loi de processus.
	\newblock {\em Bull. Soc. Math. France M\'em.}, (62):1--125, 1979.
	
	\bibitem{Skorohod.56}
	A.~V. Skorohod.
	\newblock Limit theorems for stochastic processes.
	\newblock {\em Teor. Veroyatnost. i Primenen.}, 1:289--319, 1956.
	
	\bibitem{Stebegg.14}
	F.~Stebegg.
	\newblock Model-independent pricing of asian options via optimal martingale
	transport.
	\newblock {\em Preprint, arXiv: 1412.1429}, 2014.
	
	\bibitem{StroockVaradhan.79}
	D.~Stroock and S.~R.~S. Varadhan.
	\newblock {\em Multidimensional Diffusion Processes}.
	\newblock Springer, New York, 1979.
	
	\bibitem{TanTouzi.11}
	X.~Tan and N.~Touzi.
	\newblock Optimal transportation under controlled stochastic dynamics.
	\newblock {\em Ann. Probab}, 41(5):3201--3240, 2013.
	
	\bibitem{Villani.09}
	C.~Villani.
	\newblock {\em Optimal transport: old and new}, volume 338 of {\em Grundlehren der
		Mathematischen Wissenschaften}.
	\newblock Springer-Verlag, Berlin, 2009.
	
	\bibitem{Zheng.85}
	W.~A. Zheng.
	\newblock Tightness results for laws of diffusion processes application to
	stochastic mechanics.
	\newblock {\em Ann. Inst. H. Poincar\'e Probab. Statist.}, 21(2):103--124,
	1985.
	
\end{thebibliography}
\end{document}